\documentclass[a4paper, 10pt]{scrartcl}
\pdfoutput=1

\title{Stability of Correction Procedure via Reconstruction With Summation-by-Parts
       Operators for Burgers' Equation Using a Polynomial Chaos Approach}
\author{Philipp \"Offner, Jan Glaubitz, and Hendrik Ranocha}
\date{10th November 2018}
\titlehead{\includegraphics[width=0.5\textwidth]{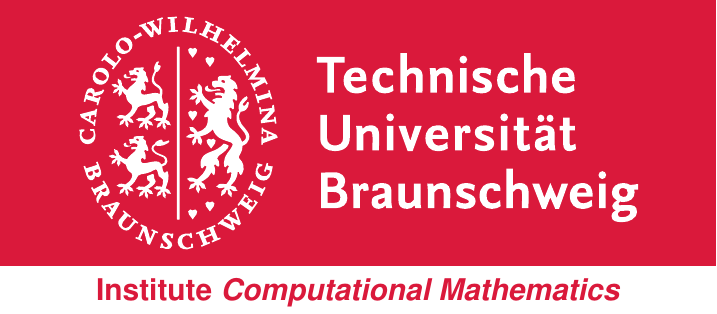}}

 \usepackage[utf8]{luainputenc}
\usepackage[UKenglish]{babel}
\usepackage{csquotes}

\usepackage[a4paper, textheight=622pt, textwidth=468pt, centering]{geometry}
\usepackage[plainpages=false,pdfpagelabels,hidelinks]{hyperref}
\makeatletter
\hypersetup{pdfauthor={\@author}}
\hypersetup{pdftitle={\@title}}
\makeatother

\usepackage{cite}

\usepackage{amsmath}
\allowdisplaybreaks
\usepackage{amssymb}
\usepackage{commath}
\usepackage{mathtools}

\usepackage{siunitx}

\usepackage{amsthm}
\theoremstyle{plain}
  \newtheorem{thm}{Theorem}[section]
  
\theoremstyle{definition}
  \newtheorem{ex}[thm]{Example}
  \newtheorem{re}[thm]{Remark}

\usepackage{color}
\usepackage{graphicx}
\usepackage[small]{caption}
\usepackage{subcaption}

\usepackage{pgfplots}
\pgfplotsset{compat=1.11}
\usepgfplotslibrary{external}
\begingroup\expandafter\expandafter\expandafter\endgroup
\expandafter\ifx\csname pdfsuppresswarningpagegroup\endcsname\relax
\else
\fi

\usepackage{booktabs}
\usepackage{rotating}

\usepackage{calc}
\usepackage{xparse}

\renewcommand{\vec}[1]{\underline{#1}}
\NewDocumentCommand{\mat}{mo}{%
  \IfValueTF{#2}{%
    \underline{\underline{#1}}{#2}
  }{%
    \underline{\underline{#1}}\,
  }%
}
\newcommand{\diag}[1]{\operatorname{diag}\left(#1\right)}

\newcommand{\scp}[2]{\left\langle{#1,\, #2}\right\rangle}

\renewcommand{\L}{\mathbf{L}}

\newcommand{\E}{\operatorname{E}}
\newcommand{\Var}{\operatorname{Var}}

\newcommand{\I}{\operatorname{I}}
\newcommand{\fnum}{f^{\mathrm{num}}}
\newcommand{\vecfnum}{\vec{f}^{\mathrm{num}}}

\newcommand{\uref}{u_{\mathrm{Ref}}}
\newcommand{\unum}{u_{\mathrm{Num}}}

\renewcommand{\epsilon}{\varepsilon}
\renewcommand{\phi}{\varphi}
\renewcommand{\rho}{\varrho}
\newcommand{\N}{\mathbb{N}}
\newcommand{\R}{\mathbb{R}}

\renewcommand{\l}{\left}

\newcommand{\Omprob}{\Omega_\mathrm{prob}}
\renewcommand{\P}{\mathcal{P}}
\newcommand{\diff}{\mathrm{d}}

\newsavebox{\DelimiterBox}
\newlength{\DelimiterHeight}
\newlength{\DelimiterDepth}
\newsavebox{\ArgumentBox}
\newlength{\ArgumentHeight}
\newlength{\ArgumentDepth}
\newlength{\ResizedDelimiterHeight}

\newcommand{\mean}[1]{\overline{#1}}

\newcommand{\jump}[1]{%
  \savebox{\ArgumentBox}{$ \displaystyle #1 $}%
  \settoheight{\ArgumentHeight}{\usebox{\ArgumentBox}}%
  \settodepth{\ArgumentDepth}{\usebox{\ArgumentBox}}%
  \savebox{\DelimiterBox}{$ [\!\![ $}%
  \settoheight{\DelimiterHeight}{\usebox{\DelimiterBox}}%
  \settodepth{\DelimiterDepth}{\usebox{\DelimiterBox}}%
  \setlength{\ResizedDelimiterHeight}{\maxof{1.2\ArgumentHeight}{\DelimiterHeight}}
  \!
  \resizebox{\width}{\ResizedDelimiterHeight}{ [\![ }
  \mkern-6.5mu
  #1
  \mkern-6.5mu
  \resizebox{\width}{\ResizedDelimiterHeight}{ ]\!] }
  \!\!
}

\begin{document}

\maketitle

\begin{abstract}
  In this paper, we consider Burgers' equation with uncertain boundary and 
initial conditions.
The polynomial chaos (PC) approach yields a 
hyperbolic system of deterministic equations, which can be solved by
several numerical methods.
Here, we apply the correction procedure via reconstruction (CPR) using 
summation-by-parts operators.
We focus especially on stability, which is proven for 
CPR methods and the systems arising from the PC approach.
Due to the usage of split-forms, the major challenge is to construct entropy 
stable numerical fluxes. 
For the first time, such numerical fluxes are constructed for all systems 
resulting from the PC approach for Burgers' equation. 
In numerical tests, we verify our results and show also the performance of the
given ansatz using CPR methods. 
Moreover, one of the simulations, i.e. Burgers' equation equipped with an 
initial shock, demonstrates quite fascinating observations. 
The behaviour of the numerical solutions from several
methods (finite volume, finite difference, CPR) differ significantly from each
other.
Through careful investigations, we conclude that the reason for this is
the high sensitivity of the system to varying dissipation. 
Furthermore, it should be stressed that the system is not strictly hyperbolic 
with genuinely nonlinear or linearly degenerate fields.
\end{abstract}

\section{Introduction}
\label{sec:introduction}

In the last decades, great efforts have been made to develop accurate and stable
numerical schemes for partial differential equations such as hyperbolic
conservation laws.
 In practical applications, real world data are used as inputs which include measurement errors. Thus, one has to deal with uncertainties in the input data and, in general,
one distinguishes between numerical errors and these uncertainties.
The errors are strictly deterministic quantities, whereas the uncertainties are stochastic quantities.
Therefore, these uncertainties are treated within a probabilistic framework. In numerical simulations,
random variables are used to model the uncertainty in boundary and initial conditions, model
parameters or
even in the geometry of the investigated domain.
Norbert Wiener developed the polynomial chaos method (PC) in \cite{wiener1938homogeneous},
where he applied Hermite polynomials to model stochastic processes with Gaussian random variables.
Ghanem and Spanos \cite{ghanem2003stochastic} introduced the polynomial chaos method
for solving partial differential equations.
We also consider this approach, which will be explained in the next section.
The main idea of polynomial chaos is that by
using a spectral method ansatz one can transform the stochastic equations back
to a strictly deterministic system of equations, which can be solved by standard
numerical methods.
The theoretical foundation of the PC method is given by the Cameron-Martin-Theorem
\cite{cameron1947orthogonal}.
One can find many works in the literature about the PC method and its
applications, see \cite{xiu2003modeling, xiu2002wiener, ghanem2003stochastic,
xiu2004supersensitivity}
and references cited therein, but it was the work of Pettersson et al.
\cite{pettersson2009numerical} that initially sparked our interest.
In \cite{pettersson2009numerical}, the authors consider Burgers' equation with uncertain initial and boundary data.
The PC approach leads to a hyperbolic system of equations, which they solve
using finite difference (FD) schemes.
In this work, the hyperbolic systems resulting from the PC method are solved
by the recent correction procedure via reconstruction (CPR)
\cite{ranocha2016summation}, which unifies the flux reconstruction
\cite{huynh2007flux} and the lifting
collocation penalty \cite{wang2009unifying} schemes in a common framework.
In \cite{ranocha2016summation, ranocha2017extended}, the authors reformulated
CPR methods using summation-by-parts (SBP) operators. We apply this approach to prove
stability for our CPR methods similarly to the FD framework in
\cite{pettersson2009numerical}. The key in our investigation is the application of
entropy stable numerical fluxes.
In this work, an ansatz is presented to
construct these entropy stable numerical fluxes in the context of PC for
Burgers' equation using CPR methods.
We start with some examples, before presenting the general setting.
The usage of split forms in the same way as in
\cite{fisher2013high, carpenter2014entropy, carpenter2013high} will be essential.
This procedure works also for other hyperbolic systems, see
for example \cite{ranocha2017shallow}.
In this article, we focus on the system formed by the PC approach for Burgers'
equation.
We demonstrate our results and quantify the behaviours of CPR methods
in numerical simulations. Furthermore, we compare our CPR methods with finite
volume (FV) methods for different test cases including
an initial rarefaction and an initial shock, always with an uncertain perturbation.
The last numerical experiment, considering a shock wave, is most remarkable.
This test case was also treated in \cite{pettersson2009numerical,
pettersson2015polynomial} and thus allows a comparison of the numerical
solutions from CPR methods as well as FV methods with the ones from FD methods
described in \cite{pettersson2009numerical}.
The behaviours of the numerical solutions from several methods (FV, FD, CPR)
\emph{differ} significantly  from each other.
Through careful investigations, we conclude that the reason for this is that the
system is \emph{not} strictly hyperbolic with genuinely nonlinear or linearly
degenerate fields and thus seems to be highly sensitive to varying dissipation.

The paper is organised as follows.
The PC approach for Burgers' equation is briefly explained in section \ref{sec:pc}.
In section 3, we repeat the main ideas of the SBP CPR method from
\cite{ranocha2016summation, ranocha2017extended} and apply this method to the
system.
We discuss stability in this context for some examples, before we investigate 
the general setting.
Here, the major key is the usage of split forms similar
to \cite{fisher2013high, carpenter2013high}. In Theorem \ref{Theorem1}, we
prove conservation (across elements) and entropy stability of the SBP CPR method.
In section \ref{sec:reference-solutions}, we derive reference solutions for
different test cases, which we will study numerically in section \ref{sec:numerical-tests}.
Here, among other things, we demonstrate that we get entirely different wave profiles in the
solutions depending highly sensitively on numerical dissipation.
Finally, we summarise our results, discuss open problems, and
give an outlook on future work.

\section{Polynomial Chaos Methods}
\label{sec:pc}

In this section, we  explain the concept of the generalised polynomial chaos
method (gPC) and its applications to hyperbolic conservation laws, in particular
to Burgers' equation.
%
 We start by introducing the notation and some preliminaries about  random fields and random inputs
before we explain the gPC approach. We follow the notation of \cite{abgrall2017uncertainty, pettersson2009numerical,xiu2002wiener,
giesselmann2017posteriori}.
For more details about the gPC, we strongly recommend these works and references cited therein.


\subsection{Random Fields and Random Inputs}

Let $(\Omega_\mathrm{prob}, \mathcal{F}, \mathcal{P}) $ be a
probability space with sample space $  \Omega_\mathrm{prob}$ and a probability measure
$\mathcal{P}$  defined on the $\sigma$-algebra $\mathcal{F}$ of subsets of
$\Omega_\mathrm{prob}$. A second measurable space $(E,\mathcal{B})$
is considered, where $E$ is a Banach space and $\mathcal{B}(E)$ the corresponding Borel $\sigma$-algebra.
An $E-$valued random field is a mapping $X:\Omprob\to E$ such that $\{\omega \in \Omprob: X(\omega)\in B \}\in
\mathcal{F}$ for
any subset $B\in \mathcal{B}$, i.e. $X$ is a measurable mapping.
For $1\leq p\leq \infty$, the Bochner space $\L^p(\Omprob;E)$ of $p$-summable random variables $X$
equipped with the norm
\begin{equation}
 ||X||_{\L^p(\Omega_\mathrm{prob}, E)}:=
					\begin{cases}
                                        \left(\int_{\Omprob}||X(\omega)||_E^p \diff \P(\omega) \right)^\frac{1}{p}, & 1\leq p<\infty, \\
                                        \operatorname{essup}_{\omega\in \Omprob} ||X(\omega)||_E, & p =\infty,
                                        \end{cases}
\end{equation}
will be considered.
With this definition, we are able to describe the random inputs.
For example, one can model uncertainties in initial data \cite{mishra2012sparse},
flux functions \cite{mishra2016numerical}, and coefficients \cite{mishra2013multi} as random fields.

In this paper, we consider uncertain initial data. Therefore, we identify the uncertain initial
data as a random field $u_0$. In particular, we consider an $\L^p(D)$-valued
random field, where $D \subset \R$. We will further assume that the initial data has the form
$u_0(x,\omega)=u(x,\xi(\omega))$ on $D\times \Omprob$.
Here, $\xi:\Omprob\to \R$ is a real valued random variable. We denote by $y=\xi(\omega)$
the image of $\omega \in \Omprob$ under $\xi$. Also, we assume that the law of the real-
valued random variable $\xi$ is absolutely continuous with respect to the Lebesgue measure.
Then, there exists a density function
$\varrho:\R\to \R^+_0$ such that $\int_{-\infty}^\infty \varrho(y)\diff y =1$
and $\P(\xi(\omega)\in A)=\int_{A}\varrho(y)\diff y$,
for any $A \in \mathcal{B}(\R)$.

Before explaining the gPC approach, we make the convention that
we will also suppress the space-time variables $(x,t)$ of $u$ in the following
sections if it is clear from the context.

\subsection{Generalised Polynomial Chaos Method}

We are interested in the following scalar conservation
law with uncertain initial data
\begin{equation}\label{Conservation_law}
 \begin{aligned}
  \partial_t u(x,t,\xi(\omega))+ \partial_x f(u(x,t,\xi(\omega)))&=0,  &(x,t,\omega) \in D\times (0,T)\times \Omprob,\\
  u(x,0,\xi(\omega))&=u_0(x,\xi(\omega)), &(x,\omega) \in D\times \Omprob.
 \end{aligned}
\end{equation}
The solution of \eqref{Conservation_law} is a random field 
$u\in \L^2( \R ,\L^2 (D\times (0,T)), \mu )$   with probability distribution $\mu=\rho(y) \dif y$
that $u(\cdot,\cdot,\xi(\omega))$
is a weak solution\footnote{In  \cite{mishra2012sparse}, random entropy
solutions are introduced. Since we are using the gPC approach only to build
a strictly deterministic system which we will investigate later,
we do not go into details about random entropy solution in this paper.}
of \eqref{Conservation_law} for $\P$-a.e. $\omega\in \Omprob$.

As described in \cite{abgrall2017uncertainty, wiener1938homogeneous},
a random field  $u\in \L^2(\R,\L^2(D\times (0,T)), \mu) $
of \eqref{Conservation_law} can be expressed by the
spectral expansion
\begin{equation}
\label{SolutionPDE}
 u(x,t,y) =\sum\limits_{i=0}^\infty u_i(x,t)\phi_i(y),
\end{equation}
where $\phi_i \in \L^2(\R, \;\mu)$  are the basis functions
and
$\set{u_i(x,t)}_{i=0}^\infty $ is a set of coefficients.

To simplify the notation, we define
the expected value by
\begin{equation}
 \E[u(x,t,\xi(\cdot))] =\int_{\omega\in \Omega_\mathrm{prob}}
u(x,t,\xi(\omega))  \, \dif \mathcal{P}(\omega)\\
    = \int_{\R} u(x,t,y)  \, \varrho(y) \dif y.
\end{equation}
Moreover, the inner product of the Hilbert space for fixed time $t$ is given by
\begin{equation}
\label{eq:inner-product-fixed-t}
 \scp{u(t)}{v(t)} := \int_\R \int_D u(x,t,y) v(x,t,y) \varrho(y)
\dif y \dif x.
\end{equation}

For the numerical approximation, we truncate the infinite series \eqref{SolutionPDE}
and consider
\begin{equation}
  u^M(x,t,y) =\sum\limits_{i=0}^M u_i(x,t)\phi_i(y).
\end{equation}
The convergence of $u^M$ to $u$ as $M\to \infty$ is guaranteed by the Cameron-Martin theorem \cite{cameron1947orthogonal}.
Classically, orthogonal polynomials are chosen
 as basis functions\footnote{Haar Wavelet and multi-wavelet expansions
are  also possible, for details see \cite{pettersson2015polynomial}.}.
The distribution of $\xi $ determines  the polynomial family.
If $\xi$ is distributed by a Gaussian measure, Hermite polynomials provide the
best convergence results, for details see \cite{xiu2002wiener}.
In this paper, we consider normalised orthogonal polynomials and, in particular, normalised Hermite polynomials.
Therefore, we speak just about the polynomial chaos (PC) method.
Basic properties of these polynomials are cited in section \ref{Appendix:Hermite}.
In our numerical tests, we compare the first and second moments
(expected value $\E[u]$ and variance $\Var(u)$) of our calculated solutions with
the analytical solution. The two moments can be expressed by coefficients
of the PC method as
\begin{equation}
\begin{aligned}
 \E[u(x,t,\xi(\cdot))]
 =&
  \int_{ \R } \sum\limits_{i=0}^\infty u_i(x,t)\phi_i(y) \varrho(y) \dif y
 \\
 =&
 u_0(x,t) \int_{\R} \phi_0(y) \varrho(y)\dif y + \int_{\R
 } \sum\limits_{i=1}^\infty u_i(x,t)\phi_i(y) \varrho(y) \dif y
 =
 u_0(x,t).
\end{aligned}
\end{equation}
In the last step, we used that $\phi_i$ are orthogonal polynomials, $\phi_0 \equiv 1$, and
$\int_{\R} \varrho(y)\dif y=1$. The variance is given by
\begin{equation}\label{eq:variance}
 \Var(u(x,t,\cdot))=\E[u^2(x,t,\cdot)]-\E^2[u(x,t,\cdot)]=
 \int_{ \R} \sum\limits_{i=0}^\infty u_i^2(x,t)\phi_i^2 (y) \varrho(y) \dif y -u_0^2
 =
 \sum\limits_{i=1}^\infty u_i^2(x,t) \E[\phi_i^2]  .
\end{equation}

\subsection{PC Method for Burgers' Equation}

We utilise  the polynomial chaos expansion for Burgers' equation
\begin{equation}\label{Burgers}
\partial_t u(x,t,y)+u(x,t,y) \partial_x u(x,t,y)=0, \quad 0\leq x\leq 1.
\end{equation}
Inserting the representation \eqref{SolutionPDE} into equation \eqref{Burgers},
\begin{equation}\label{BurgerUQ}
  \sum\limits_{i=0}^\infty \frac{\partial u_i(x,t)}{\partial t} \phi_i(y)+
  \left(  \sum\limits_{i=0}^\infty  u_i (x,t) \phi_i(y)\right)
 \cdot \left( \sum\limits_{i=0}^\infty \frac{\partial u_i(x,t)}{\partial x} \phi_i(y) \right)=0.
\end{equation}
We employ a  stochastic Galerkin approach.
It relies on a weak formulation, where the set of trial functions is the same as
the space of stochastic test functions, i.e. Hermite polynomials in this case.
We multiply \eqref{BurgerUQ} by $\phi_k$ and integrate over $\Omega_\mathrm{prob}$ with
respect to the weight function (probability density)
$\varrho$, resulting in
\begin{equation}\label{BurgerUQ2}
\int_{\R} \phi_k(y)
\sum\limits_{i=0}^\infty \frac{\partial u_i(x,t)}{\partial t} \phi_i(y) \varrho(y) \dif  y+
\int_{\R} \phi_k(y)  \left(  \sum\limits_{i=0}^\infty  u_i(x,t) \phi_i(y)\right)
 \cdot \left( \sum\limits_{i=0}^\infty \frac{\partial u_i(x,t)}{\partial x} \phi_i(y) \right) \varrho(y) \dif y=0.
\end{equation}
We get a weak approximation of \eqref{BurgerUQ2}  by choosing a finite dimensional
subspace of the polynomial chaos expansion and projecting the resulting expression
onto this subspace spanned by the basis $\set{\phi_i(\cdot)}_{i=0}^M$.
Considering the truncated PC series and using the orthogonality of $\phi_i$, we get the
symmetric system of deterministic equations
\begin{equation}\label{BurgerGalerkinUQ}
 \partial_t u_k(x,t) \E[\phi_k^2]
  +
  \sum\limits_{i=0}^M \sum\limits_{j=0}^M u_i(x,t) \partial_x u_j(x,t) \E[ \phi_i\phi_j\phi_k]
  =
  0, \quad \text{ for } k=0,1,\dots,M, \;
\end{equation}
with the triple  product
\begin{equation}
 \E[ \phi_i\phi_j\phi_k ]
 =
 \int_\R  \phi_i(y)\phi_j(y)\phi_k(y)\varrho(y) \dif y.
\end{equation}
Equation \eqref{BurgerGalerkinUQ} can  be written in matrix form as
\begin{equation}\label{BurgerGalerkinUQ2}
  B\, \partial_t u(x,t) +A(u(x,t))\partial_x u(x,t)=0
 \qquad \text{ or } \qquad
 B\, \partial_t u(x,t) + \frac{1}{2 }\partial_x \left(A(u(x,t))u(x,t)\right)=0,
\end{equation}
where the matrices $B$ and $A(u)$ are defined by
\begin{equation}\label{MatrixA}
  [B]_{jk}= \E[\phi_k^2] \ \delta_{j,k}
  \qquad \text{and} \qquad
  [A(u(x,t))]_{jk}=\sum_{i=0}^M \E[\phi_i\phi_j\phi_k] u_i(x,t).
\end{equation}
In our theoretical investigations, we use the component-wise representation
\eqref{BurgerGalerkinUQ}.
We give the matrix representation as a comparison to the works
\cite{pettersson2015polynomial, pettersson2009numerical},
where the authors analyse the system \eqref{BurgerGalerkinUQ} in the FD framework. We mention the similarities between these two approaches later.
For the scalar conservation law \eqref{Burgers}, the PC approach yields a symmetric system \eqref{BurgerGalerkinUQ} for the coefficients.
Due to symmetry, the system is hyperbolic, see for details
\cite{xiu2010numerical,chertock2015operator}. Since we apply normalised polynomials,
the matrix $B$ is the identity matrix, i.e $ \E[\phi_k^2]\equiv 1$.
For a better understanding,
we repeat the following example from \cite{pettersson2009numerical}.

\begin{ex}
\label{BeispielM=2}
  For $M=2$, we employ the  basis of normalised Hermite polynomials.
  Using \eqref{Hermitetripel}, system \eqref{BurgerGalerkinUQ} reads
  \begin{equation}
    \begin{pmatrix}
      1 & 0 &0 \\
      0 & 1 &0 \\
      0 & 0 & 1
    \end{pmatrix}
    \begin{pmatrix}
      u_0 \\ u_1 \\u_2
    \end{pmatrix}_t
  +
  \underbrace{
    \begin{pmatrix}
      u_0  & u_1 & u_2 \\
      u_1 & u_0+\sqrt{2} u_2 & \sqrt{2}u_2\\
      u_2 & \sqrt{2}u_1 & u_0+2\sqrt{2} u_2
    \end{pmatrix}}_{= A(u)}
    \begin{pmatrix}
      u_0 \\ u_1 \\u_2
    \end{pmatrix}_x= 0.
    \end{equation}
Let $f=\frac{1}{2} A(u)u$ denote the flux function for this system.
  Considering the  $\L^2$ entropy $U= \sum_{i=0}^M u_i^2$, the corresponding
  \emph{entropy flux} is given by
  \begin{equation*}
    F= u^Tf-\psi = \frac{1}{2}\left( u_0^3+3u_0u_1^2 +3u_0u_2^2 +2 \sqrt{2}u_1^2 u_2+2 \sqrt{2}u_1 u_2^2 +2\sqrt{2} u_2^3\right),
  \end{equation*}
  where
  \begin{equation}
    \psi=\frac{1}{6} u_0^3 +\frac{1}{2}u_0u_1^2+\frac{1}{2}u_0u_2^2+ \frac{\sqrt{2}}{2} u_1^2u_2+\frac{\sqrt{2}2}{3}u_2^3,
  \end{equation}
  is the flux potential. It fulfils $(\partial_u \psi)^T=f$.
  Entropy stability with respect to this entropy corresponds to
  norm stability using the inner product \eqref{eq:inner-product-fixed-t}. More
  details about entropy stability are given in section~\ref{Subsec:Stability}.
  In \cite{pettersson2015polynomial}, this example is considered on an equidistant
  mesh in the FD framework.
  Later, we analyse general $M\in \N_0$ in a semidiscrete formulation for SBP
CPR methods.
\end{ex}

\section{Stability in the Semidiscrete Setting}
\label{sec:stability}

The problem \eqref{BurgerGalerkinUQ} is hyperbolic and strictly deterministic.
Therefore, well-known numerical techniques can be applied to ensure stable and
accurate solutions.
In \cite{pettersson2009numerical}, finite difference schemes are used and the
authors also employ the summation-by-parts (SBP) property in the FD framework to
show $\L^2$ stability.
Here, we consider  correction procedure via reconstruction (CPR) methods,
also known as flux reconstruction (FR).
The CPR is a framework of high order methods for conservation laws, unifying
some discontinuous Galerkin (DG), spectral difference, and spectral volume methods
with appropriate choice of parameters.
In \cite{ranocha2016summation, ranocha2017extended}, the concept of SBP operators
is transferred to CPR methods and this property is an important tool
to show $\L^2$ stability in this setting.
The choice of an adequate numerical flux $\fnum$ is a major tool to obtain stability.
Consequently, we focus on this issue in this article.
In the following section, we present a general approach to select an
entropy stable numerical flux for SBP CPR methods in the context of generalised polynomial chaos.
We formulate the semidiscretisation of this method and prove stability. The main idea is the usage of split forms similar to
\cite{fisher2013high}.

We start with a brief description of SBP CPR methods.
For a detailed introduction to the correction procedure via reconstruction methods
and the concept of summation-by-parts
operators, we recommend the articles \cite{ranocha2016summation, ranocha2017extended,
huynh2007flux, huynh2014high, svard2014review}.

\subsection{SBP CPR Methods}

The correction procedure via reconstruction is a semidiscretisation applying a
polynomial approximation on elements. To describe the main idea, we consider a
scalar, one-dimensional hyperbolic conservation law
\begin{equation}
\label{eq:scalar-CL}
  \partial_t u + \partial_x f(u) = 0,
\end{equation}
equipped with an adequate initial condition. For simplicity, periodic boundary
conditions (or a compactly supported initial condition) will be assumed.

The domain $D \subset \R$ is split into disjoint open intervals
$D_i \subset D$ such that $\bigcup_i \overline{D_i} = D$.
We transfer each element $D_i$ onto a standard element, which is in our case simply $(-1,1)$.
All calculations are conducted within this standard element.

The solution $u$ is approximated by a polynomial of degree $p \in \N_0$. In
the basic formulation, a nodal Lagrange basis is employed. Thus, the coefficients $\vec{u}$ of
$u$ are given by the nodal values
$\vec{u}_i = u(\zeta_i), i \in \set{0, \dots, p}$, where $-1 \leq \zeta_i \leq 1$
are interpolation points in $[-1,1]$. The flux $f(u)$ is also approximated by
a polynomial, where the coefficients are given by
$\vec{f}_i = f \left( \vec{u}_i \right) = f \left( u(\zeta_i) \right)$.
The divergence of $\vec{f}$ is $\mat{D} \vec{f}$, where we apply a discrete
derivative matrix $\mat{D}$. Since the solutions will probably have discontinuities
across elements, we will have this in the discrete flux too. To avoid this problem,
we introduce a numerical flux $\vecfnum$ and also a correction term using $\mat{M}[^{-1}] \mat{R}[^T] \mat{B}$
at the boundary nodes \cite{ranocha2016summation}.
Hence, the CPR method in one element reads
\begin{equation}
\label{eq:SBP CPR}
 \begin{aligned}
\partial_t \vec{u}
  =& - \mat{D} \vec{f}
    - \mat{M}[^{-1}] \mat{R}[^T] \mat{B}\left( \vecfnum - \mat{R} \vec{f} \right)\\
  =& -\vec{\mathrm{VOL}}-\vec{\mathrm{SURF}},
 \end{aligned}
\end{equation}
where the restriction matrix $\mat{R}$ performs interpolation to the boundary,
$\vec{\mathrm{VOL}}$ is the volume term (here: $\mat{D} \vec{f}$) and
$\vec{\mathrm{SURF}}$ is the surface term.

The vector $\vecfnum = \bigl( f^{\mathrm{num}, e}_L, f^{\mathrm{num}, e}_R \bigr)$
contains the numerical fluxes of the left and right hand side of the element $e$, which gives a
common flux on the boundary using values from both neighbouring elements. Indeed,
interpolating the numerical solution in the $e$-th element to the left and right hand
side yields the values $u^{(e)}_L$ and $u^{(e)}_R$, respectively. The numerical flux
$f^{\mathrm{num}, e}_L = f^{\mathrm{num}, e-1}_R$ between the elements $e-1$ and $e$
is computed using the values $u(-) = u_R^{(e-1)}$ at the right boundary of cell
$e-1$ and $u(+) = u_L^{(e)}$ at the left boundary of cell $e$, as visualised in
\autoref{fig:num-fluxes}. For simplification, if the upper index of the element does not generate misunderstanding, it will be omitted.
\begin{figure}
  \centering
  \includegraphics{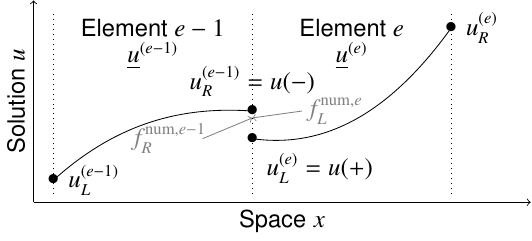}
  \caption{Notation used for numerical fluxes between elements.}
  \label{fig:num-fluxes}
\end{figure}

With respect to a chosen basis, the scalar product approximating the $\L^2$ scalar
product is represented by a matrix $\mat{M}$ and integration with respect to
the outer normal by $\mat{B} = \diag{-1,1}$. Finally, all operators are introduced and they
have to fulfil the SBP property
\begin{equation}
\label{eq:SBP}
  \mat{M} \mat{D} + \mat{D}[^T] \mat{M}
  = \mat{R}[^T] \mat{B} \mat{R},
\end{equation}
in order to mimic integration by parts on a discrete level
\begin{equation}
  \vec{u}^T \mat{M} \mat{D} \vec{v} + \vec{u}^T \mat{D}[^T] \mat{M} \vec{v}
  \approx
  \int_D u \, (\partial_x v) + \int_D (\partial_x u) \, v
  = u \, v \big|_{\partial D}
  \approx
  \vec{u}^T \mat{R}[^T] \mat{B} \mat{R} \vec{v}.
\end{equation}
Different bases can be used like  nodal Gauss-Legendre / Gauss-Lobatto-Legendre or
modal Legendre bases, as described in section \ref{sec:general}.
In the examples in section \ref{Subsec:Stability},
we choose Gauss-Lobatto-Legendre nodes, since this selection of CPR methods
is most similar to the FD setting. Thus, one recognises the similarities and differences
to the works of
Pettersson et al. \cite{pettersson2009numerical, pettersson2015polynomial}.

Again, in this paper we focus on numerical fluxes and present for the first time
(to the best of our knowledge) an approach to construct suitable (entropy
conservative and entropy stable) numerical fluxes for the PC using SBP CPR methods.

\subsection{Stability}
\label{Subsec:Stability}

We employ the normalised Hermite polynomials $\phi_i$ in the PC approach.
In general, stability of initial boundary value problems can be analysed similar
to \cite{nordstrom2017roadmap}, paying special attention to boundary conditions.
Here, we focus on energy/entropy conservative and dissipative numerical fluxes.
These will be linked to entropy conservative split forms of the volume terms
$\vec{\mathrm{VOL}}$ of SBP CPR methods as in \cite{fisher2013high}. Additionally,
they are used at the boundaries between two elements as in finite volume methods,
resulting in entropy conservative/stable semidiscretisations.

For sake of brevity, we introduce the mean value
$\mean{u}:= \frac{u(+)+u(-)}{2}$ and the jump $\jump{u} := u(+) - u(-)$ at a boundary between two elements.
For the system \eqref{BurgerGalerkinUQ}, the flux is given by
$f(u) = \frac{1}{2} A \, u$, where
$A_{ij} = \sum_{k = 0}^M  \E[\phi_i\phi_j\phi_k] u_k$.
Then, the system \eqref{BurgerGalerkinUQ} reads
\begin{equation}
\label{eq:Burgers_System2}
 \hat{\I} u_t+A u_x=0,
\end{equation}
with identity matrix $\hat{\I}$. The dimension
of the matrices depends on the selection of $M$, i.e.
$A(u) \in \R^{(M+1) \times (M+1)}$.
For stability, the $\L^2$ entropy
\begin{equation}
\label{eq:U}
  U = \frac{1}{2} |u|^2 = \frac{1}{2} \sum_{i=0}^M u_i^2
\end{equation}
is considered. It is a convex function of $u$ and the entropy variables are
simply $v = \partial_u U = u$,
i.e. the same as the conserved variables. The flux potential is given by
\begin{equation}
\label{eq:flux_potential}
 \psi
  =
  \frac{1}{6} u^T A \, u
  =
  \frac{1}{6} \sum_{i,j=0}^M u_i A_{ij} u_j
  =
  \frac{1}{6} \sum_{i,j,k=0}^M   \E[\phi_i\phi_j\phi_k] u_i u_j u_k.
\end{equation}
It fulfils
\begin{equation}
\begin{aligned}
  \partial_{u_l} \psi
  =
  \frac{1}{6} \sum_{i,j,k=0}^M \E[\phi_i\phi_j\phi_k]
  \left( \delta_{il} u_j u_k + u_i \delta_{jl} u_k + u_i u_j \delta_{kl} \right)
  =
  \frac{1}{2} \sum_{i,j=0}^M  \E[\phi_i\phi_j\phi_k] u_i u_j
  =
  \left[ f(u) \right]_l
\end{aligned}
\end{equation}
and can be used to construct the entropy flux
$F(u) = u^T f(u) - \psi(u)= \frac{1}{3} \sum_{i,j,k=0}^M
\E[\phi_i\phi_j\phi_k] u_i u_j u_k $, obeying
\begin{equation}
  \partial_u F(u)
  =
  \hat{\operatorname{I}} f(u) + u^T \partial_u f(u) - f(u)
  =
  u^T \partial_u f(u),
\end{equation}
see also Example \ref{BeispielM=2}.
Therefore, for a smooth solution $u$ of the conservation law
$\partial_t u + \partial_x f(u) = 0$, the entropy $U = \frac{1}{2} |u|^2$ fulfils
\begin{equation}
  \partial_t U(u)
  =
  \partial_u U(u) \partial_t u
  =
  - u^T \partial_x f(u)
  =
  - u^T \partial_u f(u) \partial_x u
  =
  - \partial_u F(u) \partial_x u
  =
  - \partial_x F(u).
\end{equation}
As a stability criterion, the entropy inequality $\partial_t U + \partial_x F \leq 0$
will be used. The numerical flux $\fnum$ is entropy stable in the sense of Tadmor
\cite{tadmor1987numerical, tadmor2003entropy}, i.e. in a semidiscrete scheme, if
\begin{equation}
\label{eq:fnum-EC-condition}
  \jump{u} \cdot \fnum \leq \jump{\psi},
\end{equation}
since the entropy variables are the same as the conserved variables $u$. The flux
$\fnum$ is entropy conservative if equality holds in \eqref{eq:fnum-EC-condition}.
 Condition \eqref{eq:fnum-EC-condition} considers especially the behaviour
of the numerical flux at the boundaries between two elements. We have to do the same in the following.
\begin{ex}\label{ex:FD}
To demonstrate the close connection to the FD framework, we consider first the matrix form \eqref{eq:Burgers_System2}.
For the extension of the SBP CPR methods to the system \eqref{eq:Burgers_System2}, we apply a tensor
product structure. $\otimes$  denotes the bilinear \emph{Kronecker product} and $A_G$ is the block
diagonal matrix, where the diagonal blocks are the symmetric matrices \eqref{MatrixA}.
If we apply the SBP CPR method for the system \eqref{eq:Burgers_System2} directly, we are
not able to prove stability similar to the FD framework. Therefore, we employ a skew-symmetric formulation
for \eqref{eq:Burgers_System2}. The resulting CPR method with Gauss-Lobatto-Legendre nodes
in one element $e$ reads
\begin{equation}
\label{eq:CPR_systemI}
   (\mat{\I}\otimes \hat{\I}) \partial_t \vec{u}
  + \frac{\beta}{2} (\mat{D}\otimes \hat{\I}) A_G \vec{u}
  + (1-\beta) \left( A_G(\mat{D}\otimes \hat{\I})\vec{u} \right)
                   + \left(\left(\mat{M}[^{-1}] \mat{R}[^T] \mat{B}\right)\otimes \hat{\I} \right) \left(
 \vecfnum - \frac{1}{2} (\mat{R} \otimes \hat{\I} ) A_G \vec{u} \right)
  =
  0,
\end{equation}
where $\vecfnum$ is the numerical flux
and $\vec{u}= \l(u_0( \zeta_0), \dots, u_0(\zeta_p),u_1(\zeta_0), \dots, u_M(\zeta_p)\right)^T$
is the combination vector from SBP CPR and polynomial chaos.
Investigating $\L^2$ stability, we multiply \eqref{eq:CPR_systemI} by
$\vec{u}^T (\mat{M}\otimes \hat{\I})$.
Applying the SBP property \eqref{eq:SBP}, $\beta=\frac{2}{3}$ and
simple calculations\footnote{Details of the calculation can be found in section
\ref{Appendix_Stability}.}, we get
\begin{equation}
\label{Gleichungst}
\begin{aligned}
  \frac{1}{2}\frac{\dif}{\dif t} \norm{u^{(e)}}^2_{\mat{M} \otimes \hat{\I}}
 =&
 \frac{1}{6} u^{(e),T}_R A(u^{(e)}_R) u^{(e)}_R -
 \frac{1}{6} u^{(e),T}_LA(u^{(e)}_L)u^{(e)}_L
 +u^{(e),T}_Lf^{\mathrm{num}, e}_L
 -u^{(e),T}_R f^{\mathrm{num}, e}_R.
\end{aligned}
\end{equation}
The value describes the change of the energy/entropy in one element.
To get the rate of change of the total entropy, the contributions
from all elements have to be summed up. Since the volume terms are entropy
conservative, i.e. only boundary terms remain at the right hand side of
\eqref{Gleichungst}, the behaviour at the element boundaries is essential.
We consider now two neighbouring elements as in Figure~\ref{fig:num-fluxes}.
Adding the contributions of the elements $e-1$ and $e$, we get
\begin{equation}
\label{Gleichungst_2}
\begin{aligned}
  \frac{1}{2}\frac{\dif}{\dif t} \norm{u^{(e-1,e)}}^2_{\mat{M} \otimes \hat{\I}}
 =&
\frac{1}{6} u^{(e-1),T}_R A(u^{(e-1)}_R) u^{(e-1)}_R -\frac{1}{6} u^{(e-1),T}_LA(u^{(e-1)}_L)u^{(e-1)}_L
 +u^{(e-1),T}_Lf^{\mathrm{num}, e-1}_L\\
 -&u^{(e-1),T}_R f^{\mathrm{num}, e-1}_R
 + \frac{1}{6} u^{(e),T}_R A(u^{(e)}_R) u^{(e)}_R -\frac{1}{6} u^{(e),T}_LA(u^{(e)}_L)u^{(e)}_L
 +u^{(e),T}_Lf^{\mathrm{num}, e}_L
 -u^{(e),T}_R f^{\mathrm{num}, e}_R.
\end{aligned}
\end{equation}
$u_R^{(e-1)}=u(-)$ and $u_L^{(e)}=u(+)$ are located at the same boundary. Since
the numerical flux is unique at every boundary, i.e. $f^{\mathrm{num}, e-1}_R = f^{\mathrm{num}, e}_L=f^\mathrm{num}$, we can reformulate the terms at the common
boundary in \eqref{Gleichungst_2} as
\begin{equation}\label{Gleichungstabilitaet1}
 \begin{aligned}
  &\underbrace{\frac{1}{6} u^{(e-1),T}_R A(u^{(e-1)}_R) u^{(e-1)}_R}_{\psi(-)}-\underbrace{u^{(e-1),T}_R}_{=u(-)} \underbrace{f^{\mathrm{num}, e-1}_R}_{=\fnum}
-\underbrace{\frac{1}{6} u^{(e),T}_LA(u^{(e)}_L)u^{(e)}_L}_{\psi(+)}
 + \underbrace{u^{(e),T}_L}_{=u(+)} \underbrace{f^{\mathrm{num}, e}_L}_{=\fnum} \\
 =&\psi(-)-\psi(+)+ u(+) \cdot f^\mathrm{num}-u(-) \cdot \fnum = \jump{u} \cdot \fnum - \jump{\psi}.
  \end{aligned}
\end{equation}
According to \eqref{eq:fnum-EC-condition}, this value is smaller than or equal
to zero for an entropy stable numerical flux. Applying this approach for every
boundary between two elements, we get that with an entropy stable numerical flux
the change of the total energy is
\begin{equation}
   \frac{1}{2}\frac{\dif}{\dif t} \sum_{r} \norm{u^{(r)}}^2_{\mat{M} \otimes \hat{\I}}\leq 0
\end{equation}
in a periodic setting or with compactly supported initial data. Otherwise,
boundary terms remain on the right hand side. The investigation of stable
boundary conditions goes beyond the scope of the present work.

\end{ex}

\begin{ex}
  For the PC order $M=0$, we get $A(u)=u_0$ and the flux function $f(u)=\frac{1}{2}u_0^2$.
  Thus, \eqref{Gleichungstabilitaet1} yields
  \begin{equation}
    \jump{u} \cdot \fnum - \jump{\psi}
   = \jump{u_0} \cdot \fnum-\frac{1}{6}\jump{u_0^3}.
  \end{equation}
  This is exactly the case of Burgers' equation discussed in \cite{ranocha2016summation}.
  Classical numerical fluxes as Osher's flux or local Lax-Friedrichs flux are entropy stable.
\end{ex}

\begin{ex}
  For $M=1$, we have
  \begin{equation}
    A=
    \begin{pmatrix}
    u_0 & u_1 \\
    u_1 & u_0
    \end{pmatrix}
    \text{ and }
    u=\begin{pmatrix}
    u_0 \\
    u_1
    \end{pmatrix}.
  \end{equation}
  The flux function is given by
  $f=  \begin{pmatrix} f_0 \\ f_1 \end{pmatrix}
  = \begin{pmatrix} \frac{1}{2}\left(u_0^2+u_1^2 \right) \\ u_0 u_1 \end{pmatrix}$.
 Now we have several possibilities to determine $\fnum$ in an adequate way.

  One way\footnote{See section \ref{Sub:General}
  for another.} is to approximate the mixed terms in the second component
  $f_1$ by the average of
  $u_0$ and $u_1$ and to replace $\frac{1}{2} u_1^2$ also by an analogous average.
  Then, the numerical flux is given by
  \begin{equation}
    \fnum
    =
    \begin{pmatrix}
      f_{00}+ \frac{1}{2}\overline{u_1^2} \\
      \mean{u_0} \cdot \overline{u_1}
    \end{pmatrix},
  \end{equation}
  where $f_{00}$ depends only on $u_0$.
  Inserting this $\fnum$ in \eqref{Gleichungstabilitaet1} yields
  \begin{equation}
  \begin{aligned}
    &
     \jump{u} \cdot \fnum - \jump{\psi}
    \\=&
    \frac{ u_0^3(-)-u_0^3(+) }{6} + \frac{ u^2_1(-)u_0(-)-u_1^2(+)u_0(+) }{2}
    +
    \jump{u_0} \left(f_{00} + \frac{ u_1^2(+)+u_1^2(-) }{4} \right)
    + \jump{u_1} \mean{u_0} \cdot \mean{u_1}
    \\=&
    \frac{ u_0^3(-)-u_0^3(+) }{6} - \jump{u_0}\frac{ u_1^2(-)+u_1^2(+) }{4}
    +\jump{u_0}\left(f_{00}+\frac{ u_1^2(+)+u_1^2(-) }{4} \right)
    = f_{00} \jump{u_0} - \frac{1}{6} \jump{u_0^3}.
  \end{aligned}
  \end{equation}
  This is exactly the same as for $M=0$.
  If we choose a classical numerical flux as before, we get entropy stability.
\end{ex}

\begin{ex}
  We obtain for $M=2$
  \begin{equation}
    A(u)
    =
    \begin{pmatrix}
      u_0  & u_1 & u_2 \\
      u_1 & u_0+\sqrt{2} u_2 & \sqrt{2}u_1\\
      u_2 & \sqrt{2}u_1 & u_0+2\sqrt{2} u_2
    \end{pmatrix},
  \end{equation}
  and the flux function is
  \begin{equation}
  \label{eq:M=2}
    f(u)=
    \frac{1}{2}
    \begin{pmatrix}
      u_0^2+u_1^2+u_2^2 \\
      2u_0u_1+2\sqrt{2}u_1u_2\\
      2u_0u_2+\sqrt{2}u_1^2+2\sqrt{2}u_2^2
    \end{pmatrix}
    =
    \begin{pmatrix}
      \frac{u_0^2+u_1^2+u_2^2}{2} \\
      u_0u_1+\sqrt{2}u_1u_2\\
      u_0u_2+\frac{1}{\sqrt{2}}u_1^2+\sqrt{2}u_2^2
    \end{pmatrix}.
  \end{equation}
  For the numerical flux function, we choose
  \begin{equation}
  \label{eq:M=2f_num}
    \fnum
    =
    \begin{pmatrix}
      f_{00}+\frac{1}{2} \overline{u_1^2} + \frac{1}{2}\overline{u_2^2} \\
      \mean{u_0}\;\overline{u_1}+\sqrt{2}\overline{u_1}\;\overline{u_2}\\
      2\sqrt{2} f_{22}+\mean{u_0}\;\overline{u_2} +\frac{\sqrt{2}}{2}\overline{u_1^2}
    \end{pmatrix},
  \end{equation}
  where we replaced in \eqref{eq:M=2} the values $u_0, \dots, u_2$ in all mixed terms
  and the squares $u_i^2$  by their mean values $\overline{u}$. $f_{00}$ and
  $f_{22}$ are again classical numerical flux functions like a local Lax-Friedrichs flux
  and they depend only on $u_i$.
  Employing the numerical flux function \eqref{eq:M=2f_num} in
  \eqref{Gleichungstabilitaet1}, we obtain by simple calculations
  \begin{equation}
   \jump{u} \cdot \fnum - \jump{\psi}
    =
    f_{00}\jump{u_0}-\frac{1}{6}\jump{u_0^3}
    +2\sqrt{2}\left(f_{22}\jump{u_2}-\frac{1}{6}\jump{u_2^3} \right).
  \end{equation}
  This is twice the Burgers' case from $M=0$, one time for $f_{00}$ and another
  time for $f_{22}$. If we choose  classical numerical fluxes for $f_{00}$ and
  $f_{22}$, we can ensure stability.
\end{ex}

\begin{ex}
  $M=3$ results in
  \begin{equation}
    A(u)
    =
    \begin{pmatrix}
      u_{0} & u_{1} & u_{2} & u_{3}
      \\
      u_{1} & u_{0} + \sqrt{2} u_{2} & \sqrt{2} u_{1} + \sqrt{3} u_{3} & \sqrt{3} u_{2}
      \\
      u_{2} & \sqrt{2} u_{1} + \sqrt{3} u_{3} & u_{0}
      + 2 \sqrt{2} u_{2} & \sqrt{3} u_{1} + 3 \sqrt{2} u_{3}
      \\
      u_{3} & \sqrt{3} u_{2} & \sqrt{3} u_{1} + 3 \sqrt{2} u_{3} & u_{0} + 3 \sqrt{2} u_{2}
    \end{pmatrix},
  \end{equation}
  and the flux function is given by
  \begin{equation}\label{eq:M=3flux}
    f(u)
    =
    \begin{pmatrix}
      \frac{u_{0}^{2}}{2} + \frac{u_{1}^{2}}{2} + \frac{u_{2}^{2}}{2}
      + \frac{u_{3}^{2}}{2}
      \\
      u_{0} u_{1} + \sqrt{2} u_{1} u_{2} + \sqrt{3} u_{2} u_{3}
      \\
      u_{0} u_{2} + \frac{\sqrt{2} u_{1}^{2}}{2} + \sqrt{3} u_{1} u_{3}
      + \sqrt{2} u_{2}^{2} + \frac{3 \sqrt{2}}{2} u_{3}^{2}
      \\
      u_{0} u_{3} + \sqrt{3} u_{1} u_{2} + 3 \sqrt{2} u_{2} u_{3}
    \end{pmatrix}.
  \end{equation}
  If we replace in the flux function \eqref{eq:M=3flux} all mixed terms and the
  squares by their simple means, we are not able to prove stability. The reason is
  the product of $u_j u_k$ in the $i$-component, where $k,j,i$ are distinct,
  see for example  $u_1 u_2$ in the last column of \eqref{eq:M=3flux}.
  These products provide problematic terms in the calculation of
  \eqref{Gleichungstabilitaet1} and it is not obvious, how to deal with these terms.
  In order to avoid this, we use the following trick:

  We apply a skew- symmetric form similar to \cite{fisher2013high}.
  We approximate $u_j  u_k$ with $\frac{1}{3} \left( \mean{u_j u_k} +2 \mean{u_j}\; \mean{u_k} \right)$
  in the $i$-th component, where $k,j,i$ are distinct.
  We handle the unproblematic terms as before and insert in the numerical flux
  function only  means of them. Then, the numerical flux reads as
  \begin{equation}\label{eq:M=3_fnum}
    f^\mathrm{num}
    =
    \begin{pmatrix}
      f_{00}
      + \frac{1}{2} \overline{ u_{1}^{2} }
      + \frac{1}{2} \overline{ u_{2}^{2} }
      + \frac{1}{2} \overline{ u_{3}^{2} }
      \\
      \overline{u_{0}} \cdot \overline{u_{1}}
      + \sqrt{2} \overline{u_{1}} \cdot \overline{u_{2}}
      + \sqrt{3} \left( \frac{1}{3} \overline{u_2 u_3} + \frac{2}{3} \overline{u_2} \cdot \overline{u_3} \right)
      \\
    \mean{u_0} \cdot \overline{u_2}
      + \frac{\sqrt{2}}{2} \overline{ u_{1}^{2} }
      + \sqrt{3} \left( \frac{1}{3} \overline{u_1 u_3} + \frac{2}{3} \overline{u_1}\cdot \overline{u_3} \right)
      + 2 \sqrt{2} f_{22}
      + \frac{3 \sqrt{2}}{2} \overline{ u_{3}^{2} }
      \\
      \mean{u_0} \cdot \overline{u_3}
      + \sqrt{3} \left( \frac{1}{3} \overline{u_1 u_2} + \frac{2}{3} \overline{u_1} \cdot \overline{u_2} \right)
      + 3 \sqrt{2} \overline{u_2} \cdot\overline{u_3}
    \end{pmatrix}.
  \end{equation}
  Finally, we obtain
  \begin{equation}\label{eq:M=3_en}
  \begin{aligned}
     \jump{u} \cdot \fnum - \jump{\psi}
    =&
    f_{00} \jump{ u_0 } - \frac{1}{6} \jump{ u_{0}^{3} }
    + 2 \sqrt{2} \left( f_{22} \jump{ u_2 } - \frac{1}{6} \jump{ u_{2}^{3} } \right).
  \end{aligned}
  \end{equation}
  This is analogous to $M=2$ and stability is ensured, if we choose a local
  Lax-Friedrichs flux or Osher's flux for $f_{00}$ and $f_{22}$.
\end{ex}

\begin{re}
  Furthermore, numerical calculations up to $M=9$ show that, if we apply this
  ansatz from $M=3$, where we replace the different $u_i$ by their means and / or
  using a split-form and also apply the entropy conservative flux
  \begin{equation}
    f_{ii}=\frac{1}{6}\overline{u_i^2}+\frac{1}{3}\mean{u_i}^2,
  \end{equation}
  we always obtain an entropy conservative numerical fluxes. This procedure should
  work in general.
\end{re}

\subsection{Construction of Numerical Fluxes}
\label{Sub:General}

Here, we  analyse the numerical flux function for any order $M\in \N_0$.
For the first time, we present a general approach to determine an entropy conservative numerical flux
in the context of polynomial chaos using SBP CPR methods. We
analyse inequality \eqref{eq:fnum-EC-condition}  in the componentwise setting,
where the flux potential is given by
equation \eqref{eq:flux_potential} as
\begin{equation}
 \psi
  =
  \frac{1}{6} u^T A \, u
  =
  \frac{1}{6} \sum_{i,j=0}^M u_i A_{ij} u_j
  =
  \frac{1}{6} \sum_{i,j,k=0}^M  \E[\phi_i\phi_j\phi_k] u_i u_j u_k.
\end{equation}

For the investigation of \eqref{eq:fnum-EC-condition}, we need a discrete analogue
of the product rule. For two variables, it is
\begin{equation}
\label{eq:product-2}
\begin{aligned}
  \jump{u_i u_j}
  &=
  u_i(+) u_j(+)- u_i(-) u_j(-)
  \\&=
  \frac{u_i(+) + u_i(+)}{2} (u_j(+) - u_j(-)) + (u_i(+) - u_i(-)) \frac{u_j(+) + u_j(-)}{2}
  \\&=
  \mean{u_i} \jump{u_j} + \jump{u_i} \mean{u_j}.
\end{aligned}
\end{equation}
For three variables, it can be written as
\begin{equation}
\label{eq:product-3}
\begin{aligned}
  \jump{u_i u_j u_k}
  =&
  \mean{u_i} \jump{u_j u_k} + \jump{u_i} \mean{u_j u_k}
  \\=&
  \mean{u_i} \cdot \mean{u_j} \jump{u_k} + \mean{u_i} \jump{u_j} \mean{u_k}
  + \jump{u_i} \mean{u_j u_k}
  \\=&
  \mean{u_j} \cdot \mean{u_k} \jump{u_i} + \mean{u_j} \jump{u_k} \mean{u_i}
  + \jump{u_j} \mean{u_k u_i}
  \\=&
  \mean{u_k} \cdot \mean{u_i} \jump{u_j} + \mean{u_k} \jump{u_i} \mean{u_j}
  + \jump{u_k} \mean{u_i u_j}
  \\=&
    \left( \frac{1}{3} \mean{u_i u_j} + \frac{2}{3} \mean{u_i}  \cdot \mean{u_j} \right)
    \jump{u_k}
  + \left( \frac{1}{3} \mean{u_j u_k} + \frac{2}{3} \mean{u_j} \cdot \mean{u_k} \right)
    \jump{u_i}
  + \left( \frac{1}{3} \mean{u_k u_i} + \frac{2}{3} \mean{u_k} \cdot \mean{u_i} \right)
    \jump{u_j}.
\end{aligned}
\end{equation}
The first two equalities are obtained by using \eqref{eq:product-2}, the following
equalities by cyclic permutation of the indices and the last equality by averaging
these three forms.
Thus, the jump of the flux potential $\psi$ can be written as
\begin{equation}\label{eq:prduct_psi}
\begin{aligned}
 &
  \jump{\psi}
  =
  \frac{1}{6} \sum_{i,j,k=0}^M \E[\phi_i\phi_j\phi_k]
  \jump{u_i u_j u_k}
  \\&
  =
  \frac{1}{6} \sum_{i,j,k=0}^M  \E[\phi_i\phi_j\phi_k]  \left(
    \left( \frac{1}{3} \mean{u_i u_j} + \frac{2}{3} \mean{u_i} \cdot \mean{u_j} \right)
    \jump{u_k}
    + \left( \frac{1}{3} \mean{u_j u_k} + \frac{2}{3} \mean{u_j} \cdot \mean{u_k} \right)
    \jump{u_i}
    + \left( \frac{1}{3} \mean{u_k u_i} + \frac{2}{3} \mean{u_k} \cdot \mean{u_i} \right)
    \jump{u_j}
  \right)
  \\&
  =
  \frac{1}{2} \sum_{i,j,k=0}^M \E[\phi_i\phi_j\phi_k]
  \left( \frac{1}{3} \mean{u_i u_j} + \frac{2}{3} \mean{u_i} \cdot \mean{u_j} \right) \jump{u_k}.
\end{aligned}
\end{equation}
Therefore, defining the numerical flux as
\begin{equation}
\label{eq:fnum-EC}
  [\fnum]_k
  =
  \frac{1}{2} \sum_{i,j=0}^M  \E[\phi_i\phi_j\phi_k]
  \left( \frac{1}{3} \mean{u_i u_j} + \frac{2}{3} \mean{u_i}\cdot  \mean{u_j} \right),
\end{equation}
it is entropy conservative, i.e. it fulfils $\jump{u} \cdot \fnum = \jump{\psi}$.

Before we consider \eqref{eq:fnum-EC} in the semidiscrete formulation of our SBP CPR method,
we mention some properties of the entropy conservative flux \eqref{eq:fnum-EC}.
\begin{itemize}
 \item
Of course, the entropy conservative flux is not
unique if $M \geq 1$, i.e. if not only a scalar problem is considered.
For a scalar problem, the canonical entropy conservative flux is given by
\begin{equation}
  \fnum_{M=0}
  =
  \frac{\jump{\psi}}{\jump{u}}
  =
  \frac{1}{6} \mean{u_0^2} + \frac{1}{3} \mean{u_0}^2,
\end{equation}
which is the entropy conservative flux for Burgers' equation with the $\L^2$
entropy, used inter alia in  \cite{gassner2013skew}.
\item
The numerical flux \eqref{eq:fnum-EC} is also the entropy conservative numerical flux of
Tadmor \cite[Equation (4.6a)]{tadmor1987numerical}, obtained by integration in
phase space:
\begin{equation}
\begin{aligned}
  &
  [\fnum]_k(u(-), u(+))
  =
  \left[ \int_0^1 f\left( (1-s) u(-) + s u(+) \right) \dif s \right]_k
  \\
  =&
  \sum_{i,j=0}^M  \frac{  \E[\phi_i\phi_j\phi_k] }{2}
  \int_0^1 \left( (1-s) u_i(-) + s u_i(+)  \right) \left( (1-s) u_j(-) + s u_j(+)  \right) \dif s
  \\
  =&
  \sum_{i,j=0}^M  \frac{  \E[\phi_i\phi_j\phi_k] }{2}
  \int_0^1 \left(
    (1-s)^2 u_i(-) u_j(-) + s(1-s) (u_i(-) u_j(+) + u_i(+) u_j(-) ) + s^2 u_i(+) u_j(+)
  \right) \dif s\\
  =&
  \frac{1}{2} \sum_{i,j=0}^M   \E[\phi_i\phi_j\phi_k]
  \left(
    \frac{1}{3} u_i(-) u_j(-) + \frac{1}{6} (u_i(-) u_j(+) + u_i(+) u_j(-) )
    + \frac{1}{3} u_i(+) u_j(+)
  \right)
  \\
  =&
  \frac{1}{2} \sum_{i,j=0}^M   \E[\phi_i\phi_j\phi_k]
  \left(
    \frac{1}{3} \mean{u_i u_j} + \frac{2}{3} \mean{u_i} \cdot \mean{u_j}
  \right).
\end{aligned}
\end{equation}
\end{itemize}

In section \ref{Subsec:Stability}, we consider the skew-symmetric form \eqref{eq:CPR_systemI}
with $\beta=\frac{2}{3}$. Using a subcell flux differencing form for a nodal diagonal-norm
SBP basis introduced in \cite{fisher2013discretely, fisher2013high} and applied in \cite{gassner2016split},
we are able to recover it.
 To study entropy stability of the general setting (i.e. nodal bases
not including boundary nodes) in the semidiscrete formulation,
we need a description of the volume terms of the SBP semidisretisation.
Here, $[\mathrm{VOL}_k]_n$ is the $n$-th entry of the volume term  of the component $k$, where
$n \in \{0,1,\dots,p\}$ and $k\in \{0,1,\dots,M \} $.
For example $[\mathrm{VOL}_0]_n$ describes the volume term of the component $u_0$
at $\zeta_n$ for $n \in \{0,1,\dots,p\}$.
We insert the entropy conservative flux function \eqref{eq:fnum-EC} in the general flux
differencing formulation.
In the following, the first index of $u_{i,m}$ indicates the component $u_i$ of $u$
and the second index the spatial location $\zeta_m$ at which $u_i$ is evaluated.
Finally, we get
\begin{align*}
\stepcounter{equation}\tag{\theequation}
\label{eq:fnum-EC-volume-terms}
  \left[ \mathrm{VOL}_k \right]_n
  =&
  \sum_{m=0}^p 2 \mat{D}[_{n,m}] \frac{1}{2} \sum_{i,j=0}^M   \E[\phi_i\phi_j\phi_k]
  \left(
      \frac{1}{3} \frac{u_{i,m} u_{j,m} + u_{i,n} u_{j,n}}{2}
    + \frac{2}{3} \frac{u_{i,m} + u_{i,n}}{2}  \frac{u_{j,m} + u_{j,n}}{2}
  \right)
  \\
  =&
  \sum_{i,j=0}^M   \E[\phi_i\phi_j\phi_k] \sum_{m=0}^p \mat{D}[_{n,m}]
  \left(
      \frac{1}{6} u_{i,m} u_{j,m}
    + \frac{1}{6} \left( u_{i,m} + u_{i,n} \right) \left( u_{j,m} + u_{j,n} \right)
  \right)
  \\
  =&
  \frac{1}{6} \sum_{i,j=0}^M   \E[\phi_i\phi_j\phi_k] \sum_{m=0}^p \mat{D}[_{n,m}]
  \left(
      2 u_{i,m} u_{j,m} + u_{i,m} u_{j,n} + u_{i,n} u_{j,m}
  \right)
  \\
  =&
  \frac{1}{3} \sum_{i,j=0}^M   \E[\phi_i\phi_j\phi_k] \sum_{m=0}^p \mat{D}[_{n,m}]
  \left(
      u_{i,m} u_{j,m} + u_{i,m} u_{j,n}
  \right)
  \\
  =&
  \left[
    \frac{1}{3} \sum_{i,j=0}^M   \E[\phi_i\phi_j\phi_k]
    \left( \mat{D} \vec{u_i u_j} + \mat{u_j} \mat{D} \vec{u_i} \right)
  \right]_n.
\end{align*}
The exactness of
the derivative for constants $\mat{D} \vec{1} = 0$ has been used, resulting in
$\sum_{m=0}^p \mat{D}[_{n,m}] = 0$, and the symmetry with respect to the indices $i,j$
has been exploited.
We require this volume term \eqref{eq:fnum-EC-volume-terms} in our discretisation of the SBP-CPR method
to prove entropy stability and conservation (across elements) in Theorem \ref{Theorem1}.

For entropy conservative fluxes $\fnum_\mathrm{ec}$, i.e.
$\jump{u} \cdot \fnum_\mathrm{ec} = \jump{\psi}$,
spurious oscillations in the numerical solution typically get
quite strong.
Therefore, the entropy conservative flux $\fnum_\mathrm{ec}$ at the boundaries
typically gets equipped with a dissipative term $-\frac{1}{2} Q \jump{u}$, i.e.
\begin{equation}
  \fnum = \fnum_\mathrm{ec} - \frac{1}{2} Q \jump{u}.
\end{equation}
For a positive semi-definite \textit{dissipation matrix} $Q$, this yields
\begin{equation}
  \jump{u} \cdot \fnum =
  \underbrace{\jump{u} \cdot \fnum_\mathrm{ec}}_{=\jump{\psi}}
    - \frac{1}{2} \underbrace{\jump{u} \cdot Q \jump{u}}_{\geq 0}
  \leq \jump{\psi}
\end{equation}
and thus an entropy stable flux.
For the numerical tests in section \ref{sec:numerical-tests}, the dissipation
matrix $Q$ was chosen in a local Lax-Friedrichs sense:
\begin{equation}
  Q = \lambda \, \I \quad \text{with} \quad \lambda = \max\set{ \abs{\lambda(-)},
  \abs{\lambda(+)} },
\end{equation}
where $\abs{\lambda(\pm)}$ is the greatest absolute value of all eigenvalues of
$A(u(\pm))$.

\subsection{Extension to a General Setting }
\label{sec:general}

In \eqref{eq:CPR_systemI}, we applied Gauss-Lobatto-Legendre nodes for reasons of simplicity.
In our theoretical investigation in the last section \ref{Sub:General}, we assumed
to have a diagonal-norm SBP basis including boundary nodes.
Moreover, we can also employ in our approach
Gauss-Legendre nodes or a modal Legendre basis as it was presented in
\cite{ranocha2017extended, ranocha2016summation}.
As distinguished from Gauss-Lobatto-Legendre, Gauss-Legendre nodes don't include
the boundary
and we need a further correction term for the restriction to the boundaries to guarantee
stability as demonstrated e.g. in \cite{ranocha2016summation}.
The extension to modal bases was done in \cite{ranocha2017extended}. However,
we focus only on modal Legendre basis, where an exact multiplication
of polynomials followed by an exact $\L^2$ projection is used for multiplication.
Using the $\mat{M}$-adjoint $\mat{u}[^*]=\mat{M}[^{-1}]\mat{u}[^T]\mat{M}$, the SBP CPR
method for Burgers' equation with a general basis (modal Legendre basis or
nodal Gauss-Legendre / Gauss-Lobatto-Legendre) reads
 \begin{equation}
\label{eq:Burgers-semidisc}
  \partial_t \vec{u}
  =
  - \frac{1}{3} \mat{D} \mat{u} \vec{u}
  - \frac{1}{3} \mat{u}[^*] \mat{D} \vec{u}
  + \mat{M}[^{-1}] \mat{R}[^T] \mat{B} \left(
      \vecfnum
      - \frac{1}{3} \mat{R} \mat{u} \vec{u}
      - \frac{1}{6} \left( \mat{R} \vec{u} \right) \bullet \left( \mat{R} \vec{u} \right)
    \right),
\end{equation}
see \cite{ranocha2017extended} for details\footnote{With  $ \bullet$, we denote the componentwise multiplication (Hadamard product) of two vectors.}.
Applying this to system \eqref{eq:Burgers_System2} and using our approach
from section \ref{Sub:General}, we get
\begin{thm}\label{Theorem1}
If the numerical flux $\fnum$ is entropy stable in the sense of Tadmor \eqref{eq:fnum-EC-condition},
the SBP CPR method
for the system \eqref{eq:Burgers_System2}, written componentwise as
\begin{equation}
\label{eq:semidiscsystem}
\begin{aligned}
  \partial_t \vec{u_k}
  =&
  - \frac{1}{3} \sum_{i,j=0}^M    \E[\phi_i\phi_j\phi_k]
    \left( \mat{D} \vec{u_i u_j} + \mat{u_j}[^*] \mat{D} \vec{u_i} \right)
  \\ &
  - \mat{M}[^{-1}] \mat{R}[^T] \mat{B}\left(
    \vec{f_k}^\mathrm{num}
    - \sum_{i,j=0}^M   \E[\phi_i\phi_j\phi_k] \left(
      \frac{1}{3} \mat{R} \vec{u_i u_j}
      + \frac{1}{6}\left(\mat{R} \vec{u_i}\right)  \bullet \left(\mat{R} \vec{u_j}\right)
      \right)
    \right),
\end{aligned}
\end{equation}
is conservative (across elements) and entropy stable in the discrete norm $||\cdot||_{\mat{M}\otimes \hat{\I}}$ induced by $\mat{M}$.

\end{thm}
\begin{proof}
First, we demonstrate the conservation property. Therefore, we multiply equation
\eqref{eq:semidiscsystem} from the left by $\vec{1}^T \mat{M}$.
Using the SBP property \eqref{eq:SBP}, we obtain
\begin{equation}
\begin{aligned}
 \vec{1}^{T}\mat{M}\partial_t \vec{u_k}
  =&
  - \frac{1}{3} \sum_{i,j=0}^M   \E[\phi_i\phi_j\phi_k]
    \vec{1}^{T}\mat{M}\left( \mat{D} \vec{u_i u_j} + \mat{u_j}[^*] \mat{D} \vec{u_i} \right)
  \\&
  - \vec{1}^{T}\mat{M} \mat{M}[^{-1}] \mat{R}[^T] \mat{B} \left(
    \vec{f_k}^\mathrm{num}
    - \sum_{i,j=0}^M    \E[\phi_i\phi_j\phi_k] \left(
      \frac{1}{3} \mat{R} \vec{u_i u_j}
      + \frac{1}{6}\left(\mat{R} \vec{u_i}\right)  \bullet \left(\mat{R} \vec{u_j}\right)
      \right)
    \right)\\
    =&- \frac{1}{3} \sum_{i,j=0}^M   \E[\phi_i\phi_j\phi_k] \left(
    \vec{1}^T \mat{R}[^T] \mat{B} \mat{R} \vec{u_i u_j}- \vec{1}^{T}\mat{D}[^T] \mat{M} \vec{u_i u_j}+   \vec{1}^{T}\mat{M} \mat{M}[^{-1}]\mat{u_j}[^T]\mat{M} \mat{D} \vec{u_i} \right)\\
     & - \vec{1}^{T}\mat{R}[^T] \mat{B} \left(
       \vec{f_k}^\mathrm{num}
        - \sum_{i,j=0}^M    \E[\phi_i\phi_j\phi_k] \left(
       \frac{1}{3} \mat{R} \vec{u_i u_j}
       + \frac{1}{6}\left(\mat{R} \vec{u_i}\right) \bullet \left(\mat{R} \vec{u_j}\right)
       \right)
         \right).
\end{aligned}
\end{equation}
Applying $\mat{D}\vec{1}=0$ and the SBP property \eqref{eq:SBP}, we get
\begin{equation}
\begin{aligned}
 \vec{1}^{T}\mat{M}\partial_t \vec{u_k}
 =&
 - \frac{1}{3} \sum_{i,j=0}^M   \E[\phi_i\phi_j\phi_k] \left(
  \vec{1}^{T}\mat{u_j}[^T]\mat{M} \mat{D} \vec{u_i} \right)
 -
   \vec{1}^{T}\mat{R}[^T] \mat{B} \left(
       \vec{f_k}^\mathrm{num}
        - \frac{1}{6}\sum_{i,j=0}^M    \E[\phi_i\phi_j\phi_k]
        \left(\mat{R} \vec{u_i}\right)  \bullet \left(\mat{R} \vec{u_j}\right)
       \right)\\
    =&-  \sum_{i,j=0}^M    \E[\phi_i\phi_j\phi_k] \left(
 \frac{1}{6} \vec{u_j}^T \mat{M} \mat{D} \vec{u_i}- \frac{1}{6}\vec{u_j}^T \mat{D}[^T] \mat{M} \vec{u_i}+ \frac{1}{6} \vec{u_j}^T \mat{R}[^T] \mat{B} \mat{R} \vec{u_i}
 \right) \\
  &-
   \vec{1}^{T}\mat{R}[^T] \mat{B} \left(
      \vec{f_k}^\mathrm{num}
       - \frac{1}{6}\sum_{i,j=0}^M    \E[\phi_i\phi_j\phi_k]
        \left(\mat{R} \vec{u_i}\right) \bullet \left(\mat{R} \vec{u_j}\right)
       \right).
\end{aligned}
\end{equation}
Using the symmetry with respect to the indices $i,j$ yields
\begin{equation}
  \vec{1}^{T}\mat{M}\partial_t \vec{u_k}=-
   \vec{1}^{T}\mat{R}[^T] \mat{B}
      \vec{f_k}^\mathrm{num},
\end{equation}
and conservation across elements is shown, since the numerical flux is determined
uniquely at every boundary.

For stability,  we multiply equation \eqref{eq:semidiscsystem} by $\vec{u_k}^{T}\mat{M}$.
Using the SBP property \eqref{eq:SBP} yields
\begin{equation}
\begin{aligned}
  \vec{u_k}^{T}\mat{M}\partial_t \vec{u_k}
  =&
  - \frac{1}{3} \sum_{i,j=0}^M    \E[\phi_i\phi_j\phi_k]
    \vec{u_k}^{T}\mat{M}\left( \mat{D} \vec{u_i u_j} + \mat{u_j}[^*] \mat{D} \vec{u_i} \right)
  \\&
  - \vec{u_k}^{T}\mat{M} \mat{M}[^{-1}] \mat{R}[^T] \mat{B} \left(
    \vec{f_k}^\mathrm{num}
    - \sum_{i,j=0}^M   \E[\phi_i\phi_j\phi_k] \left(
      \frac{1}{3} \mat{R} \vec{u_i u_j}
      + \frac{1}{6}\left(\mat{R} \vec{u_i}\right)  \bullet \left(\mat{R} \vec{u_j}\right)
      \right)
    \right)\\
    =&- \frac{1}{3} \sum_{i,j=0}^M   \E[\phi_i\phi_j\phi_k] \left(
    \vec{u_k}^{T}\mat{M}\mat{D} \vec{u_i u_j} +   \vec{u_k}^{T}\mat{M} \mat{M}[^{-1}]\mat{u_j}[^T]\mat{M} \mat{D} \vec{u_i} \right)\\
     & - \vec{u_k}^{T}\mat{R}[^T] \mat{B} \left(
       \vec{f_k}^\mathrm{num}
        - \sum_{i,j=0}^M   \E[\phi_i\phi_j\phi_k] \left(
       \frac{1}{3} \mat{R} \vec{u_i u_j}
       + \frac{1}{6}\left(\mat{R} \vec{u_i}\right) \bullet \left(\mat{R} \vec{u_j}\right)
       \right)
         \right)\\
          =&- \frac{1}{3} \sum_{i,j=0}^M   \E[\phi_i\phi_j\phi_k] \left(
    \vec{u_k}^{T} \mat{R}[^T] \mat{B} \mat{R} \vec{u_i u_j}- \vec{u_k}^{T}\mat{D}[^T] \mat{M} \vec{u_i u_j} +   \vec{u_k}^{T} \mat{u_j}[^T]\mat{M} \mat{D} \vec{u_i} \right)\\
     & - \vec{u_k}^{T}\mat{R}[^T] \mat{B} \left(
       \vec{f_k}^\mathrm{num}
        - \sum_{i,j=0}^M   \E[\phi_i\phi_j\phi_k]\left(
       \frac{1}{3} \mat{R} \vec{u_i u_j}
       + \frac{1}{6}\left(\mat{R} \vec{u_i}\right)  \bullet\left(\mat{R} \vec{u_j}\right)
       \right)
         \right)\\
        =&- \frac{1}{3} \sum_{i,j=0}^M   \E[\phi_i\phi_j\phi_k] \left(- \vec{u_k}^{T}\mat{D}[^T] \mat{M} \vec{u_i u_j} +   \vec{u_k}^{T} \mat{u_j}[^T]\mat{M} \mat{D} \vec{u_i} \right)\\
       & - \vec{u_k}^{T}\mat{R}[^T] \mat{B} \left(
       \vec{f_k}^\mathrm{num}
        - \sum_{i,j=0}^M    \E[\phi_i\phi_j\phi_k]  \frac{1}{6}\left(\mat{R} \vec{u_i}\right) \bullet \left(\mat{R} \vec{u_j}\right)  \right).
\end{aligned}
\end{equation}
We sum over $k$ and get
\begin{equation}\label{eq:Stabilitycomp.}
 \begin{aligned}
  \frac{1 }{2}\frac{\dif}{\dif t}\norm{u}^2_{\mat{M}\otimes \hat{\I}}=&
- \frac{1}{3} \sum_{i,j,k=0}^M    \E[\phi_i\phi_j\phi_k] \left(- \vec{u_k}^{T}\mat{D}[^T] \mat{M} \vec{u_i u_j} +   \vec{u_k}^{T} \mat{u_j}[^T]\mat{M} \mat{D} \vec{u_i} \right)\\
       & -\sum_{k=0}^M \vec{u_k}^{T}\mat{R}[^T] \mat{B} \left(
       \vec{f_k}^\mathrm{num}
        - \sum_{i,j=0}^M   \E[\phi_i\phi_j\phi_k]  \frac{1}{6}\left(\mat{R} \vec{u_i}\right) \bullet \left(\mat{R} \vec{u_j}\right)  \right).
 \end{aligned}
\end{equation}
Using the symmetry with respect to the indices $i,j,k$, we permute
$\vec{u_k}^{T}\mat{D}[^T] \mat{M} \vec{u_i u_j}$ to
$ \vec{u_k}^{T} \mat{u_j}\mat{M} \mat{D} \vec{u_i}$ and
the first sum of \eqref{eq:Stabilitycomp.} is zero. Finally, we obtain
\begin{equation}\label{eq:Stabilitycomp2}
 \frac{1 }{2}\frac{\dif}{\dif t}\norm{u}^2_{\mat{M}\otimes \hat{\I}}=
 -\sum_{k=0}^M \vec{u_k}^{T}\mat{R}[^T] \mat{B}
       \vec{f_k}^\mathrm{num}
        +  \frac{1}{6} \sum_{i,j,k=0}^M   \E[\phi_i\phi_j\phi_k] \vec{u_k}^{T}\mat{R}[^T] \mat{B}  \left( \mat{R} \vec{u_i} \right) \bullet \left( \mat{R} \vec{u_j}\right)  .
\end{equation}
This is the rate of change of the energy in one element
 as already described in Example \ref{ex:FD} for Gauss-Lobatto-Legendre nodes.
 If we rewrite \eqref{eq:Stabilitycomp2} including the elemental index $e$, it is
 \begin{equation}\label{eq:Stabilitycomp3}
  \frac{1 }{2}\frac{\dif}{\dif t}\norm{u^{(e)}}^2_{\mat{M}\otimes \hat{\I}}=
  \sum_{k=0}^M \left(-u_{k,R}^{(e)} f^{\operatorname{num},(e)}_{k,R} +u_{k,L}^{(e)}
  f^{\operatorname{num},(e)}_{k,L}\right)
  +  \frac{1}{6} \sum_{i,j,k=0}^M
   \E[\phi_i\phi_j\phi_k] \left(u_{k,R}^{(e)}u_{i,R}^{(e)}u_{j,R}^{(e)}-
   u_{k,L}^{(e)}u_{i,L}^{(e)}u_{j,L}^{(e)}\right)
 \end{equation}
and the contribution of two elements is
\begin{equation}
\begin{aligned}
  \frac{1 }{2}\frac{\dif}{\dif t}\norm{u^{(e-1,e)}}^2_{\mat{M}\otimes \hat{\I}}=&
  \sum_{k=0}^M \left(-u_{k,R}^{(e-1)} f^{\operatorname{num},(e-1)}_{k,R} +u_{k,L}^{(e-1)}
  f^{\operatorname{num},(e-1)}_{k,L}  -u_{k,R}^{(e)} f^{\operatorname{num},(e)}_{k,R} +u_{k,L}^{(e)}
  f^{\operatorname{num},(e)}_{k,L}   \right)\\
  +&  \frac{1}{6} \sum_{i,j,k=0}^M
   \E[\phi_i\phi_j\phi_k] \left(u_{k,R}^{(e-1)}u_{i,R}^{(e-1)}u_{j,R}^{(e-1)}-
   u_{k,L}^{(e-1)}u_{i,L}^{(e-1)}u_{j,L}^{(e-1)}+u_{k,R}^{(e)}u_{i,R}^{(e)}u_{j,R}^{(e)}-
   u_{k,L}^{(e)}u_{i,L}^{(e)}u_{j,L}^{(e)}\right).
 \end{aligned}
 \end{equation}
 Focusing on the terms at the common boundary, we can reformulate these terms using
 \eqref{eq:prduct_psi} and obtain
 \begin{equation}
 \begin{aligned}
    &\sum_{k=0}^M \left(-u_{k,R}^{(e-1)} f^{\operatorname{num},(e-1)}_{k,R}
  +u_{k,L}^{(e)}
  f^{\operatorname{num},(e)}_{k,L} \right)+ \frac{1}{6} \sum_{i,j,k=0}^M
   \E[\phi_i\phi_j\phi_k]
  \left(u_{k,R}^{(e-1)}u_{i,R}^{(e-1)}u_{j,R}^{(e-1)}-
   u_{k,L}^{(e)}u_{i,L}^{(e)}u_{j,L}^{(e)}\right)\\
   \quad & = \jump{u} \cdot \fnum-\jump{\psi}\leq 0,
 \end{aligned}
 \end{equation}
 since the numerical flux is entropy stable in the sense of Tadmor.
 This means that inequality \eqref{eq:fnum-EC-condition} holds and by summing \eqref{eq:Stabilitycomp3} up
 over all elements, we get stability.
\end{proof}

\section{Reference Solutions}
\label{sec:reference-solutions}

In this section, analytical solutions to two test problems will be
determined.
Furthermore, their coefficients in the normalised Hermite basis of
the underlying PC method will be computed.
This is done in order to quantify the accuracy of the numerical method in
section \ref{sec:numerical-tests}.

The reference solution to the stochastic Riemann problem with an initial shock
in section \ref{subsec:reference-solutions_shock} was already investigated
in \cite{pettersson2009numerical,pettersson2015polynomial} and is thus just
briefly revised.
Yet, it should be stressed that we use a simplified recursion relation
\eqref{eq:recursion_relation} compared to the one utilised in
\cite{pettersson2009numerical,pettersson2015polynomial}.
This improvement not just renders the calculations more straight forward, but
further enables us to apply the approach to most general orthogonal
polynomials.

\subsection{Stochastic Riemann Problem with an Initial Rarefaction}
\label{subsec:reference-solutions_rarefaction}

Consider the stochastic Riemann problem with an initial rarefaction ($a > 0$)
of uncertain strength located at $x_0 \in [0,1]$, i.e.
\begin{equation}\label{rarefaction}
\begin{aligned}
  u(x,0,\xi(\omega)) = & \begin{cases}
		  u_L = -a + p( \xi(\omega)) & \text{ if } x < x_0, \\
		  u_R = a + p( \xi(\omega)) & \text{ if } x > x_0,
                 \end{cases} \\
  \xi \sim & \ \mathcal{N}(0,1).
\end{aligned}
\end{equation}
In this work, $p(\xi(\omega)) = b\xi(\omega)$ depends linearly on $\xi(\omega)$.\footnote{
The value of $\xi(\omega)$ lies in $\R$, we will use $y=\xi(\omega)$ for the notation.}
The distribution of $\xi$ is the standard Gaussian distribution $\mathcal{N}(0,1)$.
Further, a constant $a > 0$ will be assumed.
Boundary conditions are desired that make the $\L^2$ norm of $u$ over $D = [0,1]$
bounded and thus yield a well-posed problem \cite{pettersson2015polynomial}.
Since we do not concentrate on boundary conditions in this article, we apply
a common ad-hoc procedure as follows.
For the initial condition \eqref{rarefaction}, no boundary conditions are enforced
at all, which models an outflow behaviour. This is implemented numerically by not
adding a surface term at the corresponding boundaries. For a sufficiently small
time, the solution will not interact with the boundary and this treatment yields
acceptable results.

The test problem \eqref{rarefaction} also seems to be more appropriate to quantify
the accuracy of a high order method than the later one in section
\ref{subsec:reference-solutions_shock} and was not treated by Pettersson et al.
in \cite{pettersson2009numerical,pettersson2015polynomial}.

For time $t > 0$, the analytical (entropy) solution is given by
\begin{equation}\label{solution_rarefaction}
  u(x,t,y) = \begin{cases}
      u_L \ & \text{ if } x < x_0 + t u_L, \\
      \frac{x-x_0}{t} \ & \text{ if } x_0 + t u_L < x < x_0 + t u_R, \\
      u_R \ & \text{ if } x > x_0 + t u_R,
  \end{cases}
\end{equation}
where $u_L = -a + b y$ and $u_R = a + b  y$.

The coefficients of the complete PC expansion ($M \to \infty$) can be
calculated for any given $x$ and $t$ by
\begin{equation}
  u_i(x,t) = \int_{-\infty}^{\infty} u(x,t,y) \phi_i(y) \varrho(y) \dif y.
\end{equation}
Furthermore, the coefficients of the initial
condition at $t=0$ reduce to
\begin{equation}
  u_i(x,0) = \begin{cases}
      -a \delta_{i,0} + b \delta_{i,1} & \text{ if } x < x_0, \\
      a \delta_{i,0} + b \delta_{i,1} & \text{ if } x > x_0.
    \end{cases}
\end{equation}
In our implementation, we analogously set $u_0(x_0,0) = -a$ for $x_0$ in the
interior of an element.
If $x_0$ is a (Gauss-Lobatto-Legendre) point at some element boundary, we set
$u_0(x_-,0) = -a$ for the right boundary point of the element to the left
$x_-=x_0$ and $u_0(x_+,0) = a$ for the left boundary point of the element to
the right $x_+=x_0$.

For $t > 0$, the analytical solution \eqref{solution_rarefaction} can also be
written as
\begin{equation}\label{solution_rarefaction_new}
 u(x,t, y) = \begin{cases}
      -a + b y \ & \text{ if }  y> y_2, \\
      \frac{x-x_0}{t} \ & \text{ if }  y_1 < y < y_2, \\
      a + b y \ & \text{ if }  y < y_1,
  \end{cases}
\end{equation}
where $ y_1:= \frac{x-x_0-at}{bt}$ and $ y_2:= \frac{x-x_0+at}{bt}$,
and the coefficients of the complete PC expansion are given by
\begin{equation}
\label{eq:coefficients-rarafection}
\begin{aligned}
  u_i(x,t)
  = a\delta_{i,0} + b\delta_{i,1}
    + by_1 \int_{y_1}^{y_2} \phi_i(y) \varrho(y) \dif y
   - 2a \int_{y_2}^{\infty} \phi_i(y) \varrho(y) \dif y
    - b \int_{y_1}^{y_2} y \phi_i(y) \varrho(y) \dif y
\end{aligned}
\end{equation}
for any given $x,t$ and $i\geq0$.
By the recursion relation
\begin{equation}\label{eq:recursion_relation}
  \phi_{i}(y) \rho(y) = - \frac{1}{\sqrt{i}} \frac{\dif}{\dif y} \left(
\phi_{i-1}(y) \rho(y) \right), \quad i\geq1
\end{equation}
for normalised Hermite polynomials and integration by parts, the integrals in
\eqref{eq:coefficients-rarafection} reduce to
\begin{equation}
\begin{aligned}
 \int_{y_1}^{y_2} \phi_i(y) \varrho(y) \dif y
    & \overset{i>0}{=} - \frac{1}{\sqrt{2\pi i}} \left[
      \phi_{i-1}(y_2)\exp\left(\frac{-y_2^2}{2}\right) -
      \phi_{i-1}(y_1) \exp\left(\frac{-y_1^2}{2}\right) \right], \\
  \int_{y_2}^{\infty} \phi_i(y) \varrho(y) \dif y
    & \overset{i>0}{=} \frac{1}{\sqrt{2\pi i}} \phi_{i-1}(y_2)
      \exp\left(\frac{-y_2^2}{2}\right), \\
  \int_{y_1}^{y_2} y \phi_i(y) \varrho(y) \dif y
    & \overset{i>0}{=} - \frac{1}{\sqrt{i}} \left[
      y \phi_{i-1}(y)\varrho(y)\Bigg|_{y_1}^{y_2} +
      \int_{y_1}^{y_2} \phi_{i-1}(y) \varrho(y) \dif y
      \right] \\
    & \overset{i>1}{=}  - \frac{1}{\sqrt{i}} \left[
      y \phi_{i-1}(y)\varrho(y)\Bigg|_{y_1}^{y_2} +
      \frac{1}{\sqrt{i-1}} \phi_{i-2}(y) \varrho(y)\Bigg|_{y_1}^{y_2}
      \right],
\end{aligned}
\end{equation}
for $i\geq2$.
So finally, the coefficients of the complete PC expansion for $i\geq2$ are
given by
\begin{equation}
\label{eq:coefficients-rarafection2}
\begin{aligned}
  u_i(x,t) = \frac{1}{\sqrt{2\pi i}} \Bigg[
  & b(y_1+1) \phi_{i-1}(y_1)\exp\left(\frac{-y_1^2}{2}\right) +
  ( b(y_2 - y_1) - 2a) \phi_{i-1}(y_2)\exp\left(\frac{-y_2^2}{2}\right) \\
  & -
  \frac{b}{\sqrt{i-1}} \phi_{i-2}(y_1)\exp\left(\frac{-y_1^2}{2}\right) +
  \frac{b}{\sqrt{i-1}} \phi_{i-2}(y_2)\exp\left(\frac{-y_2^2}{2}\right)
  \Bigg],
\end{aligned}
\end{equation}
when the normalised Hermite polynomials are applied.

\subsection{Stochastic Riemann Problem with an Initial Shock}
\label{subsec:reference-solutions_shock}

Consider the stochastic Riemann problem with an initial shock of uncertain
strength located at $x_0 \in [0,1]$, i.e.
\begin{equation}\label{shock}
\begin{aligned}
  u(x,0, \xi(\omega)) = & \begin{cases}
		  u_L = a + p( \xi(\omega)) & \text{ if } x < x_0, \\
		  u_R = -a + p(\xi(\omega)) & \text{ if } x > x_0,
                 \end{cases} \\
  u(0,t, \xi(\omega)) = & \ u_L, \ u(1,t,\xi(\omega)) = u_R, \\
  \xi \sim & \ \mathcal{N}(0,1).
\end{aligned}
\end{equation}
$\xi$, $p(\xi)$, $y = \xi(\omega)$, and $a$ are as in the previous section.
Choosing this particular Riemann problem allows a head-to-head
comparison with the numerical results obtained by Pettersson et al.
\cite{pettersson2009numerical,pettersson2015polynomial}.

By the Rankine-Hugoniot condition for a fixed $y$ the shock speed is $s = by$
and the shock location $x_s$ a given by
\begin{equation}
  x_s = x_0 + tby.
\end{equation}
Thus, for time $t \geq 0$, the analytical solution is
\begin{equation}\label{solution_shock}
  u(x,t,y) = \begin{cases}
      u_L \ & \text{ if } x < x_0 + tby, \\
      u_R \ & \text{ if } x > x_0 + tby.
  \end{cases}
\end{equation}

The coefficients of the complete PC expansion ($M \to \infty$) can now be
calculated for any given $x$ and $t$ by
\begin{equation}
  u_i(x,t) = \int_{-\infty}^{\infty} u(x,t,y) \phi_i(y) \rho(y) \dif y.
\end{equation}
For the coefficients of the initial condition at $t=0$, this simply reduces to
\begin{equation}
  u_i(x,0) = \begin{cases}
      a \delta_{i,0} + b \delta_{i,1} & \text{ if } x < x_0, \\
      -a \delta_{i,0} + b \delta_{i,1} & \text{ if } x > x_0,
    \end{cases}
\end{equation}
since the first two normalised stochastic Hermite polynomials are given by
$\phi_0(y) = 1$ and $\phi_1(y) = y$.
In our implementation, we set $u_0(x_0,0) = a$ for $x_0$ in the interior of an
element.
If $x_0$ was a (Gauß-Lobatto-Legendre) point at some element boundary, we set
$u_0(x_-,0)
= a$ for the right boundary point of the element to the left $x_-=x_0$ and
$u_0(x_+,0) = -a$ for the left boundary point of the element to the right
$x_+=x_0$.

For $t > 0$, the jump location in $y_s$ is given by
\begin{equation}
  y_s = \frac{x - x_0}{bt}
\end{equation}
with respect to $x$ and $t$, where $b$ is assumed to be positive.
Therefore, the coefficients of the complete PC expansion are given by
\begin{equation}\label{coefficients}
  u_i(x,t) = a\delta_{i,0} + b\delta_{i,1} - 2a \int_{-\infty}^{y_s}
\phi_i(y) \rho(y) \dif y
\end{equation}
for any given $x,t$ and $i\geq0$.
By the recursion relation \eqref{eq:recursion_relation}
for normalised Hermite polynomials and integration by parts, the integral in
(\ref{coefficients}) reduces to
\begin{equation}
\begin{aligned}
  \int_{-\infty}^{y_s} \phi_i(y) \rho(y) \dif y
   = - \frac{1}{\sqrt{i}} \phi_{i-1}(y) \rho(y) \Big|_{-\infty}^{y_s}
   = - \frac{1}{\sqrt{2\pi i}}
\phi_{i-1}(y_s) \exp\left(\frac{-y_s^2}{2}\right)
\end{aligned}
\end{equation}
for $i\geq1$.
So, the coefficients of the complete PC expansion for $i\geq1$ are
quite handily given by
\begin{equation}
\label{coefficients2-shock}
\begin{aligned}
  u_i(x,t) & = b\delta_{i,1} + a \frac{\sqrt{2}}{\sqrt{\pi i}}
\phi_{i-1}(y_s) \exp\left( \frac{-y_s^2}{2} \right)
\end{aligned}
\end{equation}
when the normalised Hermite polynomials are applied.
In \cite{pettersson2009numerical,pettersson2015polynomial}, Pettersson et al.
already argued the coefficients $u_i$ to be continuous in $x$ and $t$ for $x
\in [0,1]$ and $t > 0$.

Similar to the test problem discussed before, we do not focus on boundary conditions.
Using a common approach, Dirichlet boundary conditions are enforced as usual in
FV/DG methods using numerical fluxes. For sufficiently small times, this will suffice
since the solution does not interact too much with the boundary.

\subsection{Extension to a general PC method}
\label{subsec:Extension_to_a_general_PC_method}

Till now we have just considered normalised Hermite polynomials for the chaos expansion.
We can however generalise our PC method to further classical orthogonal polynomials.
We present the calculations of our test cases in the general setting of
classical orthogonal polynomials, before we focus on the Jacobi and Laguerre polynomials
as additional examples.

For a general PC method using some other basis of orthogonal polynomials
$\set{ \phi_i^{(\rho)} }$ with corresponding
weight function $\rho$, the coefficients of the complete gPC expansion are
still given by
\begin{equation}
  u_i(x,t) = a \scp{ 1 }{ \phi_i^{(\rho)} } + b \scp{ 1 }{ \phi_i^{(\rho)} }
  - 2a \int_{-\infty}^{y_s} \phi_i^{(\rho)}(y) \rho(y) \dif y
\end{equation}
for the stochastic Riemann problem with an initial shock and by
\begin{equation}
  u_i(x,t) =  a \scp{ 1 }{ \phi_i } + b \scp{ y }{ \phi_i }
  -2a \int_{-\infty}^{y_1} \phi_i^{(\rho)}(y) \rho(y) \dif y +
b y_1
\int_{y_1}^{y_2} \phi_i^{(\rho)}(y) \rho(y) \dif y
  -b \int_{y_1}^{y_2} y \phi_i^{(\rho)}(y) \rho(y) \dif y
\end{equation}
for the stochastic Riemann problem with an initial rarefaction.
So in general we wish to compute the inner products $\scp{ 1 }{ \phi_i^{(\rho)} },
\scp{ y }{ \phi_i^{(\rho)} }$ and integrals of the form
\begin{equation}
  \int_{a}^{b} \phi_i^{(\rho)}(y) \rho(y) \dif y
  \quad \text{ and } \quad
  \int_{a}^{b} y \phi_i^{(\rho)}(y) \rho(y) \dif y.
\end{equation}

For a general basis of orthogonal polynomials a similar recursion relation
to the one for the Hermite polynomials is obtained by \emph{Rodrigues' formula}
\begin{equation}
  \phi_i^{(\rho)}(y) = \frac{1}{e_i \rho(y)} \frac{\dif^{\ i}}{\dif
y^{i}}
  \left( \rho(y) [ Q(y) ]^{i} \right),
\end{equation}
which can be found in \cite[section 22.1]{abramowitz1972handbook}.
The numbers $e_i$ depend on the standardisation and $Q$ is a given quadratic
(at most) polynomial coming from the underlying differential equation.
For such a general basis, the recursion relation
\begin{equation}
\begin{aligned}
  \frac{\dif}{\dif y} \left( \phi_{i}^{(\rho)}(y) \rho(y) \right)
  & = \frac{1}{e_i} \frac{\dif^{\ i+1}}{\dif y^{i+1}} \left( \rho(y)
\left[ Q(y) \right]^{i} \right)
  = \frac{e_{i+1}}{e_i} \tilde{\omega}(y) \frac{1}{e_{i+1}}
\frac{\dif^{\ i+1}}{\dif y^{i+1}} \left( \tilde{\omega}(y)
\left[ Q(y) \right]^{i+1} \right)
  \\
  &= \frac{e_{i+1}}{e_i} \phi_{i+1}^{(\tilde{\omega})}(y)
\tilde{\omega}(y)
\end{aligned}
\end{equation}
holds, where $\tilde{\omega}(y) = \frac{\rho(y)}{Q(y)}$.
So in terms of the recursion relations before, one has
\begin{equation}
  \phi_{i}^{(\rho)}(y) \rho(y) = \frac{e_{i-1}}{e_{i}}
\frac{\dif}{\dif
y} \left( \phi_{i-1}^{(\omega Q)}(y) \rho(y) Q(y) \right)
\end{equation}
for $i\geq1$.
The integrals then can again be calculated by integration by parts.
One has
\begin{equation}\label{integral1}
  \int_{a}^{b} \phi_i^{(\rho)}(y) \rho(y) \dif y
  = \frac{e_{i-1}}{e_i} \phi_{i-1}^{(\omega Q)}(y) \rho(y) Q(y)
\bigg|_a^b
\end{equation}
for $i\geq1$, and
\begin{equation}\label{integral2}
  \int_{a}^{b} y \phi_i^{(\rho)}(y) \rho(y) \dif y
  = \frac{1}{e_i} \left[ e_{i-1} \phi_{i-1}^{(\omega Q)}(y) - e_{i-2}
\phi_{i-2}^{(\omega Q^2)}(y) Q(y) \right] \rho(y) Q(y) \bigg|_a^b
\end{equation}
for $i\geq2$.
Note that for the Hermite polynomials (not normalised) one has
\begin{equation}
  e_i = (-1)^i, \ \rho(y) = \frac{1}{\sqrt{2 \pi}} e^{\frac{-y^2}{2}}, \
Q(y) = 1, \ \rho(y)Q(y) = \rho(y).
\end{equation}

\subsection{Jacobi polynomials}

For the Jacobi polynomials $P_n^{(\alpha,\beta)},\; \alpha, \beta\in
(-1,\infty)$, the weight function and other parameters are given by
\begin{equation}
\begin{aligned}
  e_i = (-2)^i i!, & \quad \rho^{(\alpha,\beta)}(y) =
\mathbf{1}_{[0,1]}(y) (1-y)^{\alpha} (1+y)^{\beta}, \\
  Q(y) = 1-y^2, & \quad \rho^{(\alpha,\beta)}(y)Q(y) =
\rho^{(\alpha +1,\beta +1)}(y).
\end{aligned}
\end{equation}
Their orthogonality property then reads
\begin{equation}
\begin{aligned}
  \scp{ P_n^{(\alpha,\beta)} }{ P_i^{(\alpha,\beta)} }
  & = \int_{-1}^{1} P_n^{(\alpha,\beta)}(y) P_i^{(\alpha,\beta)}(y)
(1-y)^\alpha (1+y)^\beta \dif y \\
  & = \frac{ 2^{\alpha+\beta+1} }{2n+\alpha+\beta+1}
\frac{\Gamma(n+\alpha+1)\Gamma(n+\beta+1)}{\Gamma(n+\alpha+\beta+1) n!}
\delta_{i,n}.
\end{aligned}
\end{equation}
Since the first two Jacobi polynomials are given by
\begin{equation}
  P_0^{(\alpha,\beta)}(y) = 1 \quad \text{and} \quad
  P_1^{(\alpha,\beta)}(y) = \frac{1}{2}(\alpha-\beta) +
\frac{1}{2}(\alpha+\beta+2)y,
\end{equation}
the inner products $\scp{ 1 }{ \phi_i }, \scp{ y }{ \phi_i }$ read
\begin{equation}
\begin{aligned}
  \scp{ 1 }{ P_i^{(\alpha,\beta)} } & = \frac{ 2^{\alpha+\beta+1}
}{\alpha+\beta+1}\frac{\Gamma(\alpha+1)\Gamma(\beta+1)}{
\Gamma(\alpha+\beta+1)} \delta_{i,0}, \\
  \scp{ y }{ P_i^{(\alpha,\beta)} } & =
  -\frac{2}{\alpha+\beta+2} \scp{ P_1^{(\alpha,\beta)} }{ P_i^{(\alpha,\beta)} }
  + \frac{\alpha-\beta}{\alpha+\beta+2}
    \scp{ P_0^{(\alpha,\beta)} }{ P_i^{(\alpha,\beta)} }
  \\
  & = -\frac{2^{\alpha+\beta+2}}{(\alpha+\beta+2)(\alpha+\beta+3)}
      \frac{\Gamma(\alpha+2)\Gamma(\beta+2)}{\Gamma(\alpha+\beta+2)}
\delta_{i,1} \\
  & \quad +
\frac{(\alpha-\beta)2^{\alpha+\beta+1}}{(\alpha+\beta+2)(\alpha+\beta+1)}
\frac{\Gamma(\alpha+1)\Gamma(\beta+1)}{\Gamma(\alpha+\beta+1)} \delta_{i,0}.
\end{aligned}
\end{equation}

By (\ref{integral1}) and (\ref{integral2}), the coefficients of the complete
gPC expansion for the stochastic Riemann
problem with an initial shock and $i\geq1$ are given by
\begin{equation}
\begin{aligned}
  u_i(x,t) =
   b \frac{2^{\alpha+\beta+1}
\Gamma(\alpha+2)\Gamma(\beta+2)}{
(\alpha+\beta+2)(\alpha+\beta+3)\Gamma(\alpha+\beta+2)} \delta_{i,1}
+ \frac{a}{i} P_{i-1}^{(\alpha+1,\beta+1)}(y_s)
\rho^{(\alpha+1,\beta+1)}(y_s).
\end{aligned}
\end{equation}
The coefficients of the complete gPC expansion for the stochastic Riemann
problem with an initial rarefaction and $i\geq2$ are given by
\begin{equation}
\begin{aligned}
  u_i(x,t)
  =&
  \frac{1}{2i} \left[ 2a + b (y_1 - 1) \right]
    P_{i-1}^{(\alpha+1,\beta+1)}(y_1) \rho^{(\alpha+1,\beta+1)}(y_1)
  + \frac{1}{2i} b (1 - y_1)
    P_{i-1}^{(\alpha+1,\beta+1)}(y_2) \rho^{(\alpha+1,\beta+1)}(y_2)
    \\&
  + \frac{b}{4i(i-1)}
    P_{i-2}^{(\alpha+2,\beta+2)}(y_1) \rho^{(\alpha+2,\beta+2)}(y_1)
  - \frac{b}{4i(i-1)}
    P_{i-2}^{(\alpha+2,\beta+2)}(y_2) \rho^{(\alpha+2,\beta+2)}(y_2).
\end{aligned}
\end{equation}

\subsection{Laguerre polynomials}
\label{Appendix:Reference}

For the Laguerre polynomials $L_n^{(\alpha)},\; \alpha \in
(-1,\infty)$, the weight function and other parameters are given by
\begin{equation}
  e_i = i!, \ \rho^{(\alpha)}(y) = \mathbf{1}_{[0,\infty)}(y)
y^{\alpha} e^{-y}, \ Q(y) = y, \ \rho^{(\alpha)}(y)Q(y) =
\rho^{(\alpha +1)}(y).
\end{equation}
Their orthogonality property then reads
\begin{equation}
  \scp{ L_n^{(\alpha)} }{ L_i^{(\alpha)} }
  = \int_{0}^{\infty } L_{n}^{(\alpha)}(y) L_{i}^{(\alpha
)}(y) y^{\alpha}e^{-y} \dif y
  = \frac{\Gamma(n+\alpha+1)}{n!} \delta_{i,n}.
\end{equation}
Since the first two Laguerre polynomials are given by
\begin{equation}
  L_0^{(\alpha)}(y) = 1 \quad \text{and} \quad
  L_1^{(\alpha)}(y) = -y + 1,
\end{equation}
the inner products $\scp{ 1 }{ \phi_i }, \scp{ y }{ \phi_i }$ read
\begin{equation}
\begin{aligned}
  \scp{ 1 }{ L_i^{(\alpha)} } & = \Gamma(\alpha+1) \delta_{i,0},
  \\
  \scp{ y }{ L_i^{(\alpha)} } & = -\Gamma(\alpha+2)\delta_{i,1} +
\Gamma(\alpha+1)\delta_{i,0}.
\end{aligned}
\end{equation}

By (\ref{integral1}) and (\ref{integral2}), the coefficients of the complete
gPC expansion for the stochastic Riemann
problem with an initial shock and $i\geq1$ are given by
\begin{equation}
  u_i(x,t) = - b \Gamma(\alpha+2)
\delta_{i,1} + \frac{2a}{i}
L_{i-1}^{(\alpha+1)}(y_s) \rho^{(\alpha)}(y_s).
\end{equation}
The coefficients of the complete gPC expansion for the stochastic Riemann
problem with an initial rarefaction and $i\geq2$ are given by
\begin{equation}
\begin{aligned}
  u_i(x,t) = & \frac{1}{i} \left[ b(1-y_1) - 2a \right]
L_{i-1}^{(\alpha+1)}(y_1) \rho^{(\alpha+1)}(y_1)
  + \frac{1}{i} \frac{1}{i} b ( y_1 -1 )
L_{i-1}^{(\alpha+1)}(y_2) \rho^{(\alpha+1)}(y_2) \\
  & - \frac{b}{i(i-1)}
L_{i-2}^{(\alpha+2)}(y_1) \rho^{(\alpha+2)}(y_1)
  + \frac{b}{i(i-1)}
L_{i-2}^{(\alpha+2)}(y_2) \rho^{(\alpha+2)}(y_2).
\end{aligned}
\end{equation}

\section{Numerical tests}
\label{sec:numerical-tests}

In order to quantify the behaviour of the numerical methods, two
different test cases are considered in this section:
Burgers' equation with an initial
rarefaction and an initial shock, both with an uncertain perturbation.
The test case of an initial shock was also investigated by Pettersson et al.
\cite{pettersson2009numerical,pettersson2015polynomial} in the context of SBP FD
methods and thus allows a comparison of the numerical results.
The first test case of an initial rarefaction will demonstrate the capability
of the SBP CPR method to capture expansion waves accurately.

For the truncation of the polynomial chaos expansion,
special focus will be given to the truncation for $M=3$,
i.e. the four dimensional system \eqref{BurgerGalerkinUQ2} with matrix $A(u)$
given by
 \begin{equation}\label{eq:M=3matrix}
    A(u)
    =
    \begin{pmatrix}
      u_{0} & u_{1} & u_{2} & u_{3}
      \\
      u_{1} & u_{0} + \sqrt{2} u_{2} & \sqrt{2} u_{1} + \sqrt{3} u_{3} & \sqrt{3} u_{2}
      \\
      u_{2} & \sqrt{2} u_{1} + \sqrt{3} u_{3} & u_{0}
      + 2 \sqrt{2} u_{2} & \sqrt{3} u_{1} + 3 \sqrt{2} u_{3}
      \\
      u_{3} & \sqrt{3} u_{2} & \sqrt{3} u_{1} + 3 \sqrt{2} u_{3} & u_{0} + 3 \sqrt{2} u_{2}
    \end{pmatrix}.
  \end{equation}
Note that the reference solutions $\uref$ in the last section were
derived as solutions of the infinite order systems for $M \to \infty$.
They are smooth functions.
However, the numerical solutions are calculated by solving numerically the truncated systems
which are hyperbolic.  The numerical solutions may
contain discontinuities and, therefore, differ from the ones of
the infinite order system. See also the following Remark \ref{discontinuities_problems}
 in this context.

While the discretisation in space is done by different methods like SBP CPR and FV,
the discretisation in time is always done by the strong-stability
preserving third order explicit Runge-Kutta method using three stages,
\emph{SSPRK(3,3)}, given by Gottlieb and Shu \cite{gottlieb1998total}.
At least for linear problems
\begin{equation}
  \dot{u} = Lu,
\end{equation}
this method was shown to be strongly stable for semibounded operators $L$ under
certain time step restriction, \cite{tadmor2002semidiscrete,ranocha2018l2}.
Note that stable discretisations in space correspond to such
semibounded operators.

If nothing else is said, we chose the time step to be
\begin{equation}
  \Delta t =
    C \cdot \frac{1}{N\cdot(2p+1)^2\cdot M}
    \quad \text{with} \quad
    C = \frac{1}{2}
\end{equation}
in the numerical tests.

For comparison, we apply also the SBP FD method of Pettersson et al..
They offered a Matlab code very well prepared in \cite{pettersson2015polynomial}.
In this code, the classical RK4 method is used for time integration.

\begin{re}\label{discontinuities_problems}
As described in \cite{poette2009uncertainty}, discontinuities in
the $x$-space (here: initial conditions) lead to $M$-convergence problems.
We will also note this behaviour in our numerical tests. A detailed analysis
of convergence with respect to $M$, $p$ or $N$ for smooth solutions can
be found in \cite{giesselmann2017posteriori, 2018arXiv180510177M} but is beyond
the scope of this paper.
However, for the first test case (initial rarefaction),
we will provide a convergence study for the truncated system where the reference
solution is obtained by a high-resolution numerical solution and also an analysis
using the reference solution which was calculated in section \ref{sec:reference-solutions}.
Though, the test with the initial shock has several more issues and we will consider
these in detail.
\end{re}

\subsection{Initial Rarefaction}

As a first example, the stochastic Riemann problem (\ref{rarefaction}) with
an initial rarefaction will be covered.
The initial condition is
\begin{equation}
\begin{aligned}
  u(x,0, \xi(\omega)) = & \begin{cases}
		  u_L = -a + p( \xi(\omega)) & \text{ if } x < x_0, \\
		  u_R = a + p( \xi(\omega)) & \text{ if } x > x_0,
                 \end{cases} \\
  \xi \sim & \ \mathcal{N}(0,1),
\end{aligned}
\end{equation}
with uncertain height located at $x_0 \in [0,1]$ and
$p( \xi(\omega)) = b \xi(\omega)$ depends linearly on $\xi(\omega)$.
Here, the parameters $a=1$, $b=0.2$ and $x_0 = 0.5$ are chosen.
As described in section~\ref{subsec:reference-solutions_rarefaction}, outflow boundary
conditions are used.
Note that the expectation $\E[u]$ and variance $\Var(u)$ of the initial
condition were derived in the previous section
\ref{subsec:reference-solutions_rarefaction} as well.
There, also the analytical solution $\uref$ for the infinite order system,
i.e. $M \to \infty$, is given by \eqref{eq:coefficients-rarafection} and
\eqref{eq:coefficients-rarafection2}.

Note that all numerical solutions are obtained for a fixed $M<\infty$ and thus provide
approximations to a truncated system.
The error between a numerical solution for the truncated polynomial chaos expansion and the
reference solution is investigated by measuring the (discrete)
$\L^2$ norm error of the expected value and the variance.
On the $n$-th element $D_n$, we have
\begin{align}
  \norm{\varepsilon_{\E}}_{n}^2
    = \int_{D_n} \left| \E[\unum] - \E[\uref] \right|^2 \dif x, \qquad
  \norm{\varepsilon_{\Var}}_{n}^2
    = \int_{D_n} \left| \Var(\unum) - \Var(\uref) \right|^2 \dif x.
\end{align}
To get the global errors, we sum up the local errors, i.e.
\begin{align}
  \norm{\varepsilon_{\E}}^2
    = \sum_{n=1}^N \norm{\varepsilon_{\E}}_{n}^2, \qquad
  \norm{\varepsilon_{\Var}}^2
    = \sum_{n=1}^N \norm{\varepsilon_{\Var}}_{n}^2.
\end{align}
Tables \ref{tab:Cauchy_N-ref_rarefaction} and \ref{tab:Cauchy_p-ref_rarefaction} demonstrate
convergence of the numerical solution to a reference solution of the truncated system for $M=3$ when
the spacial discretisation is refined, either in the number of elements $N$ or the polynomial degree
$p$.
Since no analytical solution is known for the truncated system, the reference solution is obtained
by a high-resolution numerical solution, i.e. $N=12800$ in Table \ref{tab:Cauchy_N-ref_rarefaction}
and $p=5$ in Table \ref{tab:Cauchy_p-ref_rarefaction}.

\begin{table}[!htb]
    \centering
    \begin{tabular}{ r | c | c | c | c }
	& \multicolumn{2}{c|}{Expected Value $\E[u]$} & \multicolumn{2}{c}{Variance $\Var(u)$} \\
        $N$ & $\norm{\E[u_{N}]-\E[u_{N/2}]}$ & $\norm{\E[u_{12800}]-\E[u_{N}]}$
	    & $\norm{\Var(u_{N})-\Var(u_{N/2})}$ & $\norm{\Var(u_{12800})-\Var(u_{N})}$ \\
        \hline \hline

          100 &  	& 2.6e-02 &    	    & 3.6e-03 \\
          200 & 1.1e-02 & 1.6e-02 & 1.3e-03 & 2.4e-03 \\
          400 & 7.0e-03 & 9.4e-03 & 9.5e-04 & 1.5e-03 \\
          800 & 4.2e-03 & 5.3e-03 & 6.5e-04 & 9.2e-04 \\
         1600 & 2.5e-03 & 2.6e-03 & 4.2e-04 & 5.1e-04 \\
         3200 & 1.5e-03 & 1.3e-03 & 2.6e-04 & 2.5e-04 \\
         6400 & 8.8e-04 & 5.0e-04 & 1.5e-04 & 9.4e-05 \\
        12800 & 5.0e-04 & 0   	  & 9.4e-05 & 0 \\
	\hline

    \end{tabular}
    \caption{
        Initial Rarefaction.
        Cauchy differences and errors for the expected value $\E[u]$ and variance 
$\Var(u)$ of the
numerical solutions for $M=3$, $p=0$, and increasing $N$.
    }
    \label{tab:Cauchy_N-ref_rarefaction}
\end{table}

\begin{table}[!htb]
    \centering
    \begin{tabular}{ r | c | c | c | c }
	& \multicolumn{2}{c|}{Expected Value $\E[u]$} & \multicolumn{2}{c}{Variance 
$\Var(u)$} \\
        $p$ & $\norm{\E[u_{p}]-\E[u_{p-1}]}$ & $\norm{\E[u_{5}]-\E[u_{p}]}$
	    & $\norm{\Var(u_{p})-\Var(u_{p-1})}$ & $\norm{\Var(u_{5})-\Var(u_{p})}$ \\
        \hline \hline

	  0 & 	      & 2.3e-02 & 	  & 6.2e-03 \\
          1 & 2.2e-02 & 2.2e-03 & 3.5e-03 & 3.9e-04 \\
          2 & 1.6e-03 & 8.6e-04 & 1.9e-04 & 1.7e-04 \\
          3 & 6.2e-04 & 3.8e-04 & 7.8e-05 & 9.0e-05 \\
          4 & 3.1e-04 & 1.9e-04 & 4.5e-05 & 5.9e-05 \\
          5 & 1.9e-04 & 0 	& 3.4e-05 & 0	    \\
	\hline

    \end{tabular}
    \caption{
        Initial Rarefaction.
        Cauchy differences and errors for the expected value $\E[u]$ and variance 
$\Var(u)$ of the
numerical solutions for $M=3$ $N=200$, and increasing $p$.
    }
    \label{tab:Cauchy_p-ref_rarefaction}
\end{table}

Tables \ref{tab:errors_N-ref_E_rarefaction} and Table \ref{tab:errors_N-ref_Var_rarefaction} list
the errors between the numerical solution for the truncated polynomial chaos expansion and the
analytical solution for polynomial degree $p=0$ and an increasing number of elements $N$ as well as
an increasing order $M$ in the polynomial chaos expansion.
Table \ref{tab:errors_N-ref_E_rarefaction} shows the errors for the expected value.
Table \ref{tab:errors_N-ref_Var_rarefaction} shows the errors for the variance.

\begin{table}[!htb]
    \centering
    \begin{tabular}{ c | c | c | c | c | c | c | c | c }
            & \multicolumn{8}{c}{$\norm{\varepsilon_{\E}}$} \\
        $N$ & $M=1$ & $M=2$ & $M=3$ & $M=4$ & $M=5$ & $M=6$ & $M=7$ & $M=8$ \\
        \hline \hline

          100 & 1.2e-05 & 3.7e-05 & 5.1e-05 & 5.9e-05 & 6.7e-05 & 7.5e-05 & 8.2e-05 & 8.8e-05 \\
          200 & 2.3e-08 & 5.7e-07 & 2.1e-06 & 3.3e-06 & 3.7e-06 & 4.1e-06 & 4.6e-06 & 5.0e-06 \\
          400 & \textbf{6.1e-10} & \textbf{3.2e-10} & 1.3e-08 & 9.7e-08 & 2.1e-07 & 2.5e-07 &
2.5e-07 & 2.8e-07 \\
          800 & 4.0e-10 & 4.0e-10 & \textbf{4.0e-10} & \textbf{1.7e-10} & 4.0e-09 & 1.7e-08 &
2.6e-08 & 2.3e-08 \\
         1600 & 2.7e-10 & 2.7e-10 & 2.7e-10 & 2.7e-10 & \textbf{2.7e-10} & \textbf{1.7e-11} &
2.1e-09 & 4.2e-09 \\
         3200 & 1.9e-10 & 1.9e-10 & 1.9e-10 & 1.9e-10 & 1.9e-10 & \textbf{1.9e-10} &
\textbf{1.3e-10} & 8.8e-10 \\
         6400 & 1.3e-10 & 1.3e-10 & 1.3e-10 & 1.3e-10 & 1.3e-10 & 1.3e-10 & \textbf{1.3e-10} &
\textbf{8.7e-11} \\
	\hline

    \end{tabular}
    \caption{
        Initial Rarefaction.
        Errors for the expected value $\E[u]$ of the numerical solutions for
$p=0$ and increasing $M$ and $N$.
    }
    \label{tab:errors_N-ref_E_rarefaction}
\end{table}

\begin{table}[!htb]
    \centering
    \begin{tabular}{ c | c | c | c | c | c | c | c | c }
            & \multicolumn{8}{c}{$\norm{\varepsilon_{\Var}}$} \\
        $N$ & $M=1$ & $M=2$ & $M=3$ & $M=4$ & $M=5$ & $M=6$ & $M=7$ & $M=8$  \\
        \hline \hline

          100 & 4.9e-06 & 2.5e-05 & 3.9e-05 & 4.4e-05 & 4.9e-05 & 5.4e-05 & 5.8e-05 & 6.1e-05 \\
          200 & 7.6e-09 & 3.9e-07 & 2.0e-06 & 3.5e-06 & 4.0e-06 & 4.4e-06 & 4.8e-06 & 5.2e-06 \\
          400 & \textbf{1.2e-09} & \textbf{1.0e-09} & 1.1e-08 & 1.1e-07 & 2.8e-07 & 3.5e-07 &
3.5e-07 & 3.7e-07 \\
          800 & 8.2e-10 & 8.2e-10 & \textbf{8.2e-10} & \textbf{5.6e-10} & 5.0e-09 & 2.6e-08 &
4.2e-08 & 3.8e-08 \\
         1600 & 5.6e-10 & 5.6e-10 & 5.6e-10 & 5.6e-10 & \textbf{5.6e-10} & \textbf{1.2e-10} &
3.5e-09 & 7.4e-09 \\
         3200 & 3.9e-10 & 3.9e-10 & 3.9e-10 & 3.9e-10 & 3.9e-10 & \textbf{3.9e-10} &
\textbf{3.0e-10} & 8.8e-10 \\
         6400 & 2.7e-10 & 2.7e-10 & 2.7e-10 & 2.7e-10 & 2.7e-10 & 2.7e-10 & \textbf{2.7e-10} &
\textbf{1.9e-10} \\
	\hline

    \end{tabular}
    \caption{
        Initial Rarefaction.
        Errors for the variance $\Var(u)$ of the numerical solutions for $p=0$ and increasing $M$
and $N$.
    }
    \label{tab:errors_N-ref_Var_rarefaction}
\end{table}

Note that for an unresolved truncated system (not sufficiently great $N$), the error increases
when the order $M$ in the polynomial chaos expansion is increased.
For instance, see $N=100$ and $N=200$.
Only for $N=400$ the error decreases when going over from $M=1$ to $M=2$.
Yet, afterwards the error increases again.
In our numerical tests, we observed a diagonal limit to be preferable.
In particular, the order $M$ in the polynomial chaos expansion and the spacial resolution (e.g. the
number of elements $N$) should be increased simultaneously.
This is demonstrated when going over from $M=3$ to $M=4$ for $N=800$, from $M=5$ to $M=6$ for
$N=1600$, from $M=6$ to $M=7$ for $N=3200$, and from $M=7$ to $M=8$ for $N=6400$.

A similar behaviour is observed when the spacial resolution is enhanced by increasing the
polynomial degree $p$.
This is demonstrated in Table \ref{tab:errors_p-ref_rarefaction} for the expected value and for the
variance.

\begin{table}[!htb]
\resizebox{\columnwidth}{!}{%
    \centering

    \begin{tabular}{ c | c | c | c | c | c | c | c | c | c | c | c | c }
            & \multicolumn{6}{c|}{$\norm{\varepsilon_{\E}}$}
            & \multicolumn{6}{c}{$\norm{\varepsilon_{\Var}}$} \\
        $p$ & $M=1$ & $M=2$ & $M=3$ & $M=4$ & $M=5$ & $M=6$
	    & $M=1$ & $M=2$ & $M=3$ & $M=4$ & $M=5$ & $M=6$ \\
        \hline \hline

          0 & 2.3e-08 & 5.7e-07 & 2.1e-06 & 3.3e-06 & 3.7e-06 & 4.1e-06
	    & 7.6e-09 & 3.9e-07 & 2.0e-06 & 3.5e-06 & 4.0e-06 & 4.4e-06 \\
          1 & 1.0e-09 & 1.0e-09 & \textbf{1.0e-09} & \textbf{8.6e-10} & \textbf{4.2e-10} & 6.0e-09
	    & 2.0e-09 & 2.0e-09 & \textbf{2.0e-09} & \textbf{1.8e-09} & \textbf{4.4e-10} & 8.5e-09
\\
          2 & 1.0e-09 & 1.0e-09 & 1.0e-09 & \textbf{1.0e-09} & \textbf{9.7e-10} & 2.0e-09
	    & 2.0e-09 & 2.0e-09 & 2.0e-09 & 2.0e-09 & 2.0e-09 & 3.5e-09 \\
          3 & 1.0e-09 & 1.0e-09 & 1.0e-09 & 1.0e-09 & \textbf{1.0e-09} & \textbf{9.3e-10}
	    & 2.0e-09 & 2.0e-09 & 2.0e-09 & 2.0e-09 & \textbf{2.0e-09} & \textbf{1.8e-09} \\
	\hline

    \end{tabular}

}
    \caption{
        Initial Rarefaction.
        Errors for the expected value $\E[u]$ of the numerical solutions for
$N=200$ and increasing $M$ and $p$.
    }
    \label{tab:errors_p-ref_rarefaction}
\end{table}

Again, $M$ and $p$ should be increased simultaneously.
Thus, the error is reduced when going over from $M=3$ to $M=4$ (and $M=5$) for $p=1$, from $M=4$ to
$M=5$ for $p=2$, and from $M=5$ to $M=6$ for $p=3$.
Yet, the error increases if the truncated system is not sufficiently resolved anymore for
increasing $M$.

Figure \ref{fig:rarefaction} illustrates the expected value $\E[u]$ as well as the 
variance
$\Var(u)$ for the reference
solution and the numerical solution for different parameters $M$ and $p$.

\begin{figure}[!htp]
\centering
\captionsetup[subfigure]{aboveskip=-8pt,belowskip=-1pt}
  \begin{subfigure}[b]{0.495\textwidth}
    \includegraphics[width=\textwidth]{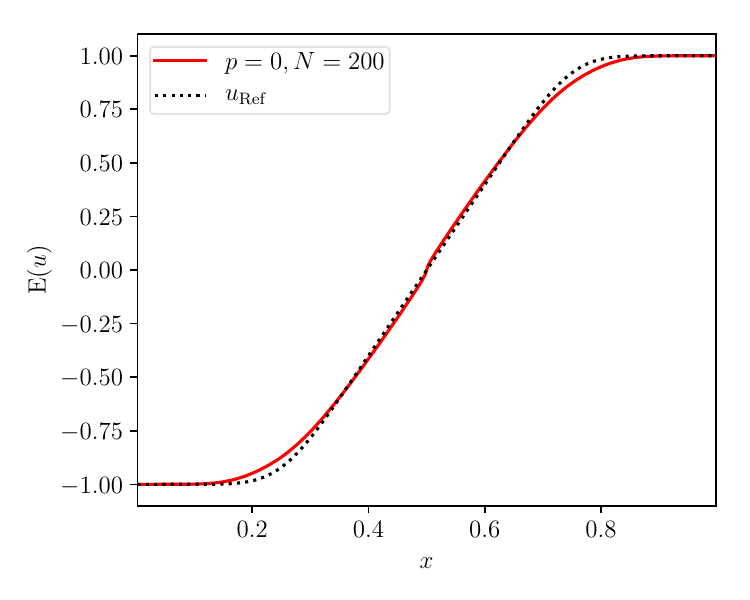}
    \caption{$M=1$, expectation $\E[u]$.}
    \label{fig:rarefaction_M_1__p_0__N_200__Eu}
  \end{subfigure}%
  ~
  \begin{subfigure}[b]{0.495\textwidth}
    \includegraphics[width=\textwidth]{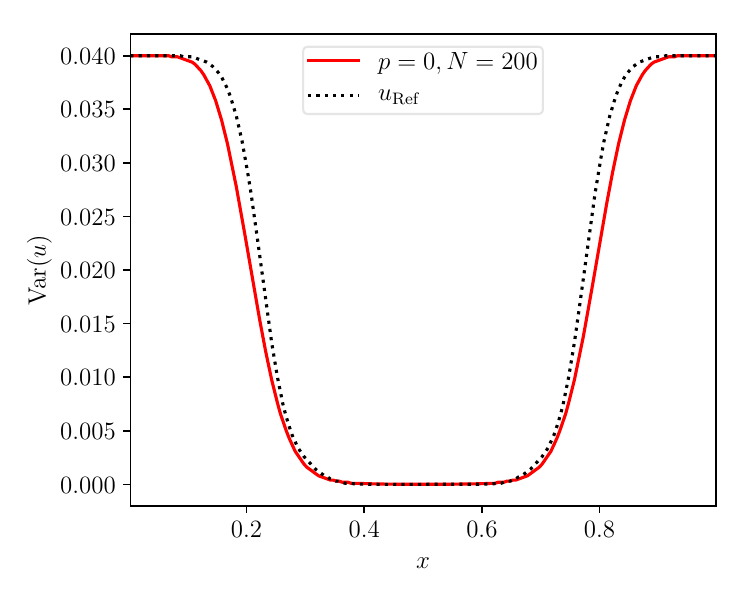}
    \caption{$M=1$, variance $\Var(u)$.}
    \label{fig:rarefaction_M_1__p_0__N_200__Varu}
  \end{subfigure}%
  \\
  \begin{subfigure}[b]{0.495\textwidth}
    \includegraphics[width=\textwidth]{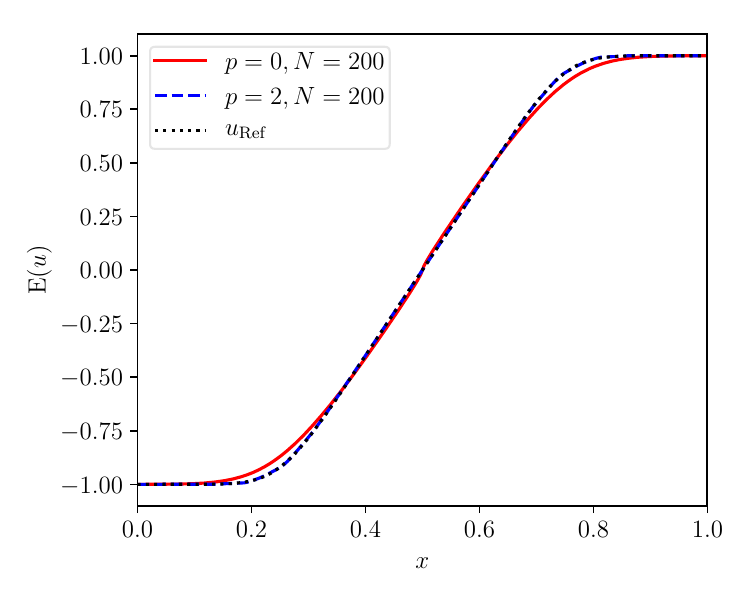}
    \caption{$M=6$, expectation $\E[u]$.}
    \label{fig:rarefaction_M_6__p_0__N_200_1600__Eu}
  \end{subfigure}%
  ~
  \begin{subfigure}[b]{0.495\textwidth}
    \includegraphics[width=\textwidth]{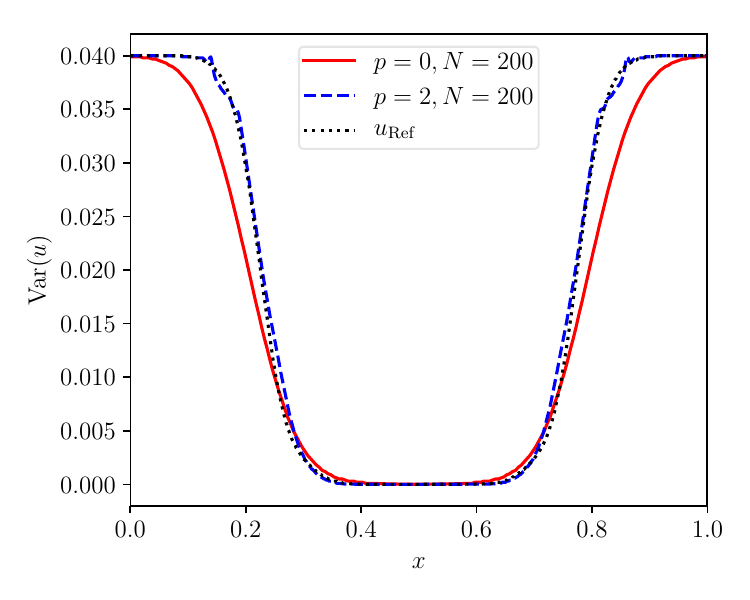}
    \caption{$M=6$, variance $\Var(u)$.}
    \label{fig:rarefaction_M_6__p_0__N_200_1600__Varu}
  \end{subfigure}%
  \caption{Rarefaction of uncertain height as initial condition for an expansion
  wave in the solution at time $t=0.25$.}
  \label{fig:rarefaction}
\end{figure}

In particular, we note that going from $M=1$ to $M=6$ for $p=0$ and $N=200$ the
numerical solution does not improve, see Figure \ref{fig:rarefaction_M_1__p_0__N_200__Eu}
compared to Figure \ref{fig:rarefaction_M_6__p_0__N_200_1600__Eu} (expected value)
and Figure \ref{fig:rarefaction_M_1__p_0__N_200__Varu} compared to
Figure \ref{fig:rarefaction_M_6__p_0__N_200_1600__Varu} (variance).
Yet, the numerical solution improves when we simultaneously increase the polynomial
degree from $p=0$ to $p=2$.

\subsection{Initial Shock}
\label{subsec:inital_shock}

Finally, the stochastic Riemann problem (\ref{shock}) with an initial
shock will be covered.
Here, the initial condition is
\begin{equation}\label{eq:Riemann_shock}
\begin{aligned}
  u(x,0, \xi(\omega)) = & \begin{cases}
		  u_L = a + p( \xi(\omega)) & \text{ if } x < x_0, \\
		  u_R = -a + p( \xi(\omega)) & \text{ if } x > x_0,
                 \end{cases} \\
   u(0,t,\xi(\omega)) = & \ u_L, \ u(1,t,\xi(\omega)) = u_R, \\
  \xi \sim & \ \mathcal{N}(0,1),
\end{aligned}
\end{equation}
with uncertain strength located at $x_0 \in [0,1]$ and function
$p( \xi(\omega)) = b\xi(\omega)$.
For the numerical tests, the parameters $a=1$, $b=0.2$ and $x_0 = 0.5$ are
considered.
As described in section~\ref{subsec:reference-solutions_shock}, Dirichlet boundary conditions
are used.
This problem has also been treated by Pettersson et al.
\cite{pettersson2009numerical,pettersson2015polynomial} and thus allows a
comparison of the numerical results.
Note that the reference solution $\uref$ was derived in the previous section
\ref{subsec:reference-solutions_shock} as the analytical solution of the system
(\ref{BurgerGalerkinUQ}) of infinite order.
For the numerical computations, however, the polynomial chaos expansion is
truncated and the system (\ref{BurgerGalerkinUQ}) is solved numerically.

Figure \ref{fig:shock_M_1_2_3__p_0__N_100_3200} displays the expected value $\E[u]$
and the variance $\Var(u)$ of the reference solution and different numerical
solutions for $p=0$ at time $t=0.5$. The numerical solutions are computed for
$N=100,3200$ and $M=1,2,3$.

\begin{figure}[!htp]
\centering
\captionsetup[subfigure]{aboveskip=-8pt,belowskip=-1pt}
  \begin{subfigure}[b]{0.468\textwidth}
    \includegraphics[width=\textwidth]{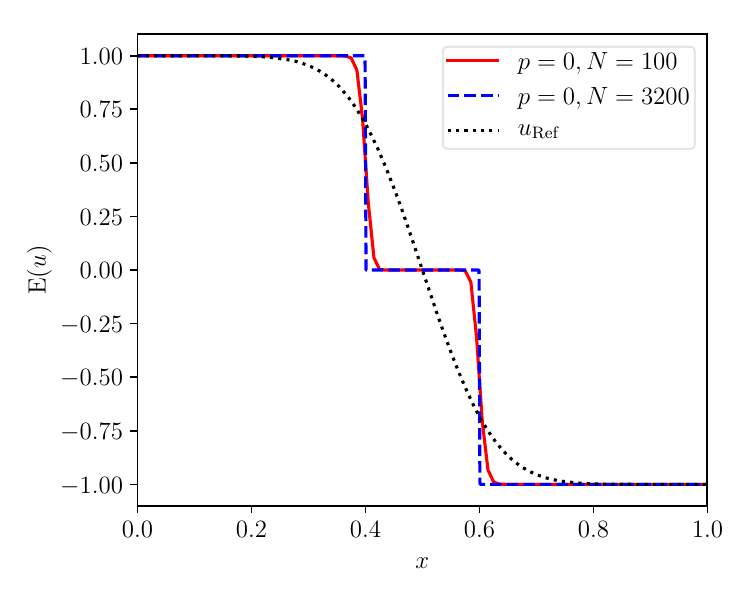}
    \caption{$M=1$, expectation $\E[u]$.}
    \label{fig:shock_M_1__p_0__N_100_3200__Eu}
  \end{subfigure}%
  ~
  \begin{subfigure}[b]{0.468\textwidth}
    \includegraphics[width=\textwidth]{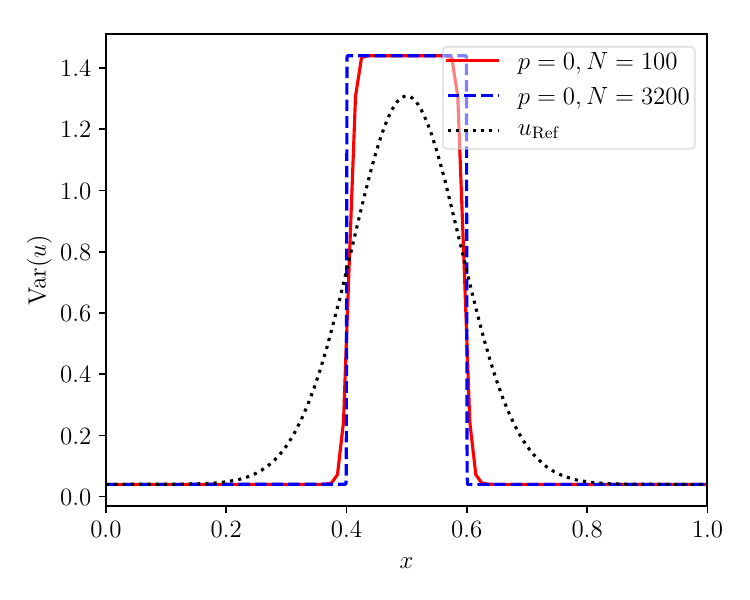}
    \caption{$M=1$, variance $\Var(u)$.}
    \label{fig:shock_M_1__p_0__N_100_3200__Varu}
  \end{subfigure}%
  \\
  \begin{subfigure}[b]{0.468\textwidth}
    \includegraphics[width=\textwidth]{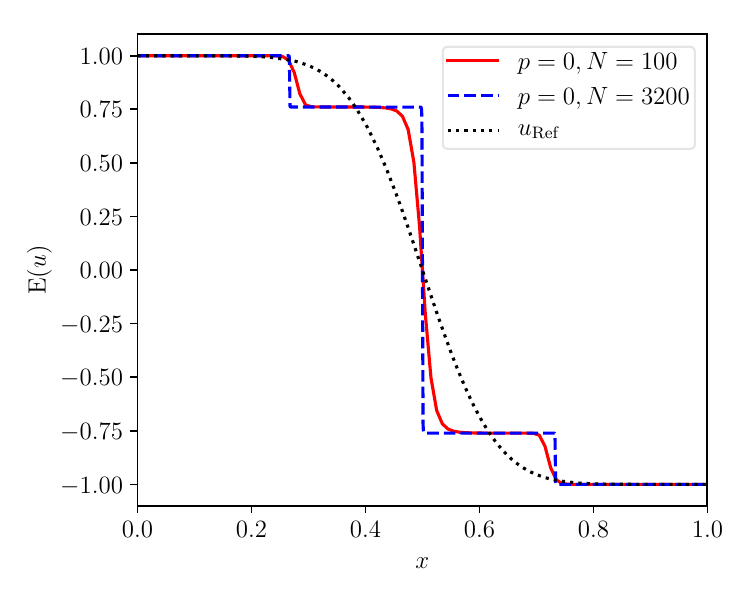}
    \caption{$M=2$, expectation $\E[u]$.}
    \label{fig:shock_M_2__p_0__N_100_3200__Eu}
  \end{subfigure}%
  ~
  \begin{subfigure}[b]{0.468\textwidth}
    \includegraphics[width=\textwidth]{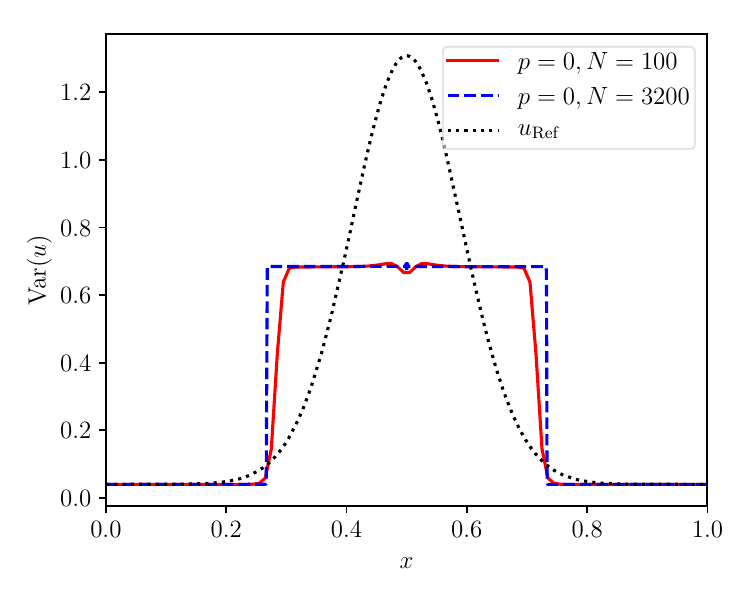}
    \caption{$M=2$, variance $\Var(u)$.}
    \label{fig:shock_M_2__p_0__N_100_3200__Varu}
  \end{subfigure}%
  \\
  \begin{subfigure}[b]{0.468\textwidth}
    \includegraphics[width=\textwidth]{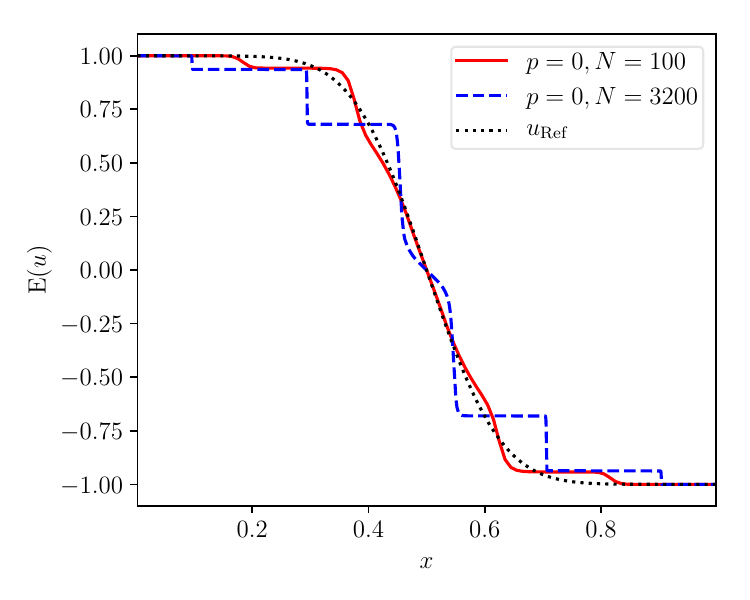}
    \caption{$M=3$, expectation $\E[u]$.}
    \label{fig:shock_M_3__p_0__N_100_3200__Eu}
  \end{subfigure}%
  ~
  \begin{subfigure}[b]{0.468\textwidth}
    \includegraphics[width=\textwidth]{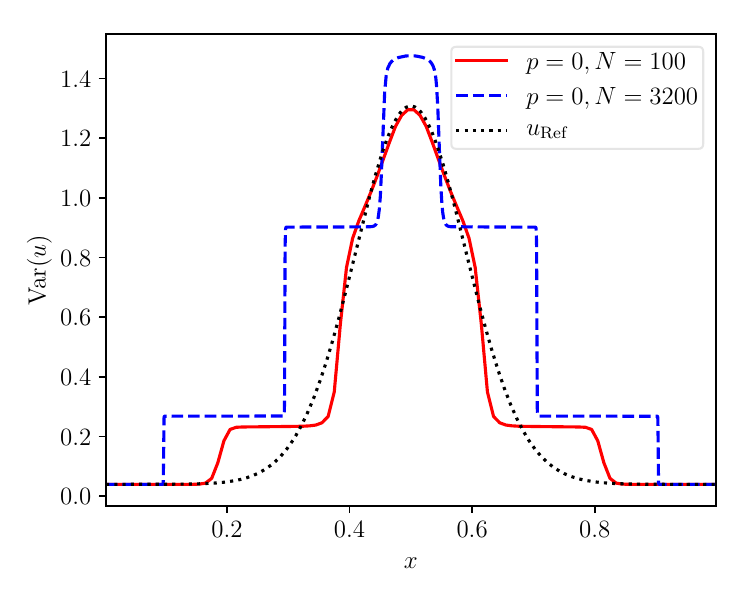}
    \caption{$M=3$, variance $\Var(u)$.}
    \label{fig:shock_M_3__p_0__N_100_3200__Varu}
  \end{subfigure}%
  \caption{Solution at time $t=0.5$ for a shock of uncertain height as initial condition.}
  \label{fig:shock_M_1_2_3__p_0__N_100_3200}
\end{figure}

For this test case distinct structures can be observed for the truncated systems
of increasing order $M$.
While these structures are blurred for the numerical solutions using only $N=100$
elements, they are better resolved for the numerical solution using $N=3200$ elements
(compared to a grid convergence study using entropy stable finite volume schemes
not presented here in detail).
Thus, it should be stressed that an increasing number of elements, i.e. an enhanced spacial
resolution, does not result in more accurate numerical solutions compared to the reference
solution of the system (\ref{BurgerGalerkinUQ}) of infinite order.
A similar behaviour was already observed in \cite{pettersson2009numerical} for
this test case: Instead of increasing the polynomial chaos order $M$, one
should increase the dissipation that is added to the scheme to smooth
the numerical solution. A scheme with a lot of dissipation like the FV scheme with $N=100$
will lead to a solution that appears to be better in comparison with the reference
solution for $M \to \infty$. That the dissipation in this test case is very
important will also be seen in the following.

Surprisingly, for $M=3$, the numerical solution using $N=100$ elements is not only blurred
compared to the more resolved numerical solution using $N=3200$ elements, but displays the
jump at different locations. This is illustrated in greater detail in Figure
\ref{fig:shock_M_3__p_0__N_100_200_400_800_1600_3200}, where the numerical
solutions for $M=3$, $p=0$, and increasing $N=100,200,400,800,1600,3200$ are shown.

\begin{figure}[!htp]
\centering
\captionsetup[subfigure]{aboveskip=-8pt,belowskip=-1pt}
  \begin{subfigure}[b]{0.468\textwidth}
    \includegraphics[width=\textwidth]{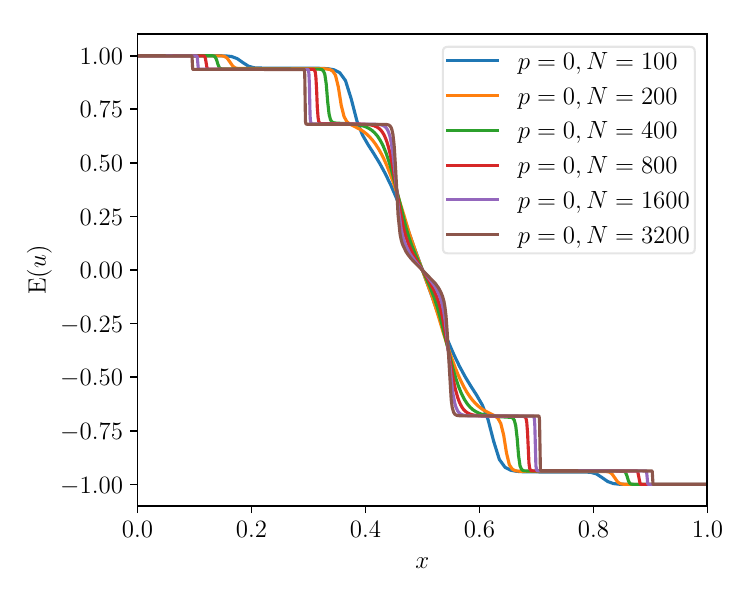}
    \caption{$M=3$, expectation $\E[u]$.}
    \label{fig:shock_M_3__p_0__N_100_200_400_800_1600_3200__Eu}
  \end{subfigure}%
  ~
  \begin{subfigure}[b]{0.468\textwidth}
    \includegraphics[width=\textwidth]{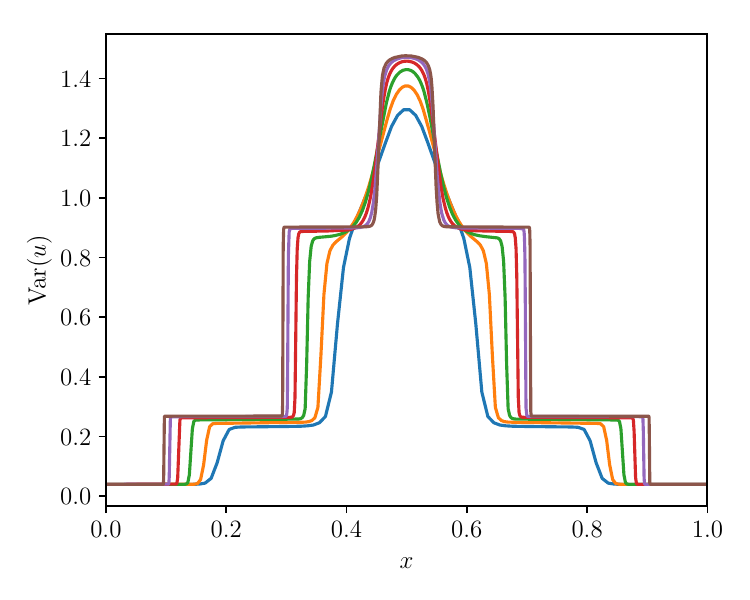}
    \caption{$M=3$, variance $\Var(u)$.}
    \label{fig:shock_M_3__p_0__N_100_200_400_800_1600_3200__Varu}
  \end{subfigure}%
  \caption{Solution at time $t=0.5$ for a shock of uncertain height as initial condition.}
  \label{fig:shock_M_3__p_0__N_100_200_400_800_1600_3200}
\end{figure}
We note that the jump locations are observed to vary with the spacial resolution.
A similar behaviour is observed when the polynomial degree $p$ is increased. Then,
we get additionally spurious oscillations, resulting from the Gibbs phenomenon.
However, a further detail can be observed for the finite order systems.
The numerical solution for higher polynomial degrees
indicates another structure of the solution.
In what follows, this phenomenon is investigated  in greater detail.

For $t=0.5$, Figure~\ref{fig:shock-CPR-vs-FV} displays the expectation $\E[u]$
and variance $\Var(u)$ of the numerical solutions from the SBP CPR and FV
methods.
There, also $\E[u]$ and $\Var(u)$ of the reference solution $\uref$ are
illustrated.
While $N=\num{2 500}$ elements were used for the SBP CPR method with polynomial
degree $p=3$, for the corresponding FV methods $N = \num{10 000}$ elements were used.
In the following for both methods, \num{100 000} time steps were used.

\begin{figure}[!htp]
\centering
\captionsetup[subfigure]{aboveskip=-8pt,belowskip=-1pt}
  \begin{subfigure}[b]{0.495\textwidth}
    \includegraphics[width=\textwidth]{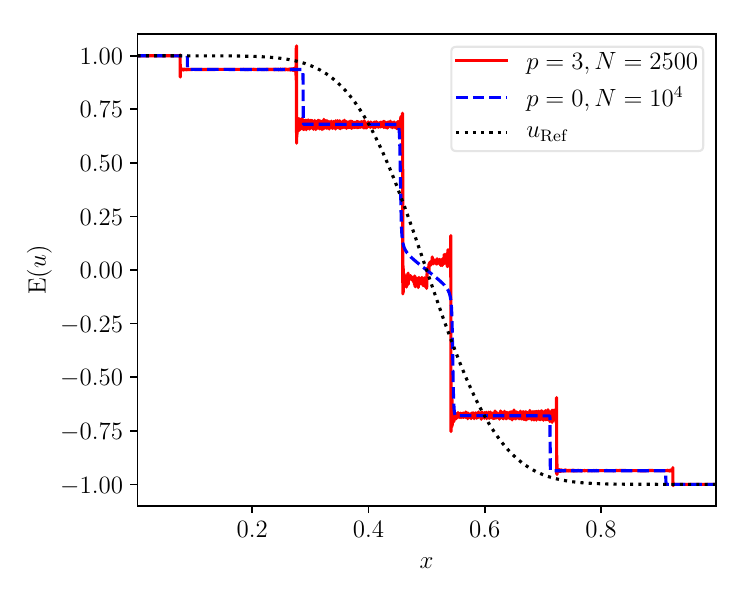}
    \caption{Expectation $\E[u]$.}
  \end{subfigure}%
  ~
  \begin{subfigure}[b]{0.495\textwidth}
    \includegraphics[width=\textwidth]{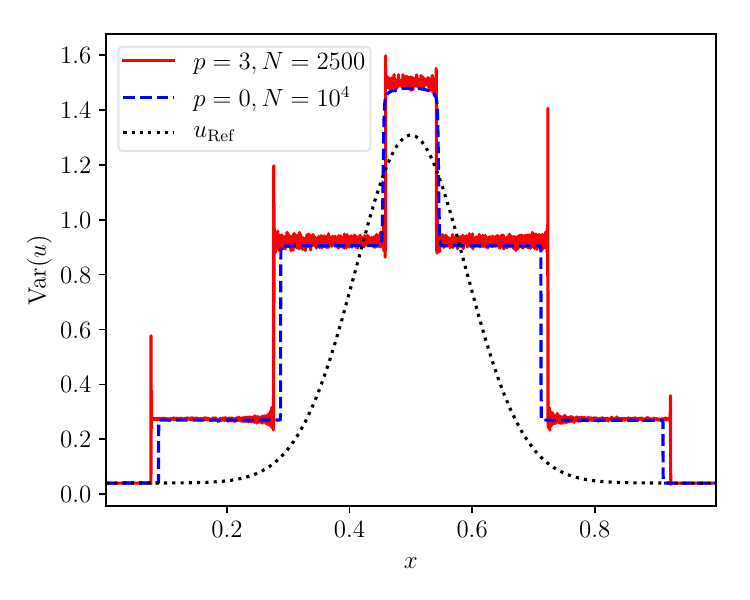}
    \caption{Variance $\Var(u)$.}
  \end{subfigure}%
  \caption{Solution at time $t=0.5$ for a shock of uncertain height as initial condition.}
  \label{fig:shock-CPR-vs-FV}
\end{figure}

Three observations should be pointed out immediately:
\begin{enumerate}
  \item
 Both numerical solutions differ significantly from the reference
  solution $\uref$.

  \item
  Both numerical solutions show more wave-fronts than one would expect
  from classical theory of Riemann problems for strictly hyperbolic systems with
  genuinely nonlinear or linearly degenerate fields.

  \item
  Also the numerical solutions themselves show quite different features,
  especially in their wave profiles around $x_0 = 0.5$.
\end{enumerate}

 The first observation was also pointed out by Pettersson et al.
\cite{pettersson2009numerical,pettersson2015polynomial} and arises from the
truncation of the infinite order system \eqref{shock} to the four dimensional
system \eqref{BurgerGalerkinUQ} for $M=3$.
Thus, the analytical solution $\uref$ of the infinite order system for $M
\to \infty$ differs from 'the analytical solution' of the truncated system.
The latter one is however approximated by the numerical solutions.
The difference between $\uref$ and the analytical solution to the
truncated system already gets stressed by mismatching regularities.
While $\uref$ was shown to be smooth in the last
section, the solution of the truncated system is expected to feature
discontinuities.  We saw this already before in our numerical tests and
refer again to Remark
\ref{discontinuities_problems} and the literature therein.

Naturally the following question arises:
\emph{What is the analytical solution of the truncated system?}
For Riemann problems of strictly hyperbolic systems with
genuinely nonlinear or linearly degenerate fields, there are in fact clear
results in the literature \cite{lax1973hyperbolic} on how solutions behave.
In particular, the analytical solution of a $M+1=4$ dimensional system
consists of at most $M+2=5$ constant states which are connected by shock
discontinuities or expansion waves.
At the same time the numerical solutions in Figure~\ref{fig:shock-CPR-vs-FV}
both show at least $6$ such constant states.
Thus, classical theory of Riemann problems for strictly hyperbolic systems
obviously fails.
This is due to the assumption of the system to be
strictly hyperbolic, i.e. to have real distinct eigenvalues.
Already the steady state $(u_0,u_1,u_2,u_3) = (0,0,0,0)$ yields matrix
$A(u) = 0$ in \eqref{eq:M=3matrix} to have eigenvalues
$\lambda_1=\lambda_2=\lambda_3=\lambda_4=0$.
Non distinct eigenvalues can also be observed in Figures
\ref{fig:shock-FV-coef-eigvals-U-F} and \ref{fig:shock-FD-coef-eigvals-U-F} for
the numerical solutions of this particular Riemann problem.

\begin{re}
For scalar conservation laws, the PC approach yields a symmetric
system. Due to the symmetry, the arising system is therefore hyperbolic, see
\cite{xiu2010numerical,chertock2015operator}.
In general, the arising system is however nonstrictly hyperbolic.
To calculate the exact analytical solution for this system, the eigenvalues have
to be known. Indeed, further results can be found in the literature, e.g. that they are analytical
\cite[Chapter II, Theorem 1.8]{kato1995perturbation}.
Nevertheless, without further assumptions on the entries of the $4 \times 4$
symmetric matrix, the eigenvalues can not be expressed by a simple and closed formula
[private communication with Harald Löwe, TU Braunschweig].
This can be proved by an algebraic approach and it is  beyond the scope of this paper.
\end{re}

Finally, the third observation - different profiles of the numerical
solutions - shall be addressed.
In Figure~\ref{fig:shock-CPR-vs-FV}, the numerical solutions for the SBP CPR
method and the corresponding FV method essentially differ in three aspects:
\begin{enumerate}
  \item Their behaviour near the centre $x_0 = 0.5$.
  \item The position of the shock discontinuities away from $x_0=0.5$, e.g. at
$x \approx 0.276$ for the SBP CPR method and at $x \approx 0.287$ for the FV
method.
  \item The height of the constant states. This can't be seen without zooming
in, which is therefore done in Figure~\ref{fig:shock-scaled-dissipation}.
\end{enumerate}

Noticing these differences, another question arises:
\emph{What is the mechanism behind this?}

By a large number of different tests for the SBP CPR, FV and SBP FD
method, the numerical dissipation added by the underlying scheme was explored
as the determining factor.
As Figure~\ref{fig:shock-scaled-dissipation} demonstrates, different
profiles for the expectation $\E[u]$ can be reproduced by all of the three
methods.
The plot (a) and (b) show the expectation $\E[u]$ and variance $\Var(u)$ for the 
SBP CPR method, (c) and (d) for the FV method, and (e) and (f) for the SBP FD
method.
The (red) solid line thereby illustrates a numerical solution obtained by
the corresponding method equipped with low numerical dissipation.
The (blue) dashed line, on the other hand, illustrates a numerical solution
obtained by the corresponding method equipped with high numerical dissipation.

\begin{figure}[!htp]
\centering
\captionsetup[subfigure]{aboveskip=-8pt,belowskip=-1pt}
  \begin{subfigure}[b]{0.468\textwidth}
    \includegraphics[width=\textwidth]{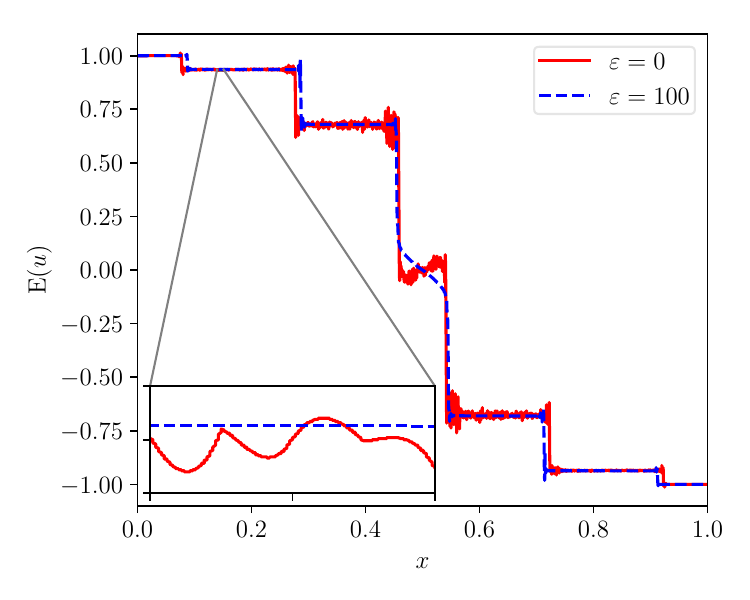}
    \caption{CPR, expectation $\E[u]$.}
  \end{subfigure}%
  ~
  \begin{subfigure}[b]{0.468\textwidth}
    \includegraphics[width=\textwidth]{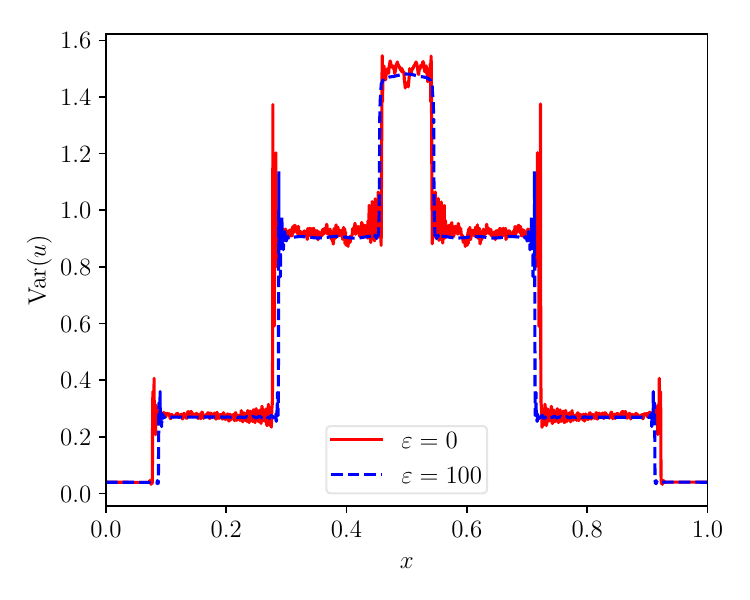}
    \caption{CPR, variance $\Var(u)$.}
  \end{subfigure}%
  \\
  \begin{subfigure}[b]{0.468\textwidth}
    \includegraphics[width=\textwidth]{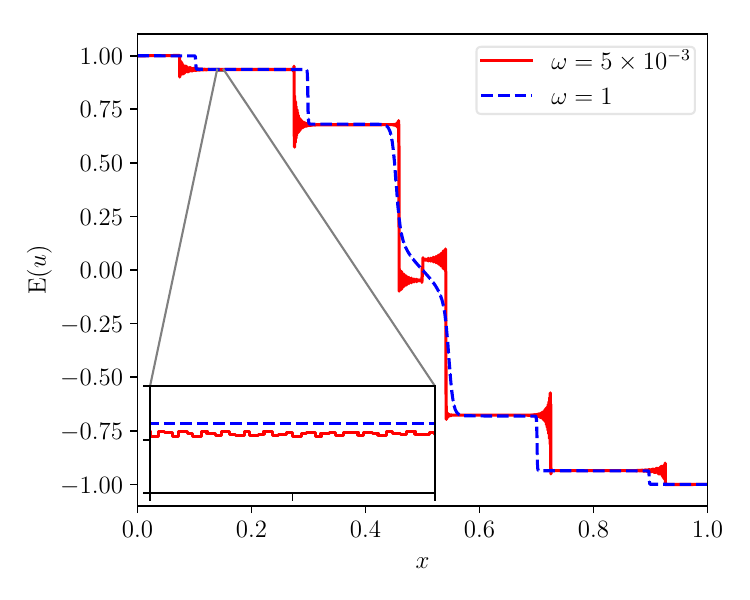}
    \caption{FV, expectation $\E[u]$.}
  \end{subfigure}%
  ~
  \begin{subfigure}[b]{0.468\textwidth}
    \includegraphics[width=\textwidth]{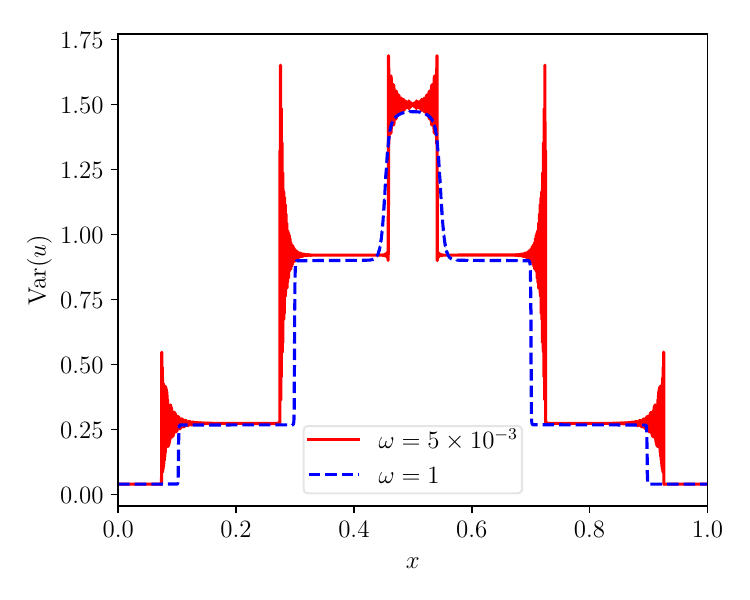}
    \caption{FV, variance $\Var(u)$.}
  \end{subfigure}%
  \\
  \begin{subfigure}[b]{0.468\textwidth}
    \includegraphics[width=\textwidth]{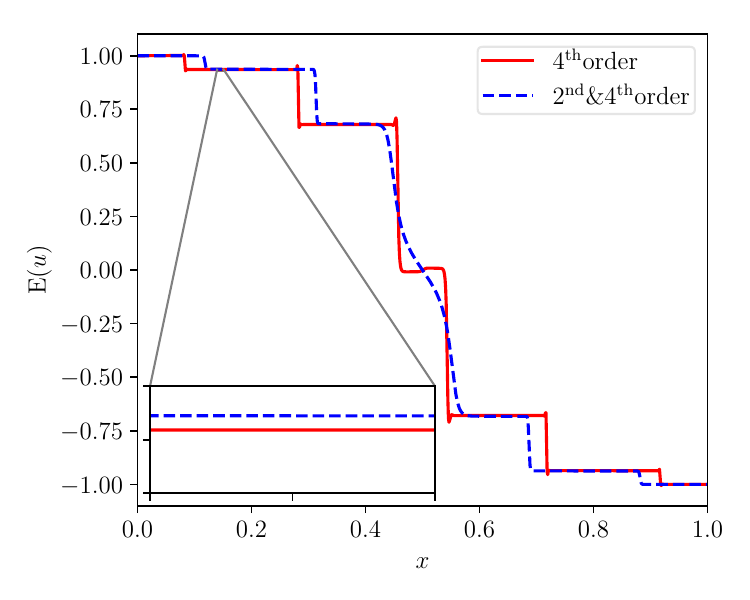}
    \caption{FD, expectation $\E[u]$.}
  \end{subfigure}%
  ~
  \begin{subfigure}[b]{0.468\textwidth}
    \includegraphics[width=\textwidth]{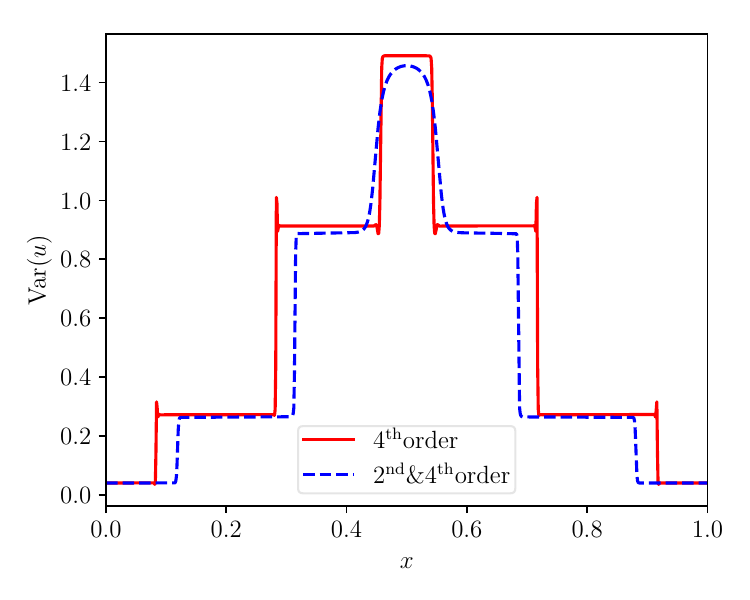}
    \caption{FD, variance $\Var(u)$.}
  \end{subfigure}%
  \caption{Scaled numerical dissipation for the shock of uncertain height as
           initial condition. Numerical solutions by SBP CPR, FV and SBP FD.
           Parameters: Polynomial chaos order $M=3$, final time $t=0.5$,
           inflow boundary conditions.}
  \label{fig:shock-scaled-dissipation}
\end{figure}

In case of the SBP CPR method, numerical dissipation was added by applying
modal filtering by an exponential filter of order $s=1$ and strength
$\varepsilon = 100$ in every element and after every time step.
See \cite{ranocha2018stability, glaubitz2016artificial, glaubitz2018application}
for details.

For the FV method however, already showing a smeared profile, numerical
dissipation was reduced.
This was done by multiplying the dissipation matrix added to the entropy
conservative flux with decreasing weights $0 < \omega \leq 1$.

Numerical dissipation in the SBP FD method of Pettersson et al.
\cite{pettersson2009numerical,pettersson2015polynomial}
refers to artificial dissipation terms of second and fourth order,
see \cite{mattsson2004stable, nordstrom2006conservative} for details.
Thus, the original conservation law is extended by viscosity terms of second
and fourth order derivatives which are properly discretised and weighted.
To reduce numerical dissipation, the second order derivative was nullified and
just the fourth order derivative was used for the artificial dissipation.
The results for the SBP FD method were furthermore computed by the matlab code
of Pettersson et al. which they have offered very well prepared in
\cite{pettersson2015polynomial}.

Summarising the results from the described numerical tests, the
following can be observed.
\begin{re}
  The numerical solutions of the truncated system \eqref{eq:Riemann_shock} for
$M=3$ differ significantly with respect to the numerical dissipation added by
the underlying scheme.
In particular, the wave profile near $x_0 = 0.5$ shows
quite varying features.\\  As it is already investigated and summarised in
\cite[Chapter 6]{pettersson2015polynomial} for SBP FD methods, the influence of dissipation is
enormously in the PC approach.
Excessive use of artificial dissipation can give a numerical solution that more closely resemble the
solution for the original problem ($M\to \infty $) compared to a solution where
the order of the polynomial chaos method is increased and only a small amount of dissipation is
applied, but these features of the deterministic $4\times 4$ system are not even mentioned there.
The numerical solutions may differ significantly depending on the artificial dissipation of the schemes,
see for instance \cite{mattsson2004stable}. But even with this knowledge, these specifics
look impressive.
From a heuristic point of view, the numerical solutions calculated with the higher amount
of dissipation seem more reasonable after all, but to be sure, one has to analyse it.
The impression suggests a contact discontinuity at this point, which yields  these different
numerical solutions depending on the dissipation.

\end{re}

Since no analytical solution of the observed system is known, other criteria
should be examined.
The idea is to identify numerical solutions which are physically reasonable and
to reject the ones which are not.
For solutions of hyperbolic conservation laws to be physically reasonable,
typically two conditions are checked.
First, the \emph{Rankine-Hugoniot jump condition}
\begin{equation}\label{RH}
  s \jump{u_k} = \jump{f_k}
\end{equation}
across every discontinuity, where $s$ is the speed of propagation of
the discontinuity.
Second, the
\emph{entropy inequality} $\partial_t U + \partial_x F \leq 0$, which is
equivalent to a \emph{Rankine-Hugoniot condition for the entropy}
\begin{equation}\label{RH_entropy}
  \jump{F} \leq s \jump{U}.
\end{equation}
The left hand sides (a) of Figures \ref{fig:shock-FV-coef-eigvals-U-F} and
\ref{fig:shock-FD-coef-eigvals-U-F} each show results from the corresponding
method using more dissipation.
The results on the right hand sides (b) each show results from the same method
using
significant less numerical dissipation.

\begin{figure}[!htp]
\centering
\captionsetup[subfigure]{aboveskip=-4pt,belowskip=-1pt}
  \begin{subfigure}[b]{0.495\textwidth}

\includegraphics[width=\textwidth]{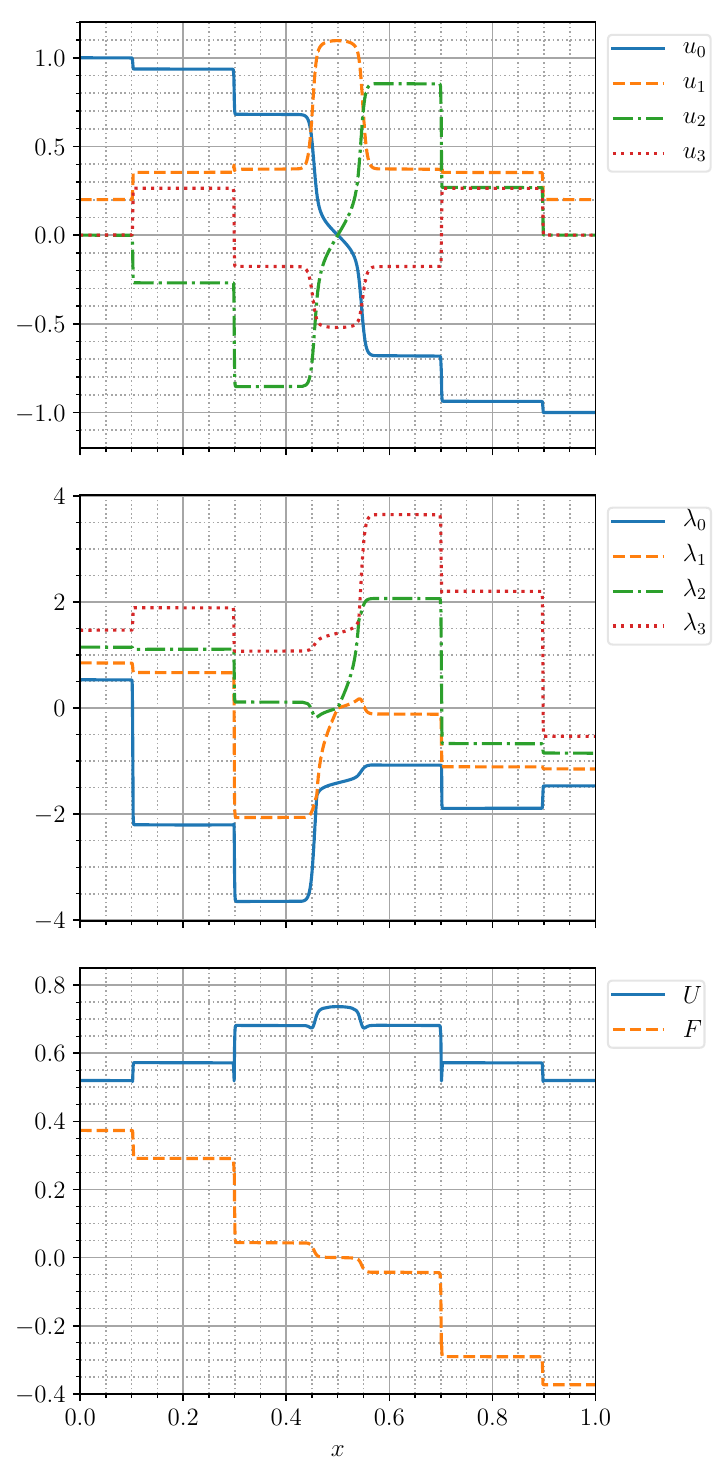}
    \caption{$\omega = 1$.}
  \end{subfigure}%
  ~
  \begin{subfigure}[b]{0.495\textwidth}

\includegraphics[width=\textwidth]{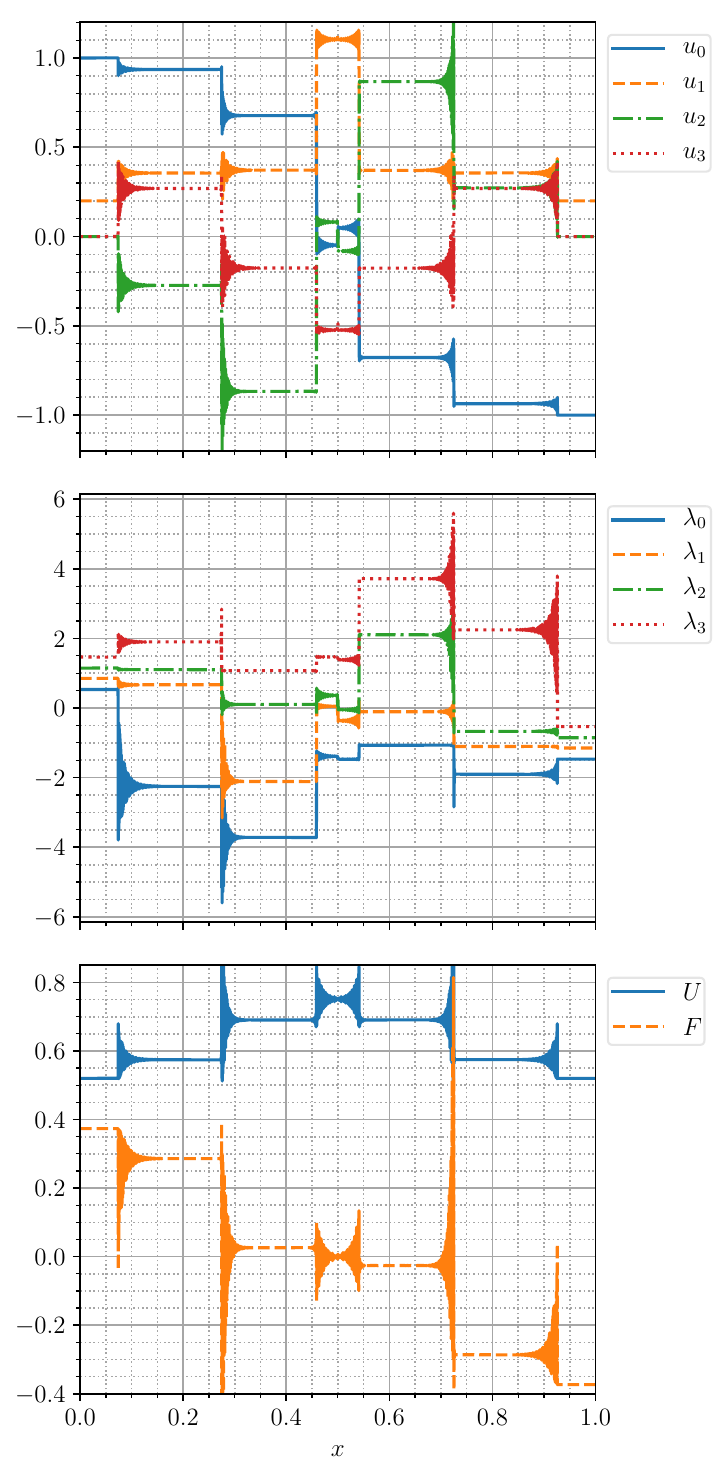}
    \caption{$\omega = 5 \times 10^{-3}$.}
  \end{subfigure}%
  \caption{Numerical solutions, eigenvalues, entropy and entropy flux for
  scaled numerical dissipation for the shock of uncertain height as initial
condition.
  Numerical solutions by FV with 'full' dissipation $\omega=1$ and reduced
dissipation $\omega = 5 \times 10^{-3}$.
  Parameters: Polynomial chaos order $M=3$, time boundary $t=0.5$,
  inflow boundary conditions.}
  \label{fig:shock-FV-coef-eigvals-U-F}
\end{figure}

The FV solution in Figure~\ref{fig:shock-FV-coef-eigvals-U-F}
already contains quite spurious oscillations (and the CPR solution even more).
Similar plots are given in Figure~\ref{fig:shock-FD-coef-eigvals-U-F} for
the FD solution.
These plots show nearly no oscillations.

\begin{figure}[!htp]
\centering
\captionsetup[subfigure]{aboveskip=-4pt,belowskip=-1pt}
  \begin{subfigure}[b]{0.495\textwidth}
    \includegraphics[width=\textwidth]{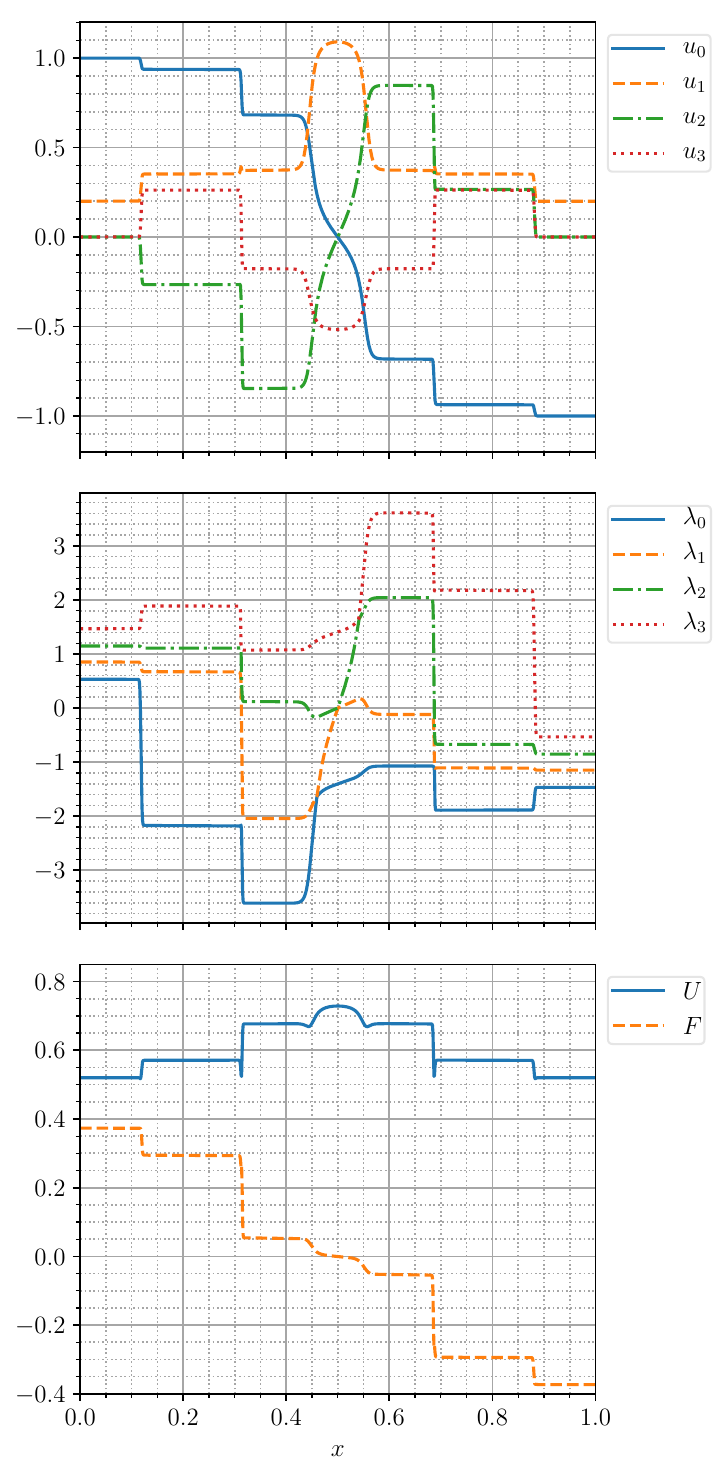}
    \caption{$2^\mathrm{nd}$ \& $4^\mathrm{th}$ order dissipation.}
  \end{subfigure}%
  ~
  \begin{subfigure}[b]{0.495\textwidth}
    \includegraphics[width=\textwidth]{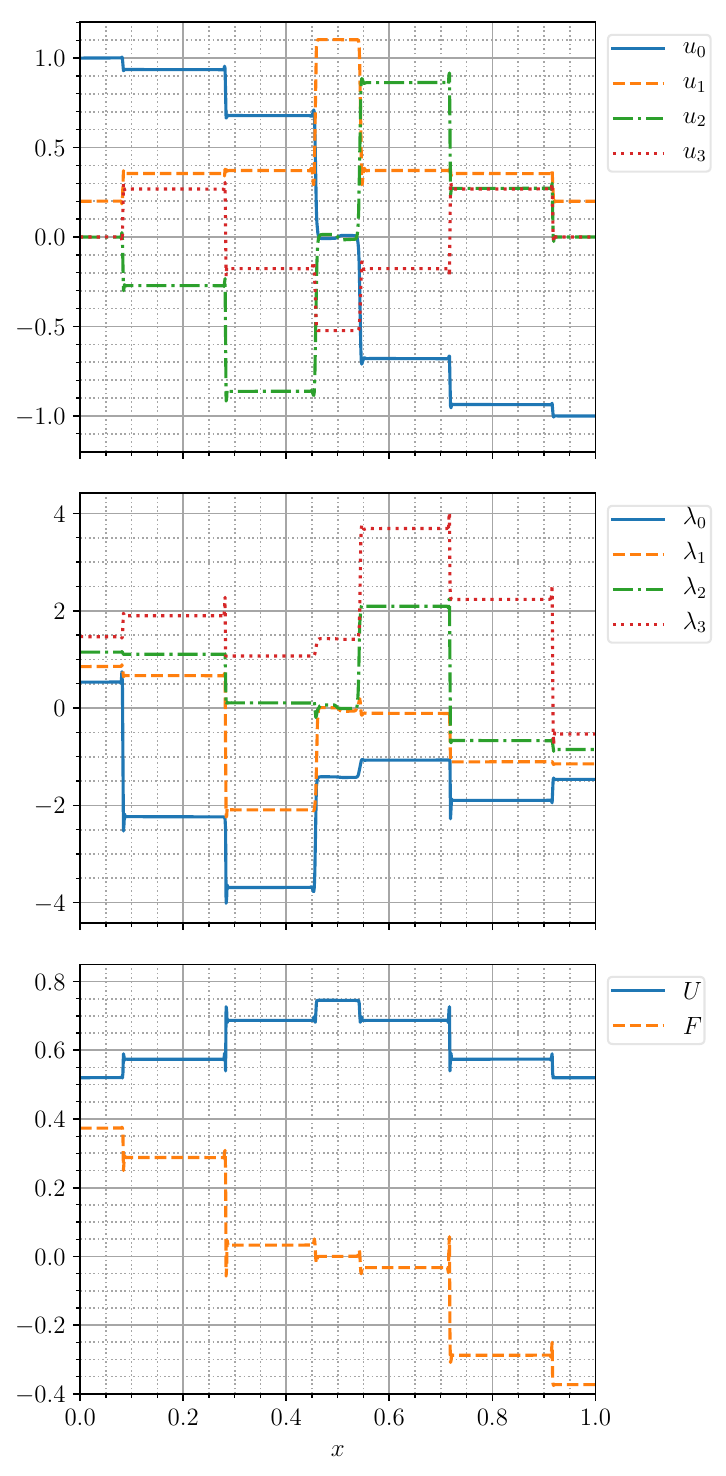}
    \caption{$4^\mathrm{th}$ order dissipation.}
  \end{subfigure}%
  \caption{Numerical solutions, eigenvalues, entropy and entropy flux for
  scaled numerical dissipation for the shock of uncertain height as initial
condition.
  Numerical solutions by SBP FD with 'full' dissipation by both, the
$2^\mathrm{th}$ and $4^\mathrm{th}$ order dissipation, and reduced dissipation
just by the $4^\mathrm{th}$ order dissipation.
  Parameters: Polynomial chaos order $M=3$, time boundary $t=0.5$,
  inflow boundary conditions.}
  \label{fig:shock-FD-coef-eigvals-U-F}
\end{figure}

Since the essential differences arise near the centre $x_0=0.5$, the focus is
on
the profiles of the numerical solutions there.
Both methods lead to smooth non-constant transitions at $x_0=0.5$ for high
numerical dissipation, see (a) in Figures \ref{fig:shock-FV-coef-eigvals-U-F}
and
\ref{fig:shock-FD-coef-eigvals-U-F}.
At the same time, there is a nearly constant transition for the FD solution
with low dissipation, see (b) in Figure~\ref{fig:shock-FD-coef-eigvals-U-F},
and a piecewise constant transition with a up-jump at $x_0=0.5$ for the FV
solution with even lower numerical dissipation, see (b) in Figure
\ref{fig:shock-FV-coef-eigvals-U-F}.

Except for the last case, the Rankine-Hugoniot jump condition is obviously
fulfilled, since no discontinuity occurs.
The FV solution (b) in Figure~\ref{fig:shock-FV-coef-eigvals-U-F} also
fulfils this condition with speed of propagation $s=0$.

Furthermore, also the Rankine-Hugoniot jump condition for the entropy
\eqref{RH_entropy} is obviously satisfied by all solutions due to no jumps in
the entropy $U$ and flux $F$ at $x_0 = 0.5$.

 Both approaches do not answer the question of what the physically relevant solution for this system is.
Other investigations are necessary.

\section{Summary and Conclusions}
\label{sec:summary}

In this work, a  polynomial chaos approach for Burgers' equation has been 
applied and the resulting hyperbolic system has been considered in the general 
framework of CPR methods using SBP operators.
Besides conservation, focus was especially given to stability, which was 
proven for the CPR method and all systems arising from the PC approach.
Due to the usage of split-forms similar to \cite{fisher2013high, 
carpenter2014entropy, carpenter2013high}, the major challenge was to construct 
entropy stable numerical fluxes.
For the first time, this was done rigorously for all systems resulting from the 
PC approach for Burgers' equation.

Furthermore, numerical results for two different test cases have been 
examined. 
Burgers' equation with an initial rarefaction demonstrated 
convergence for the truncated system, whereas for the convergence study 
to the reference solution a diagonal limit should be used. 
More interesting, they also highlighted clear differences in numerical 
dissipation added by the schemes.
It became clear in the last test case that this is crucial. 
In fact, the last test case, i.e. Burgers’ equation with an initial shock, 
has been the most remarkable one. 
Quite fascinating observations have been highlighted. 

All numerical solutions showed more wave fronts than one would expect from 
classical theory of Riemann problems for strictly hyperbolic systems with
genuinely nonlinear or linearly degenerate fields.
Furthermore, the numerical solutions featured quite different 
behaviours, especially in their wave profiles around $x_0=0.5$, highly 
depending on the numerical dissipation.
It seems likely that the numerical schemes with a lot of artificial dissipation 
yields the \emph{correct} numerical solutions, but it is still not clear and 
must further be examined. 

In fact, it remains an open problem for nonstrictly hyperbolic systems of 
conservation laws what (entropy) conditions might ensure uniqueness and even 
existence of solutions.
All of them, just one, or none of the numerical solutions might indeed converge to a 
reasonable solution.
Nevertheless, a quite fascinating dependence on the added numerical 
dissipation could be observed.

At this point, a broad field of open problems for the analytical as well as 
numerical treatment of (nonstrictly) hyperbolic problems arises.
In particular, the authors look forward to further research on this.

\section*{Acknowledgements}
The authors would like to thank Harald Löwe for his helpful investigation
and comments about the eigenvalues of a symmetric matrix.
Moreover, they would like to thank the anonymous reviewers very much for their
helpful comments which helped in improving this article.

\appendix
\section{Appendix}
\label{Appendix}

\subsection{Hermite Polynomials}
\label{Appendix:Hermite}

The probabilistic version of the Hermite polynomials\footnote{There is another
way to normalise the Hermite polynomials which is applied mostly in mathematical physics.
Here, we follow  the definition which is used in probability theory and, therefore,
sometimes the polynomials are also called probabilistic Hermite polynomials \cite{pettersson2009numerical}. } has as weight function
$\varrho(y)=\frac{1}{\sqrt{2 \pi}} \exp{\frac{-y^2}{2}}$. This is the probability density
function of a Gaussian distribution. Therefore, using normalised Hermite
polynomials as basis
functions is an intuitive choice.
A table of the ``natural'' orthogonal basis functions in dependence of different
distributions of random variables is given in \cite[Table 4.1]{xiu2002wiener}.
Here, we restrict ourselves to Gaussian measures and so we will only repeat the main
properties of the normalised Hermite polynomials.
These and further results can be found in \cite{abramowitz1972handbook,
szego1975orthogonal}.

The inner product of the normalised Hermite polynomials
$\phi $ of a Gaussian variable $\xi $ is defined by
\begin{equation}
\E[\phi_i \phi_j]
 =
 \int_{\R} \phi_i(y) \phi_j(y) \varrho(y)  \dif y
 =
 \delta_{i,j} .
\end{equation}
The triple product is given by
\begin{equation}
\label{Hermitetripel}
 \E[\phi_i \phi_j \phi_k]
 =
  \begin{cases}
  0  \qquad &\text{ if $ i+j+k$ is odd  or $\max(i,j,k)>s$ },\\
  \frac{\sqrt{i!j!k!}}{(s-i)! (s-j)! (s-k)! } \qquad &\text{ otherwise},
\end{cases}
\end{equation}
where $s=(i+j+k)/2$.

We also need the following relation
\begin{equation}
\label{recurrencerelation}
  \sqrt{i} \phi_{i} \varrho = - \left(
\phi_{i-1} \varrho \right)'.
\end{equation}

\subsection{Stability of CPR Method}
\label{Appendix_Stability}
Here, we present the calculation for equation \eqref{Gleichungstabilitaet1} from subsection
\ref{Subsec:Stability} in detail.
This is similar to the work of \cite{pettersson2009numerical}.
We start with \eqref{eq:CPR_systemI} and consider
\begin{equation}
\label{eq:CPR_systemI_a}
\begin{aligned}
  \partial_t \vec{u}
  + \frac{\beta}{2} (\mat{D}\otimes \hat{\I}) A_G \vec{u}
  + (1-\beta) \left( A_G(\mat{D}\otimes \hat{\I})\vec{u}
                   \right)+ \left(\left(\mat{M}[^{-1}] \mat{R}[^T] \mat{B}\right)\otimes \hat{\I}\right) \left(
    \vecfnum - \frac{1}{2} (\mat{R} \otimes \hat{\I} ) A_G \vec{u} \right)&
  =
  0,
\end{aligned}
\end{equation}
where $\vecfnum$ is the numerical flux
and $\vec{u}= \l(u_0( \zeta_0), \dots, u_0(\zeta_p),u_1(\zeta_0), \dots, u_M(\zeta_p)\right)^T$
is the combination vector from SBP CPR and the polynomial chaos method.
Investigating $\L^2$ stability, we multiply \eqref{eq:CPR_systemI} with
$\vec{u}^T (\mat{M}\otimes \hat{\I})$. With the SBP property \eqref{eq:SBP}, we get
\begin{equation}
\begin{aligned}
  &
  \frac{1}{2} \frac{\dif}{\dif t} \norm{\vec{u}}_{\mat{M}\otimes \hat{\I}}^2
  =
  \vec{u}^T(\mat{M}\otimes \hat{\I})(\hat{\I} \otimes \mat{\I}) \partial_t \vec{u}
  \\
  =&
  -\frac{\beta}{2} \vec{u}^T(\mat{M}\otimes \hat{\I}) (\mat{D}\otimes \hat{\I}) A_G \vec{u}
  - (1-\beta) \vec{u}^T (\mat{M}\otimes \hat{\I}) A_G (\mat{D}\otimes \hat{\I} ) \vec{u}
  \\&
  - \vec{u}^T(\mat{M}\mat{M}[^{-1}] \mat{R}[^T] \mat{B}\otimes \hat{\I})
  \left( \vecfnum-\frac{1}{2} (\mat{R}\otimes \hat{\I}) A_G \vec{u}) \right)
  \\=&
  -\frac{\beta}{2} \vec{u}^T(\mat{M}\mat{D}\otimes \hat{\I})
  A_G \vec{u}-(1-\beta) \vec{u}^T(\mat{M}\otimes \hat{\I}) A_G (\mat{D}\otimes \hat{\I}) \vec{u}
  -\vec{u}^T( \mat{R}^T\mat{B} \otimes \hat{\I})
  \left(\vecfnum -\frac{1}{2} (\mat{R}\otimes \hat{\I}) A_G \vec{u} \right)
  \\
  =&
  \frac{\beta}{2} \vec{u}^T (\mat{D}^T \mat{M}\otimes \hat{\I}) A_G \vec{u}
  -\frac{\beta}{2} \vec{u}^T (\mat{R}^T \mat{B} \mat{R} \otimes \hat{\I}) A_G \vec{u}
  - (1-\beta) \vec{u}^T (\mat{M}\otimes \hat{\I}) A_G
  (\mat{D}\otimes \hat{\I}) A_G (\mat{D}\otimes \hat{\I}) \vec{u}
  \\&
  -\vec{u}^T (\mat{R}^T \mat{B} \otimes \hat{\I})
  \left( \vecfnum-\frac{1}{2} (\mat{R}\otimes \hat{\I}) A_G \vec{u} \right).
\end{aligned}
\end{equation}
We choose $\beta=\frac{2}{3}$. This yields
\begin{equation}\label{eq:CPR_SystemII}
\begin{aligned}
  \frac{1}{2} \frac{\dif}{\dif t} \norm{u}_{M\otimes \hat{\I}}^2 =& \frac{1}{3}\vec{u}^T(\mat{D}^T \otimes \hat{\I}) (\mat{M}\otimes \hat{\I})A_G
  \vec{u} -\frac{1}{3} \vec{u}^T  (\mat{R} ^T \mat{B} \mat{R} \otimes \hat{\I}) A_G \vec{u}\\
  &- \frac{1}{3} \vec{u}^T (\mat{M}\otimes \hat{\I}) A_G (\mat{D}\otimes \hat{\I}) A_G (\mat{D}\otimes \hat{\I}) \vec{u}
  -\vec{u}^T (\mat{R}^T \mat{B} \otimes \hat{\I}) \left( \vecfnum-\frac{1}{2} (\mat{R}\otimes \hat{\I}) A_G  \vec{u}\right).
\end{aligned}
\end{equation}
$A_G$ commutes with $\mat{M}\otimes \hat{\I}$. This means that
\begin{equation}\label{A_Gkommutiert}
 A_G= (\mat{M}\otimes \hat{\I}) A_G (\mat{M}^{-1}\otimes \hat{\I}).
\end{equation}
Applying this fact, we get
\begin{equation}
\begin{aligned}
 \frac{1}{3}\vec{u}^T(\mat{D}^T \otimes \hat{\I}) (\mat{M}\otimes \hat{\I})A_G u &=\frac{1}{2 } \left( (\mat{D}\otimes \hat{\I}) \vec{u}
 \right)^T (\mat{M}\otimes \hat{\I}) A_G \vec{u} =\frac{1}{3} \left[(A_Gu)^T(\mat{M}\otimes \hat{\I})^T (\mat{D}\otimes \hat{\I}) \vec{u})\right]^T\\
 &= \frac{1}{3}\vec{u}^T(\mat{M}\otimes \hat{\I}) A_G (\vec{D}\otimes \hat{\I}) \vec{u}.
\end{aligned}
\end{equation}
Finally, we employ this in \eqref{eq:CPR_SystemII} and receive
\begin{equation}\label{eq:CPR_SystemIII}
  \frac{1}{2} \frac{\dif}{\dif t} \norm{u}_{\mat{M}\otimes \hat{\I}}^2 = \frac{1}{6}\vec{u}^T (\mat{R}^T\mat{B} \mat{R}\otimes \hat{\I}) A_G \vec{u} -\vec{u}^T(\mat{R}^T \mat{B} \otimes \hat{\I})
  \vecfnum.
\end{equation}
Inserting the right and left values in one element, we get
 we get
\begin{equation}
  \frac{1}{2}\frac{\dif}{\dif t} \norm{u}^2_{\mat{M} \otimes \hat{\I}}
 =
\frac{1}{6} u^{(e),T}_R A(u^{(e)}_R) u^{(e)}_R -
 \frac{1}{6} u^{(e),T}_LA(u^{(e)}_L)u^{(e)}_L
 +u^{(e),T}_Lf^{\mathrm{num}, e}_L
 -u^{(e),T}_R f^{\mathrm{num}, e}_R. \tag{\ref{Gleichungst}}
\end{equation}

\bibliographystyle{habbrv}
\bibliography{literature}

\begin{thebibliography}{10}

\bibitem{abgrall2017uncertainty}
R.~Abgrall and S.~Mishra.
\newblock Uncertainty quantification for hyperbolic systems of conservation
  laws.
\newblock In {\em Handbook of Numerical Analysis}, volume~18, pages 507--544.
  Elsevier, 2017.

\bibitem{abramowitz1972handbook}
M.~Abramowitz and I.~A. Stegun.
\newblock {\em Handbook of mathematical functions}.
\newblock National Bureau of Standards, 1972.

\bibitem{cameron1947orthogonal}
R.~H. Cameron and W.~T. Martin.
\newblock The orthogonal development of non-linear functionals in series of
  {F}ourier-{H}ermite functionals.
\newblock {\em Annals of Mathematics}, pages 385--392, 1947.

\bibitem{carpenter2013high}
M.~H. Carpenter and T.~C. Fisher.
\newblock High-order entropy stable formulations for computational fluid
  dynamics.
\newblock In {\em 21st AIAA Computational Fluid Dynamics Conference}. American
  Institute of Aeronautics and Astronautics, 2013.

\bibitem{carpenter2014entropy}
M.~H. Carpenter, T.~C. Fisher, E.~J. Nielsen, and S.~H. Frankel.
\newblock Entropy stable spectral collocation schemes for the {N}avier-{S}tokes
  equations: {D}iscontinuous interfaces.
\newblock {\em SIAM Journal on Scientific Computing}, 36(5):B835--B867, 2014.

\bibitem{chertock2015operator}
A.~Chertock, S.~Jin, and A.~Kurganov.
\newblock An operator splitting based stochastic {G}alerkin method for the
  one-dimensional compressible {E}uler equations with uncertainty, 2015.
\newblock \url{http://www.ki-net.umd.edu/pubs/files/Euler-UQ.pdf}.

\bibitem{fisher2013high}
T.~C. Fisher and M.~H. Carpenter.
\newblock High-order entropy stable finite difference schemes for nonlinear
  conservation laws: {F}inite domains.
\newblock {\em Journal of Computational Physics}, 252:518--557, 2013.

\bibitem{fisher2013discretely}
T.~C. Fisher, M.~H. Carpenter, J.~Nordstr{\"o}m, N.~K. Yamaleev, and
  C.~Swanson.
\newblock Discretely conservative finite-difference formulations for nonlinear
  conservation laws in split form: {T}heory and boundary conditions.
\newblock {\em Journal of Computational Physics}, 234:353--375, 2013.

\bibitem{gassner2013skew}
G.~J. Gassner.
\newblock A skew-symmetric discontinuous {G}alerkin spectral element
  discretization and its relation to {SBP}-{SAT} finite difference methods.
\newblock {\em SIAM Journal on Scientific Computing}, 35(3):A1233--A1253, 2013.

\bibitem{gassner2016split}
G.~J. Gassner, A.~R. Winters, and D.~A. Kopriva.
\newblock Split form nodal discontinuous {G}alerkin schemes with
  summation-by-parts property for the compressible {E}uler equations.
\newblock {\em Journal of Computational Physics}, 327:39--66, 2016.

\bibitem{ghanem2003stochastic}
R.~G. Ghanem and P.~D. Spanos.
\newblock {\em Stochastic finite elements: a spectral approach}.
\newblock Courier Corporation, 2003.

\bibitem{giesselmann2017posteriori}
J.~Giesselmann, F.~Meyer, and C.~Rohde.
\newblock A posteriori error analysis for random scalar conservation laws using
  the stochastic galerkin method.
\newblock {\em arXiv preprint arXiv:1709.04351}, 2017.

\bibitem{glaubitz2016artificial}
J.~Glaubitz, P.~{\"O}ffner, H.~Ranocha, and T.~Sonar.
\newblock Artificial viscosity for correction procedure via reconstruction
  using summation-by-parts operators.
\newblock In {\em XVI International Conference on Hyperbolic Problems: Theory,
  Numerics, Applications}, pages 363--375. Springer, 2016.

\bibitem{glaubitz2018application}
J.~Glaubitz, P.~{\"O}ffner, and T.~Sonar.
\newblock Application of modal filtering to a spectral difference method.
\newblock {\em Mathematics of Computation}, 87(309):175--207, 2018.

\bibitem{gottlieb1998total}
S.~Gottlieb and C.-W. Shu.
\newblock Total variation diminishing {R}unge-{K}utta schemes.
\newblock {\em Mathematics of Computation}, 67(221):73--85, 1998.

\bibitem{huynh2007flux}
H.~Huynh.
\newblock A flux reconstruction approach to high-order schemes including
  discontinuous {G}alerkin methods.
\newblock {\em AIAA paper}, 4079:2007, 2007.

\bibitem{huynh2014high}
H.~Huynh, Z.~J. Wang, and P.~E. Vincent.
\newblock High-order methods for computational fluid dynamics: {A} brief review
  of compact differential formulations on unstructured grids.
\newblock {\em Computers {\&} Fluids}, 98:209--220, 2014.

\bibitem{kato1995perturbation}
T.~Kato.
\newblock {\em Perturbation Theory for Linear Operators}.
\newblock Springer-Verlag Berlin Heidelberg, 1995.

\bibitem{lax1973hyperbolic}
P.~D. Lax.
\newblock {\em Hyperbolic systems of conservation laws and the mathematical
  theory of shock waves}.
\newblock SIAM, 1973.

\bibitem{mattsson2004stable}
K.~Mattsson, M.~Sv{\"a}rd, and J.~Nordstr{\"o}m.
\newblock Stable and accurate artificial dissipation.
\newblock {\em Journal of Scientific Computing}, 21(1):57--79, 2004.

\bibitem{2018arXiv180510177M}
F.~{Meyer}, L.~{Schlachter}, and F.~{Schneider}.
\newblock {A hyperbolicity-preserving discontinuous stochastic Galerkin scheme
  for uncertain hyperbolic systems of equations}.
\newblock {\em ArXiv e-prints}, May 2018.
  \href{http://arxiv.org/abs/1805.10177}{{\texttt{arXiv:1805.10177}}}.

\bibitem{mishra2016numerical}
S.~Mishra, N.~H. Risebro, C.~Schwab, and S.~Tokareva.
\newblock Numerical solution of scalar conservation laws with random flux
  functions.
\newblock {\em SIAM/ASA Journal on Uncertainty Quantification}, 4(1):552--591,
  2016.

\bibitem{mishra2012sparse}
S.~Mishra and C.~Schwab.
\newblock Sparse tensor multi-level monte carlo finite volume methods for
  hyperbolic conservation laws with random initial data.
\newblock {\em Mathematics of Computation}, 81(280):1979--2018, 2012.

\bibitem{mishra2013multi}
S.~Mishra, C.~Schwab, and J.~{\v{S}}ukys.
\newblock Multi-level monte carlo finite volume methods for uncertainty
  quantification in nonlinear systems of balance laws.
\newblock In {\em Uncertainty quantification in computational fluid dynamics},
  pages 225--294. Springer, 2013.

\bibitem{nordstrom2006conservative}
J.~Nordstr{\"o}m.
\newblock Conservative finite difference formulations, variable coefficients,
  energy estimates and artificial dissipation.
\newblock {\em Journal of Scientific Computing}, 29(3):375--404, 2006.

\bibitem{nordstrom2017roadmap}
J.~Nordstr{\"o}m.
\newblock A roadmap to well posed and stable problems in computational physics.
\newblock {\em Journal of Scientific Computing}, 71(1):365--385, 2017.

\bibitem{pettersson2015polynomial}
M.~P. Pettersson, G.~Iaccarino, and J.~Nordstr{\"o}m.
\newblock {\em Polynomial Chaos Methods for Hyperbolic Partial Differential
  Equations: Numerical Techniques for Fluid Dynamics Problems in the Presence
  of Uncertainties}.
\newblock Springer, 2015.

\bibitem{pettersson2009numerical}
P.~Pettersson, G.~Iaccarino, and J.~Nordstr{\"o}m.
\newblock Numerical analysis of the {B}urgers’ equation in the presence of
  uncertainty.
\newblock {\em Journal of Computational Physics}, 228(22):8394--8412, 2009.

\bibitem{poette2009uncertainty}
G.~Po{\"e}tte, B.~Despr{\'e}s, and D.~Lucor.
\newblock Uncertainty quantification for systems of conservation laws.
\newblock {\em Journal of Computational Physics}, 228(7):2443--2467, 2009.

\bibitem{ranocha2017shallow}
H.~Ranocha.
\newblock Shallow water equations: {S}plit-form, entropy stable, well-balanced,
  and positivity preserving numerical methods.
\newblock {\em GEM -- International Journal on Geomathematics}, 8(1):85--133,
  04 2017. \href{http://arxiv.org/abs/1609.08029}{{\texttt{arXiv:1609.08029}}}.

\bibitem{ranocha2018stability}
H.~Ranocha, J.~Glaubitz, P.~{\"O}ffner, and T.~Sonar.
\newblock Stability of artificial dissipation and modal filtering for flux
  reconstruction schemes using summation-by-parts operators.
\newblock {\em Applied Numerical Mathematics}, 128:1--23, 2018.

\bibitem{ranocha2018l2}
H.~Ranocha and P.~{\"O}ffner.
\newblock {$L_2$} stability of explicit {R}unge-{K}utta schemes.
\newblock {\em Journal of Scientific Computing}, 75(2):1040--1056, 2018.

\bibitem{ranocha2016summation}
H.~Ranocha, P.~{\"O}ffner, and T.~Sonar.
\newblock Summation-by-parts operators for correction procedure via
  reconstruction.
\newblock {\em Journal of Computational Physics}, 311:299--328, 2016.
  \href{http://arxiv.org/abs/1511.02052}{{\texttt{arXiv:1511.02052}}}.

\bibitem{ranocha2017extended}
H.~Ranocha, P.~{\"O}ffner, and T.~Sonar.
\newblock Extended skew-symmetric form for summation-by-parts operators and
  varying {J}acobians.
\newblock {\em Journal of Computational Physics}, 342:13--28, 04 2017.
  \href{http://arxiv.org/abs/1511.08408}{{\texttt{arXiv:1511.08408}}}.

\bibitem{svard2014review}
M.~Sv{\"a}rd and J.~Nordstr{\"o}m.
\newblock Review of summation-by-parts schemes for initial-boundary-value
  problems.
\newblock {\em Journal of Computational Physics}, 268:17--38, 2014.

\bibitem{szego1975orthogonal}
G.~Szeg{\"o}.
\newblock {\em Orthogonal Polynomials}, volume~23 of {\em Colloquium
  Publications}.
\newblock American Mathematical Society, Providence, Rhode Island, 1975.

\bibitem{tadmor1987numerical}
E.~Tadmor.
\newblock The numerical viscosity of entropy stable schemes for systems of
  conservation laws. {I}.
\newblock {\em Mathematics of Computation}, 49(179):91--103, 1987.

\bibitem{tadmor2002semidiscrete}
E.~Tadmor.
\newblock From semidiscrete to fully discrete: Stability of {R}unge-{K}utta
  schemes by the energy method. ll.
\newblock {\em Collected lectures on the preservation of stability under
  discretization}, 109:25, 2002.

\bibitem{tadmor2003entropy}
E.~Tadmor.
\newblock Entropy stability theory for difference approximations of nonlinear
  conservation laws and related time-dependent problems.
\newblock {\em Acta Numerica}, 12:451--512, 2003.

\bibitem{wang2009unifying}
Z.~Wang and H.~Gao.
\newblock A unifying lifting collocation penalty formulation including the
  discontinuous {G}alerkin, spectral volume/difference methods for conservation
  laws on mixed grids.
\newblock {\em Journal of Computational Physics}, 228(21):8161--8186, 2009.

\bibitem{wiener1938homogeneous}
N.~Wiener.
\newblock The homogeneous chaos.
\newblock {\em American Journal of Mathematics}, 60(4):897--936, 1938.

\bibitem{xiu2010numerical}
D.~Xiu.
\newblock {\em Numerical methods for stochastic computations: a spectral method
  approach}.
\newblock Princeton University Press, 2010.

\bibitem{xiu2002wiener}
D.~Xiu and G.~E. Karniadakis.
\newblock The {W}iener--{A}skey polynomial chaos for stochastic differential
  equations.
\newblock {\em SIAM Journal on Scientific Computing}, 24(2):619--644, 2002.

\bibitem{xiu2003modeling}
D.~Xiu and G.~E. Karniadakis.
\newblock Modeling uncertainty in flow simulations via generalized polynomial
  chaos.
\newblock {\em Journal of Computational Physics}, 187(1):137--167, 2003.

\bibitem{xiu2004supersensitivity}
D.~Xiu and G.~E. Karniadakis.
\newblock Supersensitivity due to uncertain boundary conditions.
\newblock {\em International Journal for Numerical Methods in Engineering},
  61(12):2114--2138, 2004.

\end{thebibliography}

\end{document}